\documentclass[a4paper,10pt]{article}

\usepackage[english]{babel}

\usepackage{amssymb,amsmath,amsthm}
\usepackage{amsfonts}

\numberwithin{equation}{section}

\theoremstyle{plain}
\newtheorem{theorem}{Theorem}[section]
\newtheorem{proposition}[theorem]{Proposition}
\newtheorem{lemma}[theorem]{Lemma}

\theoremstyle{definition}
\newtheorem{remarkbase}[theorem]{Remark}
\newenvironment{remark}{\pushQED{\qed}\remarkbase}{\popQED\endremarkbase}

\newcommand{\mA}{\mathcal{A}}
\newcommand{\mB}{\mathcal{B}}

\newcommand{\mD}{\mathcal{D}}

\newcommand{\mF}{\mathcal{F}}
\newcommand{\mG}{\mathcal{G}}
\newcommand{\mH}{\mathcal{H}}
\newcommand{\mI}{\mathcal{I}}
\newcommand{\mK}{\mathcal{K}}
\newcommand{\mL}{\mathcal{L}}
\newcommand{\mM}{\mathcal{M}}
\newcommand{\mN}{\mathcal{N}}
\newcommand{\mV}{\mathcal{V}}
\newcommand{\mW}{\mathcal{W}}

\newcommand{\mR}{\mathcal{R}}
\newcommand{\mU}{\mathcal{U}}

\renewcommand{\a}{\alpha}
\renewcommand{\b}{\beta}
\newcommand{\g}{\gamma}
\renewcommand{\d}{\delta}

\newcommand{\e}{\varepsilon}
\newcommand{\ph}{\varphi}

\newcommand{\lm}{\lambda}
\newcommand{\Om}{\Omega}
\newcommand{\om}{\omega}
\newcommand{\p}{\pi}
\newcommand{\s}{\sigma}
\renewcommand{\t}{\tau}

\renewcommand{\t}{\tau }

\renewcommand{\Re}{\mathrm{Re}\,}
\renewcommand{\Im}{\mathrm{Im}\,}

\newcommand{\gr}{\nabla}

\newcommand{\intp}{\int_{0}^{2\p}}

\newcommand{\la}{\langle}
\newcommand{\ra}{\rangle}

\newcommand{\R}{\mathbb R}
\newcommand{\C}{\mathbb C}
\newcommand{\Z}{\mathbb Z}
\newcommand{\N}{\mathbb N}
\newcommand{\T}{\mathbb T}
\newcommand{\sign}{{\rm sign}}
\newcommand{\sgn}{{\rm sign}}

\newcommand{\inv}{^{-1}}
\newcommand{\zero}{^{(0)}}
\newcommand{\uno}{^{(1)}}
\newcommand{\due}{^{(2)}}
\newcommand{\tre}{^{(3)}}

\renewcommand{\k}{^{(k)}}

\newcommand{\Ph}{\Phi}

\newcommand{\pepe}{\Pi_E^\perp} 
\newcommand{\PP}{\mathbb{P}} 
\newcommand{\pp}{\mathbb{P}}

\newcommand{\linf}{{L^{\infty}}}

\newcommand{\pa}{\partial}
\newcommand{\pat}{\partial_{t}} 
\newcommand{\pax}{\partial_{x}}

\newcommand{\patau}{\partial_{\t}}

\newcommand{\para}{\bar{a}}
\newcommand{\parb}{\bar{b}}

\newcommand{\lmcioeparb}{\parb}

\begin{document}

\title{\Large{\textbf{Periodic solutions of fully nonlinear autonomous equations of Benjamin-Ono type}}}

\date{}
\author{\large{Pietro Baldi}}

\maketitle

\begin{minipage}[c]{140mm}
	
\begin{small} 
\textbf{Abstract.} 
We prove the existence of time-periodic, small amplitude solutions of autonomous quasilinear or fully nonlinear completely resonant pseudo-PDEs of Benjamin-Ono type in Sobolev class. The result holds for frequencies in a Cantor set that has asymptotically full measure as the amplitude goes to zero.

At the first order of amplitude, the solutions are the superposition of an arbitrarily large number of waves that travel with different velocities (multimodal solutions).

The equation can be considered as a Hamiltonian, reversible system plus a non-Hamiltonian (but still reversible) perturbation that contains derivatives of the highest order. 

The main difficulties of the problem are: an infinite-dimensional bifurcation equation, and small divisors in the linearized operator, where also the highest order derivatives have nonconstant coefficients.

The main technical step of the proof is the reduction of the linearized operator to constant coefficients up to a regularizing rest, by means of changes of variables and conjugation with simple linear pseudo-differential operators, in the spirit of the method of Iooss, Plotnikov and Toland for standing water waves (ARMA 2005). 
Other ingredients are a suitable Nash-Moser iteration in Sobolev spaces, and Lyapunov-Schmidt decomposition. 

\end{small} 
\end{minipage}

\bigskip


\emph{2000MSC:} 
35B10, 37K55, 37K50.

\section{\large{The problem and main result}}

We consider autonomous 
equations of Benjamin-Ono type 
\begin{equation} \label{eq:main} 
u_t + \mH u_{xx} + \pa_x(u^3) + \mN_4(u) = 0
\end{equation}
with periodic boundary conditions $x \in \T := \R/2\p\Z$, 
where the unknown $u(t,x)$ is a real-valued function, $t \in \R$, 
$\mH$ is the periodic Hilbert transform, namely the Fourier multiplier
\[
\mH e^{ijx} = -i \, \sign(j) \, e^{ijx}, \quad j \in \Z,
\]
and $\mN_4$ is of type (I) or (II),
\begin{align} \label{g1 g2}
\text{(I)} & \quad \mN_4(u) = g_1(x,u,\mH u, u_x) 
+ \partial_x ( g_2(x,u,\mH u_x) ), 
\\
\label{g0}
\text{(II)} & \quad \mN_4(u) = g_0(x,u,\mH u, u_x, \mH u_{xx}).
\end{align}
\eqref{eq:main} is a \emph{quasilinear} problem in case (I) and a \emph{fully nonlinear} problem in case (II). 

We assume that the function $g_i(x,y)$ is defined for $y=(y_1,\ldots,y_n)$ in the ball $B_1 = \{ |y|<1 \}$ of $\R^n$, $n=2,3,4$, $g_i$ is $2\p$-periodic in the real variable $x$, and, together with its derivatives in $y$ up to order 4, it is of class $C^r$ in all its arguments $(x,y)$, with 
\begin{equation}  \label{Kgr}
\sum_{0 \leq |\a| \leq 4} \| \pa_y^\a g_i \|_{C^r(\T \times B_1)} \leq K_{g,r}, 
\end{equation}
for some constant $K_{g,r} > 0$. 
Moreover we assume that at $y=0$ 
\begin{equation}  \label{g order 4}
\pa_y^\a g_i(x,0) = 0 \quad \forall \a \in \N^n, \ |\a| \leq 3, 
\end{equation}
so that, regarding the amplitude, $\mN_4(\e u) = O(\e^4)$ as $\e \to 0$.

We assume that the nonlinearity $\mN(u) := \pa_x(u^3) + \mN_4(u)$ behaves like the linear part $\pa_t + \mH \pa_{xx}$ with respect to the parity of functions $u(t,x)$ in the time-space pair $(t,x)$. 
This means to assume the \emph{reversibility conditions} 
\begin{gather} \label{parity mN}
g_1(-x,y_1,-y_2,-y_3) = \, - g_1(x,y_1,y_2,y_3), 
\qquad 
g_2(-x,y_1,y_2) = \, g_2(x,y_1,y_2),  
\\
\label{parity mN II}
g_0(-x,y_1,-y_2,-y_3,-y_4) = \, - g_0(x,y_1,y_2,y_3,y_4),
\end{gather}
so that in both cases (I) and (II) $\mN(u)$ is odd for all even $u$, namely 
\begin{equation}  \label{mN(u) rev}
u(-t,-x) = u(t,x) \quad \Rightarrow \quad 
\mN(u)(-t,-x) = -\mN(u)(t,x). 
\end{equation}
Assumptions \eqref{g1 g2}, \eqref{g0}, \eqref{parity mN}, \eqref{parity mN II} are discussed in Section \ref{sec:comments}.

\begin{remark} Examples of such nonlinearities are:  
\[ 
\text{(I)} \quad 
\mN_4(u) = (\mH u_x)^3 \mH u_{xx} + a(x) u_x^4 + u u_x^3 + b(x) u_x^5,
\qquad  
\text{(II)} \quad 
\mN_4(u) = a(x) (\mH u_{xx})^4 + u_{x}^5,
\]
where $a(x)$ is odd and $b(x)$ is even. 
\end{remark}

We construct small amplitude time-periodic solutions $u(t,x)$ of period $T = 2\p/\om$, $\om>0$, 
where the period $T$ is also an unknown of the problem. 
Rescaling the time $t \to \om t$, this is equivalent to find $2\p$-periodic solutions of the equation 
\begin{equation} \label{eq:main 2} 
\om u_t + \mH u_{xx} + \pa_x(u^3) + \mN_4(u) = 0,
\end{equation}
with $u : \T^2 \to \R$, $\om > 0$.

Regarding the time-space pair $(t,x)$ as a point of the 2-dimensional torus $\T^2$, we consider the $L^2$-based Sobolev space of real-valued periodic functions 
\begin{equation} \label{norm two bars}
H^s = H^s(\T^2;\R) = \Big\{ u = \sum_{k \in \Z^2} u_k \, e_k \ : \ 
u_{-k} = \bar u_{k} \in \C, \ \ 
\| u \|_s^2 := \sum_{k \in \Z^2}  |u_{k}|^2 \langle k \rangle^{2s} 
< \infty \Big\} , 
\end{equation}
where $s\geq 0$, $\la k \ra := \max\{ 1, |k| \}$, and $e_k(t,x) := e^{i(k_1 t + k_2 x)}$.  

The main result of the paper is the following theorem.

\begin{theorem}  \label{thm:main} 
There exist universal constants $r_0, s_0, c_0 \in \N$ with the following properties. 

Assume hypotheses 
\eqref{g1 g2}, \ldots, \eqref{parity mN II}
on the nonlinearity $\mN$, with $r \geq r_0$. 
Let $m \geq 2$ and let $0 < k_1 < k_2 < \ldots < k_m$ be $m$ positive integers that satisfy 
\begin{equation} \label{bif e non-deg nel teorema} 
k_1 + \ldots + k_{m-1} > k_m (m-3/2), 
\qquad
k_1 + \ldots + k_m  \neq (m-1/2)j \quad \forall j \in \N.
\end{equation}
Then there exist $(i)$ a trigonometric polynomial 
\[
\bar v_1(t,x) := \sum_{j=1}^m a_j \cos(k_j \, x - k_j^2 \, t),
\]
even in the pair $(t,x)$, where $a_j \in \R$, 
\[
a_j^2  = \frac{4}{m-1/2} \, \Big( \sum_{i=1}^m k_i \Big) - 4 k_j,
\qquad j=1,\ldots,m;
\]

$(ii)$ constants $C, \e^*_0 >0$ that depend on $k_1,\ldots,k_m, K_{g,r_0}$;

$(iii)$ a measurable Cantor-like set $\mG \subset (0,\e^*_0)$ of asymptotically full Lebesgue measure, namely 
\[
\frac{| \mG \cap (0,\e_0)|}{\e_0}  \geq 1 - \e_0 C
\qquad \forall \e_0 \leq \e^*_0,
\]
such that for every $\e \in \mG$ problem \eqref{eq:main 2} with frequency 
\[
\om = 1 + 3 \e^2
\] 
has a solution $u_\e \in H^{s_0}(\T^2,\R)$ that satisfies 
\[
\| u_\e - \e \bar v_1 \|_{s_0} \leq \e^2 C, \quad 
u_\e(-t,-x) = u_\e(t,x), \quad 
\int_{\T^2} u_\e(t,x) \, dt \, dx = 0.
\]
Moreover $u_\e \in H^s(\T^2)$ for every $s$ in the interval $s_0 \leq s < (r+c_0)/2$.  

If $g_i$, $i=0,1,2$ in \eqref{g1 g2},\eqref{g0} is of class $C^\infty$, then also $u_\e \in C^\infty(\T^2)$. 
\end{theorem}

\begin{remark} \label{rem:smallest example} 
$(i)$ 
The smallest example of $k_1,\ldots,k_m$ satisfying \eqref{bif e non-deg nel teorema} is $m=2$, $k_1 = 2$, $k_2 = 3$. 
For every $m \geq 2$ there exist infinitely many choices of integers $k_1 < \ldots < k_m$ that satisfy \eqref{bif e non-deg nel teorema}. See also Remark \ref{rem:packet}.
%

$(ii)$ $s_0$, $r_0$ and $c_0$ can be explicitly calculated: $s_0 = 22$, $c_0 = 28$ (non-sharp calculation); for $r_0$ see \eqref{parametri 800} and the lines below it.
%
\end{remark}

\section{\large{Motivations, questions and comments}}
\label{sec:comments}

The original Benjamin-Ono equation 
\begin{equation} \label{original Benj-Ono}
u_t + \mH u_{xx} + uu_x = 0
\end{equation} 
models one-dimensional internal waves in deep water \cite{Benj}, and is a completely integrable \cite{Ablowitz-Fokas} Hamiltonian partial pseudo-differential equation, 
\[
\pa_t u = J \gr H(u), \quad 
J = -\pa_x, \quad 
H(u) = \int \Big( \frac{u \mH u_x}{2} + \frac{u^3}{6} \Big) dx.
\]

The local and global well-posedness in Sobolev class 
for \eqref{original Benj-Ono} and many generalizations of it 
(other powers $u^p u_x$, other linear terms $\pa_x|D_x|^\a u$, $1 < \a < 2$, etc)
have been studied by several authors in the last  years: see for example 
Molinet, Saut \& Tzvetkov \cite{Molinet-Saut-Tzvetkov 2002},
Colliander, Kenig \& Staffilani \cite{Coll-K-Staff 2003}, 
Tao \cite{Tao 2004},
Kenig \& Ionescu \cite{Kenig-Ionescu 2005},
Burq \& Planchon \cite{Burq-Planchon 2005}, 
Molinet \cite{Molinet 2007}, \cite{Molinet 2008},
and the references therein. 
On the contrary, to the best of our knowledge, there are few works about time-periodic or quasi-periodic solutions of Benjamin-Ono equations. 
One of them is \cite{ambrose-wilkening}, where 2-mode periodic solutions of \eqref{original Benj-Ono} are studied by numerical methods; another one is \cite{Liu-Yuan}, which deals with an old very interesting question. 

In \cite{Liu-Yuan} Liu and Yuan apply a Birkhoff normal form and KAM method
to show the existence of quasi-periodic solutions of a Benjamin-Ono equation that is a Hamiltonian analytic perturbation of \eqref{original Benj-Ono}, with Hamiltonian of the form
\[
H(u) + \e K(u), \quad H = \text{Benjamin-Ono}, \quad \gr K(u) = \text{bounded operator}.
\]
The resulting equation is of the type
\begin{equation}  \label{eq:Liu-Yuan}
\pa_t u = - \pa_x \{\mH u_x + \tfrac{1}{2} u^2 + \e \gr K(u)\} \ 
= A u + F(u),
\end{equation}
where the Hamiltonian vector field has a linear part $A$, which loses 
$d_A = 2$ derivatives, and a nonlinear part $F$, which loses   
$d_F = 1$  derivative and, for this reason, is an \emph{unbounded} operator. 

In general, as it was proved in the works of Lax, Klainerman and Majda on the formation of singularities (see for example \cite{klainerman-majda}), the presence of unbounded nonlinear operators can compromise the existence of invariant structure like periodic orbits and KAM tori. 
In fact, the wide existing literature on KAM and Nash-Moser theory mainly deals with problems where the perturbation is bounded (see Kuksin \cite{Kuksin oxford}, Craig \cite{Craig petit}, Berti \cite{Berti-book} for a survey. See also Moser \cite{Moser-Math-Ann-1967} where the KAM iteration is applied in problems where the Hamiltonian structure is replaced by reversibility).

For unbounded perturbations, quasi-periodic solutions have been constructed via KAM 
theory by 
Kuksin \cite{Kuksin oxford} and 
Kappeler \& P\"oschel \cite{Kappeler-Poeschel}
for KdV equations where $d_A = 3$ and the gap between the loss of derivatives of the linear and nonlinear part is $ \g := (d_A - d_F) = 2$, in analytic class; 
more recently, 
in \cite{Liu-Yuan} for NLS and \eqref{eq:Liu-Yuan} where $d_A = 2$ and $\g = 1$, in $C^\infty$ class;
by Zhang, Gao \& Yuan \cite{Zhang-Gao-Yuan} for reversible NLS equations with $d_A = 2$ and $\g = 1$; 
and by Berti, Biasco \& Procesi \cite{Berti-Biasco-Procesi-KAM-2011}, where wave equations with a derivative in the nonlinearity become a Hamiltonian system with $d_A = 1$ and $\g = 1$, in analytic class. 
See also Bambusi \& Graffi \cite{Bambusi-Graffi} for a related linear result that  corresponds to a gap $\g > 1$.  

Periodic solutions for unbounded perturbations have been obtained for wave equations by Craig \cite{Craig petit} for $\g > 1$; 
by Bourgain \cite{Bourgain Chicago 99} in the non-Hamiltonian case $u_{tt} - u_{xx} + u + u_t^2 = 0$; 
by the author in \cite{Baldi Kirchhoff} for the quasi-linear equation
$u_{tt} - \Delta u (1 + \int |\gr u|^2 dx) = \e f(t,x)$, where the integral plays a special role ($\int |\gr u|^2 dx$ depends only on time).
Also the pioneering result of Rabinowitz \cite{Rabinowitz-tesi-1969-II} for fully nonlinear wave equations of the form 
\[
u_{tt} - u_{xx} + \a u_t + \e F(x,t,u,u_x,u_t,u_{xx}, u_{xt}, u_{tt}) = 0
\]
certainly has to be mentioned here; however, the dissipative term $\a \neq 0$ destroys any Hamiltonian or reversible structure and completely avoids the resonance phenomenon of the small divisors.

The threshold $\g = 1$ in Hamiltonian problems with small divisors 
has been crossed in the works of Iooss, Plotnikov and Toland 
\cite{Plo-Tol}, \cite{Ioo-Plo-Tol}, \cite{Ioo-Plo}, \cite{Ioo-Plo-existence} about the completely resonant fully nonlinear ($\g = 0$) problem of periodic standing water waves on a deep 2D ocean with gravity. 
So far their very powerful technique, which is a combination of (1) changes of variables and conjugations with pseudo-differential operators to obtain a normal form, and (2) a differentiable Nash-Moser scheme,  
is essentially the only known method to overcome the small divisors problem in quasi-linear and fully nonlinear PDEs. 

Note that recently normal form methods for quasi-linear Hamiltonian PDEs have also been successfully applied to Cauchy problems, see Delort \cite{Delort-2009}.


Thus, some of the general, challenging and open questions that come from the aforementioned works are these:

\begin{itemize} 
\item Which gap $\g$ is the limit case for the existence of invariant tori for 
nonlinear Hamiltonian PDEs? How many derivatives can stay in the nonlinearity?

\item 
What is the role of the Hamiltonian structure? Can it be replaced by other structures? 
\end{itemize}

The motivations of the present paper are in these questions. 
Theorem \ref{thm:main} joins the above mentioned results in the aim of approaching an answer, at least in simple cases, and shows that 
\begin{itemize}
\item[$(i)$] 
if the dimension is the lowest for a PDE, $(t,x) \in \T^2$,  and 

\item[$(ii)$] 
the derivatives in the nonlinearity have a suitable structure (see \eqref{g1 g2},\eqref{g0},\eqref{parity mN},\eqref{parity mN II}), 
\end{itemize}
then problem \eqref{eq:main}, where $\g = 0$ (the nonlinearity $\mN(u)$ loses 2 derivatives like the linear part) admits solutions that bifurcate from the equilibrium $u=0$. 
The Hamiltonian structure here is replaced by reversibility:
\eqref{eq:main}, in general, is a \emph{non}-Hamiltonian perturbation of the cubic Benjamin-Ono Hamiltonian equation 
\[ 
\partial_{t} u + \mH \partial_{xx} u + \pa_x(u^3) = 0,
\]
but $\mN(u)$ satisfies the reversibility condition \eqref{mN(u) rev}. 

Let us explain the reversible structure in some detail. 
As a dynamical system, problem \eqref{eq:main} is 
\begin{equation} \label{dyn system}
\pa_t u(t) = V(u(t)),
\end{equation}
a first order ordinary differential equation in the infinite-dimensional phase space $L^2(\T;\R)$, where the vector field $V : H^2(\T;\R) \to L^2(\T;\R)$, $u \mapsto V(u)$ is 
\[
V(u)(x) = - \mH \partial_{xx} u(x) - \pa_x(u^3(x)) - \mN_4(u)(x).
\]
The phase space can be split into two subspaces $L^2_e \oplus L^2_o$ of even and odd functions of $x \in \T$ respectively,
\[
u = u^e + u^o, \quad 
u^e(-x) = u^e(x), \quad 
u^o(-x) = - u^o(x), \quad 
x \in \T, \quad 
u \in L^2(\T;\R).
\]
To decompose $u = u^e + u^o$ means to split the real and imaginary part of each Fourier coefficient of $u \in L^2(\T;\R)$, namely
\[
u(x) = \sum_{j \in \Z} \hat{u}_j \, e^{ijx}, \quad 
u^e(x) = \sum_{j \in \Z} (\mathrm{Re}\, \hat{u}_j) \, e^{ijx}, \quad 
u^o(x) = \sum_{j \in \Z} i (\mathrm{Im}\, \hat{u}_j) \, e^{ijx}.
\]
Consider the reflection 
\begin{equation} \label{S involution even odd}
R : \ u = u^e + u^o \ \mapsto \ Ru = u^e - u^o.
\end{equation}
$R$ is a $\R$-linear bijection of $L^2(\T;\R)$, and $R^2$ is the identity map. 
In terms of Fourier coefficients, 
\begin{equation} \label{S involution Fourier}
R : \ u(x) = \sum_{j \in \Z} \hat{u}_j \, e^{ijx} \ \mapsto \ Ru(x) = \sum_{j \in \Z} \overline{\hat{u}_j} \, e^{ijx},
\end{equation}
where $\overline{\hat{u}_j}$ is the complex conjugate of $\hat{u}_j$. 
Note that $Ru$ is real-valued for every real-valued $u$.
\eqref{dyn system} is a \emph{reversible} system in the sense that 
\begin{equation} \label{rev vera}
V \circ R = - R \circ V.
\end{equation}
It is immediate to check \eqref{rev vera} for the linear part $\mH \pa_{xx}$ of $V$ using \eqref{S involution Fourier}, and for the cubic part $\pa_x(u^3)$ using \eqref{S involution even odd}.
To prove \eqref{rev vera} for $\mN_4(u)$, using \eqref{parity mN}, \eqref{parity mN II} and \eqref{S involution even odd} one has
\[
\a(-x) = - \b(x), \quad 
\a(x) := \mN_4(Ru)(x), \quad 
\b(x) := \mN_4(u)(x).
\]
Splitting $\a = \a^e + \a^o$, $\b = \b^e + \b^o$ and projecting the equality $\a(-x) = - \b(x)$ onto $L^2_e$ and $L^2_o$ give $\a^e = - \b^e$ and $\a^o = \b^o$, namely $R\b = - \a$, which is \eqref{rev vera} for $\mN_4$.

\eqref{rev vera} implies that $V(u) \in L^2_o$ for all $u \in L^2_e \cap H^2$. 
For, $L^2_e$ is the set of fixed points $u = Ru$, therefore $V(u) = -R V(u)$, whence $(V(u))^e = 0$.

By \eqref{rev vera}, if $u(t)$ solves \eqref{dyn system}, then also $Su(t) := R(u(-t))$ is a solution of \eqref{dyn system}. 
Thus we look for solutions of \eqref{dyn system} in the subspace $X$ of the fixed points of $S$. It is easy to see, 
using \eqref{S involution even odd}, \eqref{S involution Fourier},  
that $X$ is the space of functions $u(t,x)$ that are even in the time-space pair $(t,x)$, namely $u(-t,-x) = u(t,x)$. 

\medskip

\medskip

To prove Theorem \ref{thm:main} we apply (and slightly modify, under certain technical aspects; see below) the method of Iooss, Plotnikov and Toland. 
Like in \cite{Ioo-Plo-Tol}, 
the main difficulties here are: 
$(i)$ in the bifurcation equation, which is infinite-dimensional   
(for this reason \eqref{eq:main} is said to be a \emph{completely resonant} problem);  
and, especially, $(ii)$ in the inversion of the linearized operator, which has non-constant coefficients also in the highest order derivatives and, therefore, contains small divisors that are not explicitly evident.

The main tool in the inversion proof is the reduction of the linearized operator $\mL$ to constant coefficients up to a regularizing rest, by means of changes of variables first (to obtain proportional coefficients in the highest order terms), then by the conjugation with simple linear pseudo-differential operators that imitate the structure of $\mL$ 
(they are the composition of multiplication operators with the Hilbert transform $\mH$), to obtain constant coefficients also in terms of lower order, and to lower the degree of the highest non-constant term.

Since we look for periodic solutions, after a finite number of steps this reducibility scheme implies the invertibility of $\mL$, by standard Neumann series. 

\medskip

Other, and minor, technical points are the following.
Like in \cite{Ioo-Plo-Tol}, the Lyapunov-Schmidt decomposition is not used directly on the nonlinear equation, as it would be made in classical applications 
(see \cite{Berti-book} for the Lyapunov-Schmidt decomposition in completely resonant problems).
Instead, it is used a first time at the beginning of the proof, in a formal power series expansion of the nonlinear problem, to look for a suitable starting point of the Nash-Moser iteration. In other words, this means to find a non-degenerate solution of the ``unperturbed bifurcation equation''. 
In Theorem \ref{thm:main} the existence and the non-degeneracy conditions are the first and the second inequality in \eqref{bif e non-deg nel teorema} respectively. 
Then the Lyapunov-Schmidt decomposition is used a second time in the inversion proof for the linearized operator, in each step of the Nash-Moser scheme.

This method seems to be more complicated than the usual Lyapunov-Schmidt decomposition on the nonlinear problem, at least at a first glance. 
However, it simplifies the analysis when working with changes of variables (namely compositions with diffeomorphisms of the torus $\T^2$). 
In fact, changes of variables do not behave very well with respect to the orthogonal projections onto subspaces of $L^2$, because they are not ``close to the identity'' in the same way as multiplications operators are
(in the language of harmonic analysis, changes of variables are Fourier integral operators, and not pseudo-differential operators. See also Remark \ref{rem:loss of 1 for Psi-I}). 
For this reason, it is simpler to work in the  whole function space $H^s(\T^2)$ instead of distinguishing bifurcation and orthogonal subspaces, at least for the first step of reducibility. 

Nonetheless, in our setting \eqref{bif setting} we keep track of the natural ``different amount of smallness'' between the bifurcation and the orthogonal components of the problem. 
Thanks to this small change with respect to \cite{Ioo-Plo-Tol}, we avoid factors $\e^{-1}$ in the Nash-Moser scheme and simplify the measure estimate for the small divisors. 

Regarding the Nash-Moser scheme, the recent and powerful abstract Nash-Moser theorem for PDEs that is contained in \cite{Berti-Bolle-Procesi-AIHP-2010} does not apply directly here, as it designed to be used with Galerkin approximations, while in our Nash-Moser scheme, after the reduction to constant coefficients, it is natural to insert the smoothing operators in a different position: see \eqref{sequence un}. Even if our iteration scheme is very close to the usual one, this small difference brings our problem out of the field of applicability of the theorem in \cite{Berti-Bolle-Procesi-AIHP-2010}.

Going back to the ``unperturbed bifurcation equation'', we point out that the restriction of the functional setting to the subspace $X$ of even functions (a restriction that can be made because of the reversible structure) eliminates a degeneration and makes it possible to prove the non-degeneracy of the solution. 
Moreover, the solutions we find in Theorem \ref{thm:main} 
are genuinely \emph{multimodal}: 
for $m=1$ the second inequality in \eqref{bif e non-deg nel teorema} is never satisfied, whereas for every $m \geq 2$ there exist suitable integers $k_1,\ldots,k_m$ that satisfy \eqref{bif e non-deg nel teorema} and produce a non-degenerate solution.
This is a nonlinear effect: the solutions of Theorem \ref{thm:main} exist as a consequence of the nonlinear interaction of different modes.

Regarding the special structure \eqref{g1 g2},\eqref{g0}, the restriction of assuming (I) or (II), instead of considering the more general case
\begin{equation} \label{more general mN4}  
\mN_4(u) = g(x,u,\mH u, u_x, \mH u_x, u_{xx}, \mH u_{xx}),
\end{equation}
is due to a technical reason: when $\mN_4(u)$ is of the type (I) or (II), in the process of reducing the linearized operator $\mL$ to constant coefficients we use simple transformations, namely changes of variables, multiplications, the Hilbert transform $\mH$ and negative powers of $\pa_x$ (which are Fourier multipliers).
On the contrary, in the general case \eqref{more general mN4} these special transformations are not sufficient to conjugate $\mL$ to a normal form, and one needs more general transformations: changes of variables should be replaced by general Fourier integral operators. 
In the intermediate case in which $\mN_4$ in \eqref{more general mN4} does not depend on $u_{xx}$ (but it does on $\mH u_x$), an additional term of the type $b(t) \pa_x \mH$ appears in the transformed linearized operators after the changes of variables. 
This term could be removed by a simple Fourier integral operator: see Remark \ref{rem:FIO}. 

Regarding the choice of the leading term $\pa_x(u^3)$ in \eqref{eq:main} (which is the first natural case to study after the integrable one $\pa_x(u^2)$), we remark that the cubic power has no special reversibility property: $\pa_x(u^p)$ satisfy the reversibility condition \eqref{rev vera} for every (both even and odd) power $p \in \N$. The proof of this fact is the same as above: if $f(u) = \pa_x(u^p)$, using \eqref{S involution even odd} one proves that $\{f(Ru)\}(-x) = - \{ f(u)\}(x)$, then $f \circ R = - R \circ f$.

Finally, the coefficient 3 in the frequency-amplitude relation $\om = 1 + 3 \e^2$ could be replaced by any other positive number: 
3 is simply the most convenient choice to do when working with the cubic nonlinearity $\pa_x(u^3)$. 
On the contrary, what is determined by the nonlinearity in an essential way is \emph{the sign} of that coefficient: for the equation 
\[
u_t + \mH u_{xx} - \pa_x(u^3) + \mN_4(u) = 0,
\] 
in which the cubic nonlinearity has opposite sign, Theorem  \ref{thm:main} holds with $\om = 1 - 3 \e^2$ (the only changes to do are in the bifurcation analysis of Section \ref{sec:bifurcation}).


\bigskip

The paper is organized as follows.
In Section \ref{sec:functional setting} the setting for the problem is introduced. 
In Section \ref{sec:Lyapunov-Schmidt} the formal Lyapunov-Schmidt reduction is performed up to order $O(\e^4)$. 
In Section \ref{sec:bifurcation} non-degenerate solutions $\bar v_1$ of the ``unperturbed bifurcation equation'' are constructed. Here the non-homogeneous dispersion relation of the unperturbed Benjamin-Ono linear part
\[
l + j|j| = 0,
\]
where $l$ is the Fourier index for the time and $j$ the one for the space, is used in a crucial way. The basic properties of this relation are proved in Appendix \ref{Appendix A. Kernel properties}.
In Sections \ref{sec:The linearized equation} and \ref{sec:reduction} the linearized operator is reduced to constant coefficients. Most of the proofs of the related estimates are in Appendix \ref{sec:appendix proofs} and use classical results of Sobolev spaces (tame estimates for changes of variables, compositions and commutators with the Hilbert transform) that are listed in Appendix \ref{sec:Appendix B. Tame estimates}.
In Section \ref{sec:inversion} the transformed linearized operator is inverted. 
In Section \ref{sec:iteration} the Nash-Moser induction is performed, and the measure of the Cantor set of parameters is estimated.

\medskip

\emph{Acknowledgements.} I express my gratitude to Massimiliano Berti for many fruitful discussions and suggestions, Pavel Plotnikov, G\'erard Iooss and Thomas Alazard for useful conversations, and John Toland for 
introducing me to the problem. 

This work is partially supported by the Italian PRIN2009 grant \emph{Critical Point Theory and Perturbative Methods for Nonlinear Differential Equations}, and by the European Research Council, FP7, project \emph{New connections between Dynamical Systems and Hamiltonian PDEs with Small Divisors Phenomena}.

\section{\large{Functional setting}} \label{sec:functional setting}
Let 
\[
\mF(u,\om) := \om u_t + \mH u_{xx} + \mN(u), \quad 
\mN(u) := \pa_x(u^3) + \mN_4(u).
\]
Let $Z := L^2(\T^2,\R)$. Decompose
\begin{equation*} 
\Z^2 = \Z^2_C + \Z^2_T + \Z^2_E, 
\quad 
\Z^2_C = \{ (0,0) \}, 
\quad 
\Z^2_T = \{ (l,0) : l \neq 0 \},
\quad
\Z^2_E = \{ (l,j) : j \neq 0 , \ l \in \Z\},
\end{equation*}
let 
\[
Z_C = \R, \quad 
Z_T = \Big\{ u \in L^2(\T) : \intp u(t) \, dt = 0 \Big\} ,
\quad 
Z_E = \Big\{ u \in Z : \intp u(t,x) \, dx = 0 \Big\} ,
\]
so that $Z = Z_C \oplus Z_T \oplus Z_E$, namely 
every $u(t,x) \in Z$ splits into three components
\begin{align*} 
u(t,x) & 
= \Big( \sum_{\Z^2_C} + \sum_{\Z^2_T} + \sum_{\Z^2_E} \Big) 
 \hat u_{l,j}\,e^{i(lt+jx)} 
= \hat u_{0,0} + \sum_{l \neq 0} \hat u_{l,0} \, e^{ilt} 
+ \sum_{j \neq 0} u_j(t) \, e^{ijx} ,
\end{align*}
and denote $\Pi_C, \Pi_T, \Pi_E$ the projections onto $Z_C, Z_T, Z_E$.
Let $Z_0$ be the space of zero-mean functions, and $\pp$ the projection onto $Z_0$,
\begin{equation} \label{Z0 pp}
Z_0 := Z_T \oplus Z_E, \quad 
\PP := I - \Pi_C = \Pi_T + \Pi_E.
\end{equation}
We define $\partial_x\inv$ as the Fourier multiplier 
\begin{equation*} 
\partial_x\inv e^{ijx} = \frac{1}{ij}\, e^{ijx} \quad \forall j \neq 0, 
\quad \partial_x\inv 1 = 0 ,
\end{equation*}
and similarly $\partial_t\inv$. 
Note that $\partial_x\inv \partial_x = \Pi_E$, $\mH \mH = - \Pi_E$.

To eliminate a degeneration that appears in the bifurcation equation, as it was mentioned above where the reversible structure was discussed, 
we consider the subspaces of even/odd functions with respect to the time-space vector $(t,x)$: 
\[
X := \big\{ u \in Z : \ u(-t,-x) = u(t,x) \big\} ,
\quad  
Y := \big\{ u \in Z : \ u(-t,-x) = - u(t,x) \big\} .
\]
In terms of Fourier coefficients, every $u \in Z$ is 
$u = \sum_{k \in \Z^2} u_k e_k$ with $u_{-k} = \bar u_k$ (because $u$ is real-valued), 
namely $u_k = a_k + i b_k$, with $a_k, b_k \in \R$ and $a_{-k} = a_k$, $b_{-k} = - b_k$, therefore  
\[
X = \Big\{ u = \sum_{k \in \Z^2} a_k e_k : \ a_k \in \R, \ a_{-k} = a_k \Big\} , 
\qquad 
Y = \Big\{ u = \sum_{k \in \Z^2} i b_k e_k : \ b_k \in \R, \ b_{-k} = - b_k \Big\} ,
\]
and $L^2(\T^2,\R) = Z = X \oplus Y$.
The usual rules for even/odd functions hold: $uv \in X$ if both $u,v \in X$ or both $u,v \in Y$, and $uv \in Y$ if $u \in X$, $v \in Y$.
Moreover $\mH, \partial_x, \partial_t$ are all operators that change the  parity, namely they map $Y$ into $X$ and viceversa, because they are diagonal operators with respect to the basis $\{e_k\}$ with purely imaginary eigenvalues. 
Assumption \eqref{parity mN} implies that the nonlinearity $\mN$ maps $X \cap H^2$ into $Y$, like the linear part $\om \pa_t + \pa_{xx}\mH$ does, therefore 
$\mF(u,\om) \in Y$ for all $u \in X \cap H^2$.

Also, we denote 
\[
X_0 := X \cap Z_0,
\]
while $Y \cap Z_0 = Y$. 
We set problem \eqref{eq:main 2} in the space $X_0$ of even functions with zero mean, namely we look for solutions of the equation 
\begin{equation} \label{mF=0 even odd}
\mF(u,\om) = 0, \quad u \in X_0.
\end{equation}

\medskip

\emph{Notation.} To distinguish $L^2$- and $L^\infty$-based Sobolev spaces, 
in the whole paper the following notation is used: 
two bars for $L^2$-based Sobolev norms $\| u \|_s$ \eqref{norm two bars}, 
and one bar for $L^\infty$-based Sobolev norms
\begin{equation*} 
|u|_s 
= \| u \|_{W^{s,\infty}} 
= \sum_{0 \leq |\a| \leq s} \, \sup_{(t,x)} |\pa^\a_{(t,x)} u(t,x)|, \quad s \in \N.
\end{equation*}

\section{\large{Linearization at zero and formal Lyapunov-Schmidt reduction}}
\label{sec:Lyapunov-Schmidt}

Let 
\[
L := \partial_t + \partial_{xx}\mH, 
\quad 
L [e^{i(lt+jx)}] = i(l+j|j|) \, e^{i(lt+jx)}.
\]
Split $\Z^2 = \mV \cup \mW$,  
\[ 
\mV := \{ (l,j) \in \Z^2 : \ l+j|j| = 0 \} 
= \{ (-j|j|,j) : \ j \in \Z \}, 
\qquad 
\mW := \Z^2 \setminus \mV 
\] 
and $Z = V \oplus W$,  
\[
V := \Big\{ u = \sum_{k \in \mV} u_k e_k \in Z \Big\},
\quad 
W := \Big\{ u = \sum_{k \in \mW} u_k e_k \in Z \Big\}.
\]
$V$ is the kernel of $L$ and $W$ is its range.
Also, let $V_0 := V \cap Z_0$, so that $Z_0 = V_0 \oplus W$. 

We write a finite number of terms of a formal power series expansion to obtain a good starting point for our Nash-Moser scheme. 
Let 
\[
\om = 1 + \sum_{k \geq 1} \om_k \e^k, \qquad 
u = \sum_{k \geq 1} u_k\, \e^k\, \in Z_0, \quad 
u_k = v_k + w_k, \quad 
v_k \in V_0, \ \ w_k \in W.
\]
Then 
\begin{align*} 
\mF(u,\om) & = L u + (\om-1)\pat u + \partial_x(u^3) + \mN_4(u) \\
& = \e \, L u_1 
+ \e^2 \big\{ L u_2 + \om_1 \pat u_1 \big\} 
+ \e^3 \big\{ L u_3 + \om_1 \pat u_2 + \om_2 \pat u_1 + \pax(u_1^3) \big\} 
\\
& \quad 
+ \e^4 \big\{ L u_4 + \om_1 \pat u_3 + \om_2 \pat u_2 + \om_3 \pat u_1 + \pax(3 u_1^2 u_2) + \e^{-4} \mN_4(\e u_1)\big\} 
+ O(\e^5) 
\\ & = \sum_{k \geq 1} \e^k \mF_k.
\end{align*}
In general, $\mN_4(\e u_1)$ also contains terms of higher order than $\e^4$; in any case, $\mN_4(u) - \mN_4(\e u_1) = O(\e^5)$.

At order $\e$, 
$\mF_1 = L u_1 = 0$ if $w_1 = 0$ and $u_1 = v_1 \in V_0$. 
Then $\mF_2$ becomes
\[
\mF_2 = L u_2 + \om_{1} \partial_t u_{1} 
= L w_2 + \om_{1} \partial_t v_{1}.
\]
$L w_2 \in W$ and $\om_{1} \partial_t v_{1} \in V_0$. 
Since we look for $v_1 \neq 0$, we have $\mF_2 = 0$ if 
$w_2 = 0$, $\om_1 = 0$, $u_2 = v_2 \in V_0$.

At order $\e^3$ the nonlinearity begins to give a contribution: 
$\mF_3 = L w_3 + \om_2 \partial_t v_1 + \partial_x(v_1^3)$.
The ``unperturbed bifurcation equation'' is the equation 
$\Pi_V \mF_3 = 0$ in the unknown $v_1$, namely 
\begin{equation} \label{bif 0}
\om_2 \partial_t v_1 + \Pi_V \partial_x(v_1^3) = 0.
\end{equation}
In the next section (see Proposition \ref{prop:bif}) we construct nontrivial, nondegenerate solutions $\bar v_1$ of \eqref{bif 0} with 
$\om_2 = 3$.
A solution $v_1$ of \eqref{bif 0} for any other value $\om_2 > 0$ can be obtained by homogeneity by taking $v_1 = \lm \bar v_1$, $\lm = (\om_2/3)^{1/2}$.  Hence there is no loss of generality in fixing $\om_2 = 3$. 
At order $\e^4$, 
\[
\mF_4 = L u_4 + 3 \pat v_2 + \om_3 \pat v_1 + \pax(3 v_1^2 v_2) + \e^{-4} \mN_4(\e v_1).
\]
We fix $\om_3 = 0$. 
The ``linearized unperturbed bifurcation equation'' is the equation 
$\Pi_V \mF_4 = 0$ in the unknown $v_2$, namely 
\begin{equation} \label{linearized unperturbed bif eq}
3 \pat v_2 + \Pi_V \pax(3 v_1^2 v_2) = - \e^{-4} \Pi_V \mN_4(\e v_1),
\end{equation}
which has a unique solution $\bar v_2(\e)$ because $\bar v_1$ is a nondegenerate solutions of \eqref{bif 0}.
Thus, at $u = \e \bar v_1 + \e^2 \bar v_2(\e)$ and $\om = 1+3\e^2$,
\begin{align} \label{mF bar v1 bar v2}
\mF(\e \bar v_1 + \e^2 \bar v_2, \, 1+3\e^2) 
& = \e^3 \Pi_W \pa_x(\bar v_1^3) 
+ \e^4 \Pi_W \pa_x(3 \bar v_1^2 \bar v_2) 
+ \mN_4(\e \bar v_1 + \e^2 \bar v_2) 
- \mN_4(\e \bar v_1)
\notag 
\\ & \quad 
+ \Pi_W \mN_4(\e \bar v_1) 
+ \e^5 \pa_x(3 \bar v_1 \bar v_2^2)   
+ \e^6 \pa_x(\bar v_2^3).
\end{align}
With these power of $\e$, 
the sufficient accuracy is achieved to start the quadratic Nash-Moser scheme (see section \ref{sec:iteration}). 
Hence, for $\e > 0$, let
\begin{align} 
F(u, \e) 
& := ( \e^{-4} \Pi_V + \e^{-2} \Pi_W ) \mF(\e \bar v_1 + \e^2 u,\om) 
\label{bif setting}
\\ 
& = \e^{-2} P_\e\inv  \mF ( \e \bar v_1 + \e^2 u, \, 1+3\e^2) 
\notag 
\\
& = \Pi_V \{ 3\pa_t u + \pa_x (3 \bar v_1^2 u + \e 3 \bar v_1 u^2 + \e^2 u^3) 
+ \e^{-4} \mN_4(\e \bar v_1 + \e^2 u) \}  
\label{formula F}
\\
& \quad + \Pi_W \{ L u + \e^2 3\pa_t u + \e \pa_x [(v_1 + \e u)^3] 
+ \e^{-2} \mN_4(\e \bar v_1 + \e^2 u) \}, 
\notag
\end{align}
\begin{equation*} 
\om := 1 + 3 \e^2, \quad 
P_\e := \e^{2} \,\Pi_V + \Pi_W, \quad 
P_\e\inv = \e^{-2} \,\Pi_V + \Pi_W. \quad
\end{equation*} 

By \eqref{mF bar v1 bar v2}, $F(\bar v_2,\e) = O(\e)$
(see Lemma \ref{lemma:tame per F} for precise estimates). 
For $\e > 0$, problem \eqref{mF=0 even odd} becomes 
\begin{equation} \label{equiv 2}
F(u, \e) = 0, \quad u \in X_0.
\end{equation}
Like $\mF$ does, $F$ also maps $X_0$ into $Y$.

\section{\large{Bifurcation}} \label{sec:bifurcation}

In this section we construct a solution $v \in V_0$ of \eqref{bif 0} and prove its non-degeneracy. 
Recall that in $\mV$ it is $l+j|j|=0$. Let 
\begin{equation} \label{qj}
q_j(t,x) := e^{i(-j|j|t+jx)} , \quad j \in \Z \,.
\end{equation}
Note that $q_{j_1} q_{j_2} = 1 = q_0$ if $j_1 + j_2 = 0$.

\begin{lemma} \label{lemma:prodotti in V}
1) \emph{(Product of two terms)}. 
Let $j_1, j_2 \in \Z$ be both nonzero integers. 
Then 
$\Pi_V (q_{j_1} q_{j_2}) = 0$ 
except the case when $j_1 + j_2 = 0$. 

2) \emph{(Product of three terms)}. 
Let $j_1, j_2, j_3 \in \Z$ be all nonzero integers. Then 
$\Pi_V(q_{j_1} q_{j_2} q_{j_3} ) = 0$
except the case when $j_1 + j_2 = 0$ or $j_1 + j_3 = 0$ or $j_2 + j_3 = 0$. 
\end{lemma}

\begin{proof} 
See Appendix \ref{Appendix A. Kernel properties}. 
\end{proof}

Consider $m$ positive distinct integers 
$0 < k_1 < k_2 < \ldots < k_m$, $m \geq 1$,
and let 
\[
\mK := \{ k_1, k_2, \ldots, k_m, -k_1, -k_2, \ldots, -k_m \} \,.
\]
Consider three elements $v, v', v'' \in V_0 \cap X$ with only Fourier modes in $\mK$, 
\[ 
v = \sum_{j \in \mK} a_j q_j , \quad 
v' = \sum_{j \in \mK} b_j q_j , \quad 
v'' = \sum_{j \in \mK} c_j q_j , 
\]
with $a_{-j} = a_j \in \R$, and similar for $b_j, c_j$.
Then
\[
v v' v'' 
= \sum_{j_1, j_2, j_3 \in \mK} a_{j_1} b_{j_2} c_{j_3}
\,  q_{j_1} q_{j_2} q_{j_3} ,
\quad 
\Pi_V (v v' v'')
= \sum_{j_1, j_2, j_3 \in \mK} a_{j_1} b_{j_2} c_{j_3}
\, \Pi_V( q_{j_1} q_{j_2} q_{j_3} ) \,.
\]
Develop the sum with respect to $j_1$. 
Let $k \in \mK$.
For $j_1 = k$, $\Pi_V(q_{j_1} q_{j_2} q_{j_3})$ is nonzero only if: 
\begin{equation} \label{four situations}
\begin{pmatrix} 
j_1 = k \\ 
j_2 = k \\
j_3 = -k
\end{pmatrix} 
\quad \text{or} \quad
\begin{pmatrix} 
j_1 = k \\ 
j_2 = - k \\
j_3 \in \mK 
\end{pmatrix} 
\quad \text{or} \quad
\begin{pmatrix} 
j_1 = k \\ 
j_2 \neq \pm k \\
j_3 = - k
\end{pmatrix} 
\quad \text{or} \quad
\begin{pmatrix} 
j_1 = k \\ 
j_2 \neq \pm k \\
j_3 = - j_2
\end{pmatrix}.
\end{equation}
Hence in the sum only these four cases give a nonzero contribution:
\begin{equation}  \label{4 somme}
\Pi_V (v v' v'')
= \sum_{k \in \mK} a_{k} b_{k} c_{k} \, q_k 
+ \sum_{k,j \in \mK} a_{k} b_{k} c_{j} \, q_{j} 
+ \sum_{k \in \mK, j \neq \pm k} a_{k} b_{j} c_{k} \,q_{j} 
+ \sum_{k \in \mK, j \neq \pm k} a_{k} b_{j} c_{j} \,q_{k} \,.
\end{equation}
Since $\sum_{k \in \mK, j \neq \pm k} = \sum_{k,j \in \mK} - \sum_{k \in \mK, j = k} - \sum_{k \in \mK, j = -k}$,  
the third sum in \eqref{4 somme} is 
\begin{align*} 
\sum_{k \in \mK, j \neq \pm k} a_{k} b_{j} c_{k} \, q_{j} 
& = \sum_{k,j \in \mK} a_{k} b_{j} c_{k} \, q_{j} 
- \sum_{k \in \mK} a_{k} b_{k} c_{k} \, q_{k} 
- \sum_{k \in \mK} a_{k} b_{k} c_{k} \, q_{-k} 
\\ & 
= \sum_{k,j \in \mK} a_{k} b_{j} c_{k} \, q_{j} 
- \sum_{k \in \mK} a_{k} b_{k} c_{k} \, q_{k} 
- \sum_{k \in \mK} a_{k} b_{k} c_{k} \, q_{k} 
\end{align*}
(in the last equality we have made the change of summation variable $k= -k'$). Analogously, the fourth sum in \eqref{4 somme} is 
\begin{align*} 
\sum_{k \in \mK, j \neq \pm k} a_{k} b_{j} c_{j} \, q_{k} 
& = \sum_{k,j \in \mK} a_{k} b_{j} c_{j} \, q_{k} 
- \sum_{k \in \mK} a_{k} b_{k} c_{k}\, q_{k} 
- \sum_{k \in \mK} a_{k} b_{k} c_{k} \, q_{k} \,.
\end{align*}
Thus 
\begin{equation} \label{formula 3-lineare}
\Pi_V (v v' v'') 
= \sum_{k \in \mK} 
\Big\{ -3 a_k b_k c_k 
+ a_{k} \Big( \sum_{j \in \mK}  b_{j} c_{j} \Big) 
+ b_{k} \Big( \sum_{j \in \mK}  a_{j} c_{j} \Big) 
+ c_{k} \Big( \sum_{j \in \mK}  a_{j} b_{j} \Big) \Big\} 
\, q_{k} \,.
\end{equation}
The formula for $\Pi_V [\pa_x(v v' v'')] = \pa_x \Pi_V(vv'v'')$ simply has 
$i k\, q_k$ instead of $q_k$.
For $v=v'=v''$, \eqref{formula 3-lineare} gives
\begin{align*} 
\Pi_V (v^3) & 
= 3 \sum_{k \in \mK} \Big( - a_k^2 + \sum_{j \in \mK}  a_{j}^2 \Big)
a_{k} \, q_{k} \,.
\end{align*}
Then 
\begin{align*} 
3 \partial_t v + \Pi_V [ \partial_x( v^3)]
& = 3 \sum_{k \in \mK} 
\Big( - |k| - a_k^2 + \sum_{j \in \mK} a_{j}^2 \Big) \, a_k \, i k \, q_{k} \,. 
\end{align*} 
This is zero if 
\begin{equation} \label{pass 1}
\Big( \sum_{j \in \mK} a_{j}^2 \Big) - a_k^2 = |k| \quad \forall k \in \mK \,.
\end{equation}
Since $\sum_{j \in \mK} a_{j}^2 = 2(a_{k_1}^2 + \ldots + a_{k_m}^2)$,
\eqref{pass 1} is equivalent to 
\begin{equation} \label{pass 2}
\begin{cases} 
a_{k_1}^2 + 2 a_{k_2}^2 + 2 a_{k_3}^2 + \ldots + 2 a_{k_m}^2 
& = k_1 \\ 
2 a_{k_1}^2 + a_{k_2}^2 + 2 a_{k_3}^2 + \ldots + 2 a_{k_m}^2 
& = k_2 \\ 
\qquad \ldots \qquad \ldots \qquad \ldots  
& \quad \ldots \\ 
2 a_{k_1}^2 + 2 a_{k_2}^2 + 2 a_{k_3}^2 + \ldots + a_{k_m}^2 
& = k_m , 
\end{cases} 
\end{equation}
which is a system of $m$ equations in the $m$ unknowns $a_{k_1}^2$, \ldots, $a_{k_m}^2$. 
Let $M$ the $m \times m$ matrix that has $1$ on the principal diagonal and $2$ everywhere else. $M$ is invertible, and its inverse $M\inv$ is the $m \times m$ matrix that has $\a$ on the principal diagonal and $\b$ everywhere else, with
\[
\a = - \frac{m-3/2}{m-1/2}\,, \quad \b = \frac{1}{m-1/2}\,.
\]
Hence \eqref{pass 2} is equivalent to 
\begin{equation} \label{pass 3}
a_{k_1}^2 = \rho_1, \quad \ 
a_{k_2}^2 = \rho_2, \quad \ 
\ldots \quad \ \  
a_{k_m}^2 = \rho_m,
\end{equation}
where $(\rho_1, \ldots, \rho_m) := M\inv (k_1, \ldots, k_m)$, namely
\begin{equation} \label{rho} 
\rho_i := \a k_i + \b \sum_{j \neq i} k_j 
\ = \frac{1}{m-1/2} \, \Big( \sum_{j=1}^m k_j \Big) - k_i , 
\qquad i=1,\ldots,m\,.
\end{equation}
\eqref{pass 3} has solutions with all $a_j \neq 0$ 
if all $\rho_j$ are positive. 
Note that $\rho_j > \rho_{j+1}$, because $\b - \a = 1$ and 
\[
\rho_j - \rho_{j+1} = \a k_j + \b k_{j+1} - \b k_j - \a k_{j+1} 
= k_{j+1} - k_j > 0 \,.
\] 
Hence all $\rho_j > 0$ if $\rho_m > 0$, namely if 
\begin{equation} \label{bif cond} 
k_1 + \ldots + k_{m-1} > k_m (m-3/2) \,.
\end{equation}
When $a_j$ satisfy \eqref{pass 3}, 
\begin{equation} \label{sommina}
\sum_{j \in \mK} a_j^2 
= 2(a_{k_1}^2 + \ldots + a_{k_m}^2) 
= \frac{1}{m-1/2} \, \sum_{i=1}^m k_i \,.
\end{equation}

\begin{remark}  \label{rem:packet}
$k_1, \ldots, k_m$ satisfy \eqref{bif cond} if they are sufficiently close, as if they form a ``packet'' of integers. 
Note also that if the smallest and the biggest integers satisfy the stronger condition
\begin{equation} \label{diadic}
\frac{k_m}{k_1} \, < \frac{m-1}{m-3/2} \,,
\end{equation}
then $k_1, k_2, \ldots, k_m$ satisfy \eqref{bif cond} for every choice of the intermediate integers $k_2, \ldots, k_{m-1}$, because 
\[
k_1 + k_2 + \ldots + k_{m-1} 
> (m-1) k_1 >  (m-3/2) k_m .
\]
\eqref{diadic} is meaningful because $(m-1)/(m-3/2) > 1$.
\end{remark}

\medskip

Now we prove that for every $f \in V_0 \cap Y$ there is a unique $h \in V_0 \cap X$ such that 
\begin{equation} \label{lin bif}
3 \partial_t h + \Pi_V \partial_x (3 v^2 h) = f.
\end{equation}
Let $f \in V \cap Y$ and $h \in V \cap X$,
\[
f = \sum_{j \neq 0} i y_j q_j \in V \cap Y, \quad  
y_{-j} = - y_{j} \in \R, 
\qquad 
h = \sum_{j \neq 0} h_j q_j \in V \cap X, \quad  
h_{-j} = h_{j} \in \R.
\]
Split 
\[
f = \Pi_\mK f + \Pi_\mK^\perp f, \quad 
\Pi_{\mK} f := \sum_{j \in \mK} i y_j q_j , \quad 
\Pi_{\mK}^\perp f := \sum_{j \notin \mK} i y_j q_j , 
\]
and similarly $h = \Pi_\mK h + \Pi_\mK^\perp h$. 
The formula for $\Pi_V \partial_x (v^2 \Pi_\mK h)$ is obtained from \eqref{formula 3-lineare} with $b_j = a_j$ and $c_j = h_j$, namely
\begin{align*} 
\Pi_V \partial_x( v^2 (\Pi_\mK h))
& = \sum_{k \in \mK} 
\Big\{ -3 a_k^2 h_k  
+ 2 a_{k} \Big( \sum_{j \in \mK}  a_{j} h_{j} \Big) 
+ h_{k} \Big( \sum_{j \in \mK}  a_{j}^2 \Big) \Big\} \, i k \, q_{k} .
\end{align*}
Hence
\begin{align*}
3 \partial_t (\Pi_\mK h) + \Pi_V \partial_x (3 v^2 \Pi_\mK h) 
& = 3 \sum_{k \in \mK} 
\Big\{ - |k| h_k 
-3 a_k^2 h_k  
+ 2 a_{k} \Big( \sum_{j \in \mK}  a_{j} h_{j} \Big) 
+ h_{k} \Big( \sum_{j \in \mK}  a_{j}^2 \Big) \Big\} \, i k \, q_{k} 
\end{align*}
which is, replacing $|k|$ by \eqref{pass 1}, 
\[ 
= 3 \sum_{k \in \mK} \Big\{ 
- 2 a_k^2 h_k  + 2 a_{k} \Big( \sum_{j \in \mK}  a_{j} h_{j} \Big) 
\Big\} \, i k \, q_{k}  
= 6 \sum_{k \in \mK} \Big\{ 
- a_k h_k  + \sum_{j \in \mK}  a_{j} h_{j} \Big\} \, a_k \,i k \, q_{k} .
\]
Note that this sum has only Fourier modes in $\mK$; in other words, 
the space of functions in $V$ that are Fourier-supported on $\mK$ is an invariant subspace for the operator $3\partial_t + \Pi_V \partial_x(3v^2 \cdot \,)$ (with, of course, the change of parity $X \to Y$). 

Thus, the equation $3 \partial_t (\Pi_\mK h) + \Pi_V \partial_x (3 v^2 (\Pi_\mK h)) = \Pi_\mK f$ is equivalent to
\[
- a_k h_k  + \sum_{j \in \mK}  a_{j} h_{j} = \frac{y_k}{6 k a_k} =: y'_k
\quad \forall k \in \mK,
\]
namely to the system 
\begin{equation} \label{pass 4}
M \begin{pmatrix} a_{k_1} h_{k_1} \\ \vdots \\ a_{k_m} h_{k_m} \end{pmatrix}
= \begin{pmatrix} y'_{k_1} \\ \vdots \\ y'_{k_m} \end{pmatrix} 
\end{equation}
because $y'_{-k} = y'_k$ for all $k \in \mK$, where $M$ is the $m \times m$ matrix defined above ($1$ on the principal diagonal and $2$ everywhere else). Therefore there exists a unique solution of \eqref{pass 4}, 
\[
h_{k_i} = \frac{1}{a_{k_i}} \, \Big( \a y'_{k_i} + \b \sum_{j \neq i} y'_{k_j} \Big) .
\]
Since $a_j$ solve \eqref{pass 3}, 
\[
\sum_{j \in \mK} h_j^2 \leq C \sum_{j \in \mK} y_j^2 ,
\]
where $C>0$ depends only on $k_1, \ldots, k_m$ and $m$.

Now consider $\Pi_\mK^\perp h, \Pi_\mK^\perp f$. 
In the product 
\[
v^2 (\Pi_\mK^\perp h) = \sum_{j_1, j_2 \in \mK, j_3 \notin \mK} a_{j_1}  a_{j_2} h_{j_3} \, q_{j_1} q_{j_2} q_{j_3} 
\]
only the second case of \eqref{four situations} occurs, 
namely $j_1 = k = -j_2 \in \mK$, $j_3 \notin \mK$. Hence 
\[
\Pi_V \partial_x(v^2 (\Pi_\mK^\perp h)) 
= \sum_{k \in \mK, j \notin \mK} a_k^2 h_j \, i j\, q_j 
= \Big( \sum_{k \in \mK} a_k^2 \Big) \sum_{j \notin \mK} i j \, h_j \, q_j 
= \frac{k_1 + \ldots + k_m}{m-1/2} \, \partial_x(\Pi_\mK^\perp h)
\]
by \eqref{sommina}. Therefore
\[
3 \partial_t (\Pi_\mK^\perp h) 
+ \Pi_V \partial_x (3 v^2 (\Pi_\mK^\perp h)) 
 = 3 \sum_{j \notin \mK} 
\Big( - |j| + \frac{k_1 + \ldots + k_m}{m-1/2} \Big) \, ij h_j\, q_j .
\]
Analogously as above, note that this sum has only Fourier modes out of $\mK$; in other words, 
the space of functions in $V$ that are Fourier-supported on the complementary of $\mK$ is invariant for the operator $3\partial_t + \Pi_V \partial_x(3v^2 \cdot \,)$ (with the change of parity $X \to Y$). 
The condition for the invertibility is 
\begin{equation} \label{non-deg temp 1}
\frac{k_1 + \ldots + k_m}{m-1/2} \, \neq |j| \quad 
\forall j \notin \mK .
\end{equation}
When \eqref{bif cond} holds, $k_1 + \ldots + k_m > k_m (m-1/2)$, therefore $(k_1 + \ldots + k_m)/(m-1/2)$ is automatically out of $\mK$.
Hence \eqref{non-deg temp 1} can be more easily written in this equivalent form:
\begin{equation} \label{non-deg}
\frac{k_1 + \ldots + k_m}{m-1/2} \, \notin \N .
\end{equation}
\eqref{non-deg} implies that 
\begin{equation} \label{non-deg implies}
\Big| - |j| + \frac{k_1 + \ldots + k_m}{m-1/2} \Big| \geq \d |j|
\quad \forall j \neq 0,
\end{equation}
where $\d > 0$ depends only on $k_1, \ldots, k_m$ and $m$. 
Therefore the equation $3 \partial_t (\Pi_\mK^\perp h) + \Pi_V \partial_x (3 v^2 (\Pi_\mK^\perp h)) = \Pi_\mK^\perp g$ has a unique solution $\Pi_\mK^\perp h$, with 
\[
|h_j| \leq \frac{C}{|j|^2} \, |y_j| \, \quad \forall j \neq 0, \ 
j \notin \mK.
\]
Also, by \eqref{sommina} and Lemma \ref{lemma:prodotti in V}, 
$(k_1 + \ldots + k_m)/(m-1/2) = \Pi_C(v^2)$, therefore  
\eqref{non-deg implies} can be written as 
$| \Pi_C(v^2) - |j| | \geq \d |j|$ for all $j \neq 0$. 

We have proved the following result:

\begin{proposition}[Bifurcation for cubic nonlinearities] 
\label{prop:bif}
Let $m \geq 2$. Let 
$0 < k_1 < k_2 < \ldots < k_m$ 
be $m$ positive integers that satisfy \eqref{bif cond} and \eqref{non-deg}.
Then there exist $m$ positive numbers $\rho_1, \ldots, \rho_m > 0$, given by \eqref{rho}, and  constants $C,\d > 0$ that depend only on $k_1, \ldots, k_m$ and have the following property. 

Let $\mK := \{k_1, \ldots, k_m, -k_1, \ldots, -k_m\}$. 
Every function $v = \sum_{j \in \mK} a_j q_j \in V_0 \cap X$ which is Fourier-supported on $\mK$ with 
\[
a_{k_1}^2 = \rho_1, \quad \ldots \quad a_{k_m}^2 = \rho_m 
\]
is a solution of the unperturbed bifurcation equation 
$3 \partial_t v + \Pi_V \partial_x (v^3) = 0$.

For every $f \in V_0 \cap Y$ there exists a unique $h \in V_0 \cap X$ such that 
$3 \partial_t h + \Pi_V \partial_x (3 v^2 h) = f$.

If $f \in H^s$, $s \geq 0$, then $h \in H^{s+1}$, with $\| h \|_{s+1} \leq C \| f \|_{s}$.
Moreover 
\[
| \Pi_C( v^2) - |j| | \geq \d |j|
\quad \forall j \in \Z, \ j \neq 0.
\]
\end{proposition}

\section{\large{The linearized equation}}
\label{sec:The linearized equation}

Remember that 
\[
F(u, \e) = \e^{-2} P_\e\inv \mF ( \e \bar v + \e^2 u, \, \om), \quad 
\om = 1+3\e^2, \quad 
P_\e\inv = \e^{-2} \,\Pi_V + \Pi_W,
\]
where $\bar v := \bar v_1$ is a solution of the unperturbed bifurcation equation \eqref{bif 0} as in Proposition \ref{prop:bif}. 
The linearized operator $F'(u,\e)$ applied to $h$, 
namely the Fr\'echet derivative $\partial_u F(u,\e)[h]$ of $F$ with respect to $u$ in the direction $h$, 
is then 
\[
F'(u,\e)h = \e^{-2} P_\e\inv \mF'(\e \bar v + \e^2 u,\om)[\e^2 h] = P_\e\inv \mL(u,\e)h,
\]
\[
\mL(u,\e)h := \mF'(\e \bar v + \e^2 u,\om)[h]
= \om \partial_t h + (1+a_1) \mH \pa_{xx} h + a_2 \mH \pax h 
+ a_3 \pax h + a_4 \mH h + a_5 h
\]
where the coefficients $a_i = a_i(t,x) = a_i(u,\e)(t,x)$ are periodic functions of $(t,x)$, depending on $u,\e$, and are obtained from $\pa_x(U^3)$ and the partial derivatives of $g_1$, $g_2$ or $g_0$ evaluated at  
$(x,U(t,x), \mH U(t,x),\ldots)$, 
$U := \e \bar v + \e^2 u$. 
For example, in case (I) 
\begin{equation} \label{a2 in case (I)}
a_1(t,x) = (\pa_{y_2} g_2)(x,U(t,x),\mH U_x(t,x)), 
\qquad a_2(t,x) = \pa_x a_1(t,x),
\end{equation}
and in case (II) 
\begin{equation} \label{a2 in case (II)}
a_1(t,x) = (\pa_{y_4} g_0)(x,U(t,x), \mH U(t,x), U_x(t,x), \mH U_{xx}(t,x)),
\qquad a_2(t,x) = 0.
\end{equation}

$\mN(U) = \pa_x (U^3) + O(U^4)$, and $U = \e \bar v + \e^2 u = O(\e)$, therefore $a_1, a_2, a_4 = O(\e^3)$, $a_3,a_5 = O(\e^2)$. 
More precisely: let $\d_0 \in (0,1)$ be a universal constant such that 
\begin{equation}  \label{piccolezza base}
\|(U, \mH U, U_x, \mH U_x, \mH U_{xx})\|_{\linf} < 1 \quad 
\forall U \in H^4(\T^2), \ \| U \|_4 < \d_0.  
\end{equation}

\begin{proposition} \label{prop:a12345}
Let $K > 0$. There exists $\e_0 \in (0,1)$, depending on $K$, 
with the following property:
if $\e \in (0,\e_0)$, 
$\| u \|_4 \leq K$, and 
\begin{equation} \label{pallata}
\| \e \bar v_1 + \e^2 u \|_4 
\leq \e_0 \| \bar v_1 \|_{4} + \e_0^2 \| u \|_4
< \d_0,
\end{equation}
then the coefficients $a_i(u,\e)(t,x)$, $i=1,\ldots,5$ satisfy 
\begin{equation} \label{coeff a12345}
| a_1 |_s + |a_2 |_s + | a_3 - \e^2 3\bar v^2 |_s + |a_4|_s
+ | a_5 - \e^2 (3\bar v^2)_x |_s 
\leq \e^3 C(s,K) (1+\| u \|_{s+4}),  
\quad 0 \leq s \leq r.
\end{equation}
$a_i$ is of class $C^1$ as a function of $(u,\e)$, with
\begin{multline} \label{coeff a12345 der u}
\sum_{i=1,2,4} | \pa_u a_i(u,\e)[h] |_s 
+ | \pa_u a_3(u,\e)[h] - \e^3 6 \bar v h|_s 
+ | \pa_u a_5(u,\e)[h] - \e^3 (6 \bar v h)_x |_s 
\\
\leq \e^4 C(s,K) (\|h\|_{s+4} + \|u\|_{s+4} \|h\|_4), 
\end{multline}
\begin{equation} \label{coeff a12345 der epsilon}
\sum_{i=1,2,4} | \pa_\e a_i(u,\e) |_s 
+ | \pa_\e a_3(u,\e) - \e 6\bar v^2 |_s 
+ | \pa_\e a_5(u,\e) - \e (6\bar v^2)_x |_s 
\leq \e^2 C(s,K) (1 + \|u\|_{s+4}),
\end{equation}
for $0 \leq s \leq r$.
The constant $C(s,K) >0$ depend on $s$, $K$, and $K_{g,r}$ of \eqref{Kgr}. In these estimates the norm $\| \bar v_1 \|_{s+4}$ appears like a constant $C(s)$ depending on $s$. 
\end{proposition}

\begin{proof} In Section \ref{sec:appendix proofs}. 
\end{proof}

\begin{remark} In general, the inequality $\|\mH u\|_{L^\infty} \leq C \|u\|_{L^\infty}$ is false (see, for example, \cite{Koosis}), while it is trivially true that $\| \mH u \|_s \leq \| u \|_s$ for all $s$. Therefore to obtain the estimate $\|\mH u_{xx}\|_{L^\infty} \leq C \|u\|_4$ (which is used to prove \eqref{piccolezza base}) the right chain of inequalities is $\| \mH u_{xx} \|_{L^\infty} \leq C \| \mH u_{xx} \|_2 \leq C \|u_{xx} \|_2 \leq C \| u \|_4$.
\end{remark}

Since $\bar v, u \in X$, 
\begin{equation*} 
a_1, a_3, a_4 \in X, \quad 
a_2, a_5 \in Y,
\end{equation*}
and $\mL(u,\e)$ maps $X \cap H^2 \to Y$.

As a pseudo-differential operator, we write 
\begin{equation*} 
\mL := \mL(u,\e) 
= \om \partial_t + (1+a_1(t,x)) \mH \pa_{xx} 
+ a_2(t,x) \mH \pax + a_3(t,x) \pax + a_4(t,x) \mH + a_5(t,x).
\end{equation*}
In this operator notation a function $p(t,x)$ is identified with the multiplication operator $h \mapsto p(t,x)h$, and the composition is understood: for example, $\pax p$ is the operator $p \pax + p_x$, because
$\pax (ph) = p \pax h + p_x h$.

To emphasize that we are in the space of zero mean functions, write 
\begin{equation*} 
\tilde\mL := \pp \mL \pp,
\end{equation*}
where $\pp = I - \Pi_C$ is defined in \eqref{Z0 pp}. 
Since $F$ maps $X_0 \to Y$, also $F'(u,\e)$ maps $X_0 \to Y$, 
therefore 
\[
\tilde\mL h = \mL h \quad \forall h \in X_0 
\]
because $\pp h = h$ and $\pp f = f$ for all $h \in X_0$, $f \in Y$.

\section{\large{Reduction to constant coefficients}}
\label{sec:reduction}

In this section the linearized operator is conjugated to a linear operator with constant coefficients plus a regularizing rest. 
The transformation is performed in several steps.

\subsection{Change of variables} 
\label{sec:change of variables}

As a first step in the reduction proof, we construct a change of variables that transforms $\mL$ into a new operator with constant coefficients in the highest order derivatives $\pa_t$ and $\mH \pa_{xx}$. 
Since $\mL$ maps $X_0$ into $Y$, we want that our transformation maps $X_0 \to X_0$ and $Y \to Y$.

We consider diffeomorphisms of the torus $(t,x) \in \T^2$ which are the composition of ($i$) a time-dependent change of the space variable $x \to x+\b(t,x)$, and ($ii$) a change of the time variable $t \to t+\a(t)$ that does not depend on space. 
Diffeomorphisms of this type preserve the special role of the time variable as ``a parameter'' with respect to pseudo-differential operators of the space variable like $\mH$. 
 
Let 
\begin{equation*} 
\psi : \T^2 \to \T^2, \quad 
\psi(t,x) := (t+\a(t),\, x+\b(t,x)) = (\t,y)
\end{equation*}
and let $\Psi$ be the transformation $\Psi : u \mapsto \Psi u$, 
\[
(\Psi u)(t,x) := u(\psi(t,x)) = u(t+\a(t),\, x+\b(t,x)) = u(\t,y).
\]
$\a(t)$ and $\b(t,x)$ are periodic functions in $Y$ to be determined.

The conjugate $\Psi\inv p \Psi$ of any multiplication operator $p : h(t,x) \mapsto p(t,x)h(t,x)$ is the multiplication operator $(\Psi\inv p)$ that maps $v(\t,y) \mapsto (\Psi\inv p)(\t,y) \, v(\t,y)$.
By conjugation, the differential operators become 
\[
\Psi\inv \pa_t \Psi = [1+(\Psi\inv \a')(\t)] \,\pa_\t 
+ (\Psi\inv \b_t)(\t,y) \, \pa_y, 
\quad 
\Psi\inv \pa_x \Psi = [1+(\Psi\inv \b_x)(\t,y)]\, \pa_y , \quad 
\]
\[
\Psi\inv \pa_{xx} \Psi = [1+(\Psi\inv \b_x)(\t,y)]^2 \, \pa_{yy} 
+ (\Psi\inv \b_{xx})(\t,y) \, \pa_y , 
\quad 
\Psi\inv \mH \Psi = \mH + \mR_\mH,
\]
where $\mR_\mH$ is defined by the last equality, and it is regularizing in space, bounded in time, see Lemma \ref{lemma:commutators}$(iii)$.

Since $\a,\b \in Y$, $\Psi$ maps $X \to X$ and $Y \to Y$. 
However, in general, $\Psi$ does not map $X_0$ into $X_0$.%
\footnote{For example: let $u(t,x) = \cos t \in X_0$, $\b = 0$ and $\a$ such that the inverse of $t \mapsto t+\a(t)$ is $\t \mapsto \t + (1/2) \sin \t$. 
Changing variable in the integral, $\int_{\T^2} (\Psi u) \, dt \, dx = (1/2) \int_{\T^2} \cos^2 \t \, d\t \, dy > 0$, therefore $\Psi u \notin X_0$.}  
To obtain a transformation of $X_0$ onto itself, consider the projection onto $Z_0$,
\[
\tilde \Psi := \PP \Psi \PP.
\]
Since $\Psi\inv \Pi_C = \Pi_C$, one has  
$\pp \Psi\inv \Pi_C = \pp \Pi_C = 0$, and 
\begin{equation} \label{chiavetta}
\pp \Psi\inv \pp = \pp \Psi\inv (I-\Pi_C) = \pp \Psi\inv.
\end{equation}
As a consequence, 
\begin{equation*} 
(\pp \Psi\inv \pp)(\pp \Psi \pp) 
= \pp \Psi\inv \pp \Psi \pp 
= \pp \Psi\inv \Psi \pp 
= \pp,
\end{equation*}
therefore $\tilde \Psi : Z_0 \to Z_0$ is invertible, with inverse 
\[
(\tilde \Psi)\inv = (\pp \Psi \pp)\inv = \pp \Psi\inv \pp.
\]
Thus $\tilde \Psi$ is a linear bijective operator of $X_0 \to X_0$ and $Y \to Y$. Also, 
\begin{equation} \label{[PiC,Psi]}
[\Psi,\pp]h = [\Pi_C, \Psi]h 
= \Pi_C (\tilde\a' + \tilde\b_y + \tilde\a' \tilde\b_y )h 
= \frac{1}{(2\pi)^2} \int_{\T^2} 
h\,\big( \tilde\a' + \tilde\b_y + \tilde\a' \tilde\b_y \big) \,d\t\,dy, 
\end{equation}
where $(\t,y) \mapsto (\t+\tilde\a(\t), \, y+\tilde\b(\t,y)) = \psi\inv(\t,y)$ is the inverse of $\psi$, and similarly 
\[
[\Psi\inv,\pp] = [\Pi_C, \Psi\inv] = \Pi_C (\a'+\b_x+\a'\b_x).
\] 
These commutators are regularizing operators, both in space and time (by integrations by parts, any derivative applied to the argument $h$ moves to $\a,\b$ or $\tilde\a,\tilde\b$).

By \eqref{chiavetta}, 
\begin{equation*} 
\tilde\mL_1 := \tilde \Psi\inv \tilde \mL \tilde \Psi
= \pp \Psi\inv \pp \mL \pp \Psi \pp 
= \pp \Psi\inv \mL \pp \Psi \pp = \pp \mL_1 \pp ,
\end{equation*}
where 
\begin{align} 
\mL_1 & = \om [1+(\Psi\inv \a')(\t)] \,\pa_\t 
+ [1+(\Psi\inv a_1)(\t,y)] \, [1+(\Psi\inv \b_x)(\t,y)]^2 \, \pa_{yy} \mH
\notag 
\\ & \quad 
+ \{ [1+(\Psi\inv a_1)(\t,y)] \, (\Psi\inv \b_{xx})(\t,y) \, 
+ (\Psi\inv a_2)(\t,y) [1+(\Psi\inv \b_x)(\t,y)] \} \, \pa_y \mH 
\notag
\\ & \quad 
+ \{ \om (\Psi\inv \b_t)(\t,y) \, 
+ (\Psi\inv a_3)(\t,y) [1+(\Psi\inv \b_x)(\t,y)] \} \, \pa_y  
\notag
\\ & \quad 
+ (\Psi\inv a_4)(\t,y) \mH 
+ (\Psi\inv a_5)(\t,y) 
+ \mR_1 ,
\notag
\\ 
& \notag
\\
\label{mR1}
\mR_1 & = 
[1+(\Psi\inv a_1)(\t,y)] \, [1+(\Psi\inv \b_x)(\t,y)]^2 \, \pa_{yy} \mR_\mH
\\ & \quad 
+ \{ [1+(\Psi\inv a_1)(\t,y)] \, (\Psi\inv \b_{xx})(\t,y) \, 
+ (\Psi\inv a_2)(\t,y) [1+(\Psi\inv \b_x)(\t,y)] \}\, \pa_y \mR_\mH 
\notag
\\ & \quad 
+ (\Psi\inv a_4)(\t,y) \mR_\mH 
- \pp (\Psi\inv a_5)(\t,y) [\Pi_C, \Psi]
\notag 
\end{align}
because $\mL \Pi_C = a_5 \Pi_C$. 
We look for $\a,\b$ such that the coefficients of $\pa_\t$ and $\pa_{yy}\mH$ are proportional, namely
\begin{equation} \label{propoli 1}
[1+(\Psi\inv a_1)(\t,y)] \, [1+(\Psi\inv \b_x)(\t,y)]^2 
= \mu_2 \, [1+(\Psi\inv \a')(\t)] 
\end{equation}
for some $\mu_2 \in \R$.  \eqref{propoli 1} is equivalent to 
\begin{equation} \label{propoli 2}
\big( 1+a_1(t,x) \big) \, \big(1+\b_x(t,x)\big)^2 
= \mu_2 \, (1+\a'(t)).
\end{equation}
Take the square root of \eqref{propoli 2},
\begin{equation} \label{propoli 3}
1+\b_x(t,x) = \mu_2^{1/2} \, (1+\a'(t))^{1/2} \big( 1+a_1(t,x) \big)^{-1/2},
\end{equation}
and integrate in $dx$,
\begin{equation*} 
1 = \mu_2^{1/2} \, (1+\a'(t))^{1/2} \frac{1}{2\p} \intp (1+a_1)^{-1/2} dx.
\end{equation*}
Take the square, 
\begin{equation} \label{propoli 4}
\mu_2 \, (1+\a'(t)) = \Big( \frac{1}{2\p} \intp (1+a_1)^{-1/2} dx \Big)^{-2} 
=: \rho(t).
\end{equation}
Integrating in $dt$ determines $\mu_2 \in \R$, 
\begin{equation*} 
\mu_2 \, = \Pi_C(\rho) 
= \frac{1}{2\p} \intp \Big( \frac{1}{2\p} \intp (1+a_1)^{-1/2} dx \Big)^{-2} \,dt,
\end{equation*}
then $\a(t) \in Y$ is also determined, 
\begin{equation*} 
\a(t) = \frac{1}{\mu_2} \, \pa_t\inv(\Pi_T \rho)(t).
\end{equation*}
Since $a_1 \in X$, also $\rho \in X$, therefore $\a \in Y$, as it was required. \eqref{propoli 3} gives 
\begin{equation} \label{propoli 5}
\b_x = \rho^{1/2} \, (1+a_1)^{-1/2} - 1 
= \frac{p}{\Pi_{T+C}(p)}\,-1 
= \frac{\Pi_E(p)}{\Pi_{T+C}(p)}\,, \quad 
p := (1+a_1)^{-1/2},
\end{equation}
therefore the $Z_E$-component of $\b$ is determined, 
\begin{equation*} 
(\Pi_E \b)(t,x) = \frac{1}{(\Pi_T p)(t) + \Pi_C(p)} \, (\pa_x\inv \Pi_E p)(t,x).
\end{equation*}
Since $a_1 \in X$, also $p \in X$, and $\Pi_E \b \in Y$, as it was required. The $Z_T$-component of $\b$  
will be determined later. 
With this choice of $\a,\b$, \eqref{propoli 1} is satisfied.
By \eqref{propoli 1}, 
\[
\mL_1 = \mM \mL_2, 
\]
where $\mM$ is the multiplication operator of factor 
$[1+(\Psi\inv \a')(\t)]$, 
\begin{equation} \label{mL3}
\mL_2 = 
\om \pa_\t + \mu_2 \pa_{yy} \mH 
+ a_6(\t,y) \, \pa_y \mH 
+ a_7(\t,y) \, \pa_y 
+ a_8(\t,y) \, \mH 
+ a_9(\t,y) \, 
+ \mR_2,
\end{equation}
\begin{align*}
& a_6(\t,y) := \Psi\inv \Big( \frac{(1+a_1) \b_{xx} + a_2(1+\b_x)}{1+\a'} \Big)(\t,y),
& 
& a_8(\t,y) := \Psi\inv \Big( \frac{a_4}{1+\a'} \Big)(\t,y),
\\ 
& a_7(\t,y) := \Psi\inv \Big( \frac{\om \b_t + a_3(1+\b_x)}{1+\a'} \Big)(\t,y),
& 
& a_9(\t,y) := \Psi\inv \Big( \frac{a_5}{1+\a'} \Big)(\t,y),
\\
& \mR_2 := \frac{1}{1+(\Psi\inv \a')(\t)}\, \mR_1. 
&
&
\end{align*}
We show that 
\begin{equation} \label{trucco magico}
a_6(\t,y) \in Z_E.
\end{equation}
For each fixed $\t = t+\a(t)$, changing variable $y = x+\b(t,x)$, 
$dy = (1+\b_x(t,x))\,dx$ in the integral, 
\[
\intp a_6(\t,y) \, dy 
= \intp \frac{(1+a_1(t,x)) \b_{xx}(t,x) + a_2(t,x)(1+\b_x(t,x))}{1+\a'(t)} (1+\b_x(t,x)) \, dx .
\]
By \eqref{propoli 2}, 
\[
\frac{(1+a_1) \b_{xx} + a_2(1+\b_x)}{1+\a'}\,(1+\b_x) 
= \mu_2\, \frac{(1+a_1) \b_{xx} + a_2(1+\b_x)}{(1+a_1)(1+\b_x)} .
\]
In case (I) $a_2 = (a_1)_x$ (see \eqref{a2 in case (I)}), therefore 
\[
\frac{(1+a_1) \b_{xx} + a_2(1+\b_x)}{(1+a_1)(1+\b_x)} \,
= \frac{[(1+a_1)(1+\b_x)]_x}{(1+a_1)(1+\b_x)} \,
= \pax \{ \log [(1+a_1)(1+\b_x)] \}\,;
\]
in case (II) $a_2 = 0$ (see \eqref{a2 in case (II)}), therefore 
\[
\frac{(1+a_1) \b_{xx} + a_2(1+\b_x)}{(1+a_1)(1+\b_x)} \,
= \frac{\b_{xx}}{1+\b_x} \,
= \pax \{ \log (1+\b_x) \}.
\]
Hence in both cases (I) and (II), by periodicity, $\intp a_6 \, dy = 0$, which is \eqref{trucco magico}. 

\begin{remark} \label{rem:FIO}
The assumptions (I),(II) on the nonlinearity $\mN_4(u)$ have been used to prove \eqref{trucco magico}. In more general situations, when (I)(II) are not satisfied, a term $b(\t) \mH \pa_y$ also appears, where $b(\t) \in Z_T$ is the $Z_T$-component of the coefficient $a_6$
(which here is zero by \eqref{trucco magico}). 
This term can be removed by using the Fourier integral operator 
\[
u(\t,y) = \sum_{j \in \Z} u_j(\t) \, e^{ijy} \ \mapsto \ 
Au(\t,y) = \sum_{j \in \Z} u_j(\t) \,e^{ijy + |j|p(\t)},
\]
where $p(\t) = \pa_{\t}^{-1} b(\t)$.
\end{remark}

Now we choose the $Z_T$-component of $\b$ so that $\Pi_T a_7 = 0$.
Denote $\g(t) := (\Pi_T \b)(t)$. As above, 
\[
\frac{1}{2\p} \int_\T a_7(\t,y)\,dy 
= \frac{1}{2\p} \intp  
\frac{\om \b_t(t,x) + a_3(t,x)(1+\b_x(t,x))}{1+\a'(t)} \, (1+\b_x(t,x)) \, dx.
\]
This integral is equal to some constant $\mu_1 \in \R$ if and only if 
\begin{equation} \label{propoli sigma}
\om \g'(t) + \s(t) = \mu_1 (1 + \a'(t)), \quad 
\s(t) := \frac{1}{2\p} \intp \Big( \om \b^E_t (1+\b^E_x) 
+ a_3 (1+\b^E_x)^2 \Big) \, dx, \quad 
\b^E := \Pi_E \b.
\end{equation}
Hence an integration in $dt$ on $\T$ determines $\mu_1 \in \R$ and $\g \in Z_T$,  
\begin{equation}  \label{nu gamma}
\mu_1 = \Pi_C (\s), 
\quad 
\g(t) = \frac{\mu_1 \a(t) - (\pa_t\inv \Pi_T \s)(t)}{\om} \, \in Z_T.
\end{equation}
Thus
\begin{equation}  \label{a7 mu1}
\Pi_C(a_7) = \mu_1, \quad a_7 - \mu_1 \in Z_E.
\end{equation}
$\s \in X$ because $a_3 \in X$, therefore $\g \in Y$ as it was required.  
Hence $\b = \g + (\Pi_E \b) \in Y$.  
As a consequence, 
\begin{equation} \label{a 6789 parity}
a_6, a_9 \in Y, \quad  a_7, a_8 \in X.
\end{equation}
Since $I = \pp + \Pi_C$, 
\begin{equation*} 
\tilde \mL_1 = \pp \mL_1 \pp 
= \pp \mM \mL_2 \pp 
= (\pp \mM \pp) (\pp \mL_2 \pp) - \pp \mM \Pi_C \mL_2 \pp 
= \tilde \mM \tilde \mL_3,
\end{equation*} 
where
\[
\tilde \mM := \pp \mM \pp, \quad 
\tilde \mL_3 := \pp \mL_3 \pp, \quad 
\mL_3 = \mL_2 - \tilde\mM\inv  \mM \Pi_C \mL_2 .
\]
Thus 
\[
\mL_3 = 
\om \pa_\t + \mu_2 \pa_{yy} \mH + a_6(\t,y) \, \pa_y \mH 
+ a_7(\t,y) \, \pa_y + a_8(\t,y) \, \mH + a_9(\t,y) \, + \mR_3,
\]
\begin{equation*} 
\mR_3 := \mR_2 - \tilde\mM\inv \mM \Pi_C \mL_2.
\end{equation*}
$\tilde \mM$ is invertible, its inverse $\tilde \mM\inv$ maps $X_0 \to X_0$ and $Y \to Y$, and  
\begin{equation}  \label{formula tilde mM inv}
\tilde \mM\inv h = mh - \frac{m}{\Pi_C(m)}\, \Pi_C(mh), \qquad m(\t) := \frac{1}{1+(\Psi\inv \a')(\t)}\,,
\end{equation}
whence 
\begin{equation*} 
\tilde \mM\inv \mM \Pi_C = -\Big( \frac{(\pp m)}{\Pi_C(m)}\Big) \, \Pi_C. 
\end{equation*}
Formula \eqref{formula tilde mM inv} can be proved by a direct calculation: $\tilde\mM \tilde\mM\inv h 
= \tilde\mM\inv \tilde\mM h = h$ for all $h \in Z_0$.

From Proposition \ref{prop:a12345} 
and the explicit formulae above, $\mu_2, \mu_1, \rho, \a,\b,  \g$ all depend on $(u,\e)$ in a $C^1$ way, and the following estimates hold.

\begin{proposition}
\label{prop:Psi mM mR3}
Let $K > 0$. There exists $\e_0 \in (0,1)$, depending on $K$, 
such that, if $\e \in (0,\e_0)$, 
$\| u \|_8 \leq K$, and $\|u\|_4,\e_0$ satisfy \eqref{pallata}, 
then all the following inequalities hold. 

$\mu_2(u,\e)$ and $\mu_1(u,\e)$ satisfy 
\begin{align} 
\label{stima mu2} 
|\mu_2 - 1| 
& \leq \e^3 C(K), 
&
|\pa_u \mu_2[h]| 
& \leq \e^4 C(K) \|h\|_4, 
&
|\pa_\e \mu_2| 
& \leq \e^2 C(K),
\\
\label{stima mu1}
| \mu_1 - \e^2 \Pi_C(3 \bar v^2)| 
& \leq \e^3 C(K), 
&
| \pa_u \mu_1[h]| 
& \leq \e^4 C(K) \|h\|_{5}, 
&
| \pa_\e \mu_1 - \e \Pi_C(6 \bar v^2)| 
&\leq \e^2 C(K).
\end{align}
$\psi(t,x) = (t+\a(t),x+\b(t,x))$ and its inverse 
$\psi\inv(\t,y) = (\t+\tilde\a(\t), y+\tilde\b(\t,y))$ 
are diffeomorphisms of $\T^2$, with 
\begin{equation} \label{stima diffeo psi basic}
|\a |_{1} + |\b |_{1} + |\tilde\a |_{1} + |\tilde\b |_{1} 
< \e^3 C(K) < 1/2,
\quad 
|\a|_{s} + | \b|_s + |\tilde\a |_{s} + |\tilde\b |_{s} 
\leq \e^3 C(s,K) (1+\|u\|_{s+4}), 
\end{equation}
for all $1 \leq s \leq r$.
$\a,\b,\tilde\a,\tilde\b$ are $C^1$ functions of $(u,\e)$.  
For $1 \leq s \leq r-1$, their derivatives satisfy
\begin{align} 
\label{stima alfa beta der u}
|\pa_u \a[h] |_{s} + |\pa_u \b[h] |_s 
+ |\pa_u \tilde\a[h] |_{s} + |\pa_u \tilde\b[h] |_s 
& \leq \e^4 C(s,K) (\| h \|_{s+4} + \| u \|_{s+5} \| h \|_{5}), 
\\
\label{stima alfa beta der epsilon}
|\pa_\e \a|_{s} + |\pa_\e \b|_s
+ |\pa_\e \tilde\a|_{s} + |\pa_\e \tilde\b|_s
& \leq \e^2 C(s,K) (1+ \| u \|_{s+5}).
\end{align}
The operators $\Psi, \Psi\inv$ satisfy
\begin{equation} \label{stima Psi}
\| \Psi f \|_s + \| \Psi\inv f \|_s 
\leq C(s,K) (\|f\|_s + \|u\|_{s+4} \|f\|_1),
\quad 
\| \Psi f \|_0 + \| \Psi\inv f \|_0 
\leq 2 \|f\|_0,
\end{equation}
\begin{equation} \label{stima Psi-I}
\| (\Psi-I)f \|_s + \| (\Psi\inv-I)f \|_s 
\leq \e^3 C(s,K) (\|f\|_{s+1} + \|u\|_{s+5} \|f\|_1),
\end{equation}
for all $1 \leq s \leq r$.
\eqref{stima Psi},\eqref{stima Psi-I} also hold for $\tilde\Psi, \tilde\Psi\inv$. 
Moreover, for $1 \leq s \leq r$,
\begin{gather} 
\label{stima Psi infty}
| \Psi f |_s + | \Psi\inv f |_s 
\leq C(s,K) (|f|_s + \|u\|_{s+4} |f|_1),
\quad 
| \Psi f |_0 = | \Psi\inv f |_0 = |f|_0,
\\ 
\label{stima Psi-I infty}
| (\Psi-I)f |_s + | (\Psi\inv-I)f |_s 
\leq \e^3 C(s,K) (|f|_{s+1} + \|u\|_{s+5} |f|_1).
\end{gather}
The operators $\Psi,\Psi\inv$ depend on $(u,\e)$ via $\a,\b$. 
The derivatives of $\Psi f$, $\Psi\inv f$ with respect to $u$ in the direction $h$ and with respect to $\e$ satisfy
\begin{align}
\label{stima Psi der u}
\| \pa_u (\Psi f)[h] \|_s 
+ \| \pa_u (\Psi\inv f)[h] \|_s 
& \leq \e^4 C(s,K) ( \| f \|_{s+1} \| h \|_5 
+ \| f \|_{1} \| h \|_{s+4}  
+ \| u \|_{s+5} \| f \|_{1} \| h \|_5 ),
\\
\label{stima Psi der epsilon}
\| \pa_\e \Psi f \|_s 
+ \| \pa_\e \Psi\inv f \|_s 
& \leq \e^2 C(s,K) ( \| f \|_{s+1} + \| u \|_{s+5} \| f \|_{1}),
\end{align}
for all $1 \leq s \leq r-1$.
\eqref{stima Psi der u} and \eqref{stima Psi der epsilon} also hold with $| \ |_s$ instead of $\| \ \|_s$ on the left-hand side and on $f$.
\eqref{stima Psi der u} and \eqref{stima Psi der epsilon} also hold 
for $\tilde\Psi, \tilde\Psi\inv$.

For $2 \leq s \leq r$,  
\begin{equation} \label{stima mM-I}
\| (\tilde\mM-I) f\|_s + \|(\tilde\mM\inv -I)f\|_s 
\leq \e^3 C(s,K) ( \|f\|_s + \|u\|_{s+4} \| f \|_2).
\end{equation} 
The derivatives of $\tilde\mM f$, $\tilde\mM\inv f$ with respect to $u$ in the direction $h$ and with respect to $\e$ satisfy
\begin{align}
\label{stima mM der u}
\| \pa_u (\tilde\mM f)[h] \|_s 
+ \| \pa_u (\tilde\mM\inv f)[h] \|_s 
& \leq \e^4 C(s,K) ( \| f \|_{s} \| h \|_6 
+ \| f \|_{2} \| h \|_{s+5}  
+ \| u \|_{s+6} \| f \|_{2} \| h \|_5 ),
\\
\label{stima mM der epsilon}
\| \pa_\e \tilde\mM f \|_s 
+ \| \pa_\e \tilde\mM\inv f \|_s 
& \leq \e^2 C(s,K) 
( \| f \|_{s} + \| u \|_{s+6} \| f \|_{2}),
\end{align}
for $2 \leq s \leq r-2$.

The coefficients of $\mL_3$ satisfy
\begin{align} \label{a 6789 epsilon}
| a_6 |_{s} + | a_7 - \e^2 3 \bar v^2 |_{s} 
+ | a_8 |_s + | a_9 - \e^2 (3 \bar v^2)_x |_s 
& \leq \e^3 C(s,K) (1+\|u\|_{s+6}),
\\
| \pa_u a_6[h] |_{s} + | \pa_u a_7[h] |_{s}
+ | \pa_u a_8[h] |_{s} + | \pa_u a_9[h] |_{s} 
& \leq \e^4 C(s,K) (\| h \|_{s+4} + \| u \|_{s+6} \| h \|_5),
\\
| \pa_\e a_6 |_{s} + | \pa_\e a_7 - \e 6 \bar v^2 |_{s} 
+ | \pa_\e a_8 |_s + | \pa_\e a_9 - \e (6 \bar v^2)_x |_s 
& \leq \e^2 C(s,K) (1+\|u\|_{s+6}).
\end{align}

For $s,m_1,m_2 \geq 0$, $m=m_1+m_2$, $m+s+1 \leq r$, 
\begin{equation} \label{stima mRmH}
\| \pa_x^{m_1} \mR_\mH \pa_x^{m_2} f \|_s 
\leq \e^3 C(s,m,K) \big( \| f \|_s (1+\|u\|_{m+5})  
+  \|u\|_{s+m+5} \| f \|_0).
\end{equation}
For $m,s \geq 0$, $m+s+3 \leq r$, 
\begin{equation} \label{stima mR123}
\| \mR_i \pa_y^m f \|_s 
\leq \e^3 C(s,m,K) \big( \|f\|_s (1 + \|u\|_{m+7}) + \| f \|_0 \|u\|_{s+m+7}\big), 
\quad  i=1,2,3. 
\end{equation}
\end{proposition}

\begin{proof} 
In Section \ref{sec:appendix proofs}.
\end{proof}

\begin{remark} \label{rem:loss of 1 for Psi-I}
The loss of one derivative for the difference $\Psi-I$ in \eqref{stima Psi-I},\eqref{stima Psi-I infty}  
is typical of any change of variables: in general, if we want to estimate a difference $h(x+p(x)) - h(x)$ with a factor of size $p$, 
we can do nothing but making a derivative, $h(x+p(x)) - h(x) \simeq h'(x) p(x)$.
\end{remark}

\subsection{Descent method: conjugation with pseudo-differential operators}
\label{subsec:descent}

We construct an invertible linear operator $\tilde \Phi = \pp \Phi \pp$ that maps $X_0 \to X_0$ and $Y \to Y$ 
and conjugates $\tilde\mL_3$ to a new operator 
\begin{equation} \label{mL4}
\tilde \mL_4 := \tilde \Phi\inv \tilde \mL_3 \tilde \Phi 
= \pp \mL_4 \pp, \quad \mL_4 = \mD + \mR, 
\end{equation}
where $\mD$ has constant coefficients and the remainder $\mR$ is regularizing in space, bounded in time. 
We look for $\mD$ of the form 
\[
\mD = \om \patau + \mu_2 \pa_{yy} \mH + \mu_1 \pa_y + \nu_0' + \nu_0 \mH 
+ (\nu_{-1}' + \nu_{-1} \mH) \pa_y\inv 
+ (\nu_{-2}' + \nu_{-2} \mH) \pa_y^{-2},
\]
where $\mu_2,\mu_1$ are the constants calculated in the previous section,  
$\nu_k, \nu_k'$, $k=0,-1,-2$ are constants to be determined. 
We look for $\Phi$ such that $(\pp \mL_3 \pp)(\pp \Phi \pp) - (\pp \Phi \pp)(\pp \mD \pp)$ is an operator of order $\leq -3$ in $y$. 
Write $\Phi$ as   
\[
\Ph = \Ph_0 + \Ph_1 + \Ph_2 + \Ph_3, 
\quad 
\Ph_k = (\a\k + \mH \b\k) \partial_y^{-k}, 
\quad 
k=0,1,2,3,
\]
namely $\Ph_k h = \a\k \partial_y^{-k} h + \mH (\b\k \partial_y^{-k} h)$,  
where $\a\k(\t,y)$, $\b\k(\t,y)$ are functions to be determined. 
$\Phi$ is close to the identity if $\a\zero$ is close to 1 and all the other $\a\k, \b\k$ are small.

Calculate and write the terms of order $1,0,-1,-2$ in $y$, and move all the `$\mH$' on the left-hand side, introducing the corresponding commutators (for example, write $\a \mH$ as $\mH \a + [\a,\mH]$).  
Note that 
\[
\mH^2 = \mH \mH = -\Pi_E = -I + \pepe, \qquad 
\pepe := I - \Pi_E = \Pi_T + \Pi_C.
\]
$\pepe$ is regularizing in $y$ because it is the operator that takes the mean of a function with respect to $y$. Therefore, up to a regularizing rest, sums and products of terms of the type $(\a+\mH\b)$ follow the same algebraic rules as those of complex numbers, where the role of $i$ is played by $\mH$. 
As a consequence, to perform the calculations up to terms containing $\pepe$ or commutators with $\mH$ it is comfortable to introduce the complex notation: 
\[ 
\begin{cases}
f\k := \a\k + i \b\k, \quad 
\mL_3 = \om \pa_\t + \mu_2 i \pa_{yy} + a_{76} \pa_y + a_{98} + \mR_3, 
\quad
a_{76} := a_7 + i a_6, \quad
a_{98} := a_9 + i a_8, 
\\
\mD = \om \pa_\t + \mu_2 i \pa_{yy} + \mu_1 \pa_y + c_0 + c_{-1} \pa_y\inv 
+ c_{-2} \pa_y^{-2}, 
\quad
c_{-k} := \nu_{-k}' + i \nu_{-k}, 
\\
\text{where $i$ means $\mH$.}
\end{cases}
\]
We stress that this is only a notation, as $\mH$ maps real-valued functions into real-valued functions, and therefore $\a + \mH \b$ is real when $\a,\b$ are real.
Straightforward calculations (use $\pp = I - \Pi_C$ for $a_9$) give 
\begin{equation} \label{commu pseudo}
\tilde \mL_3 \tilde \Phi - \tilde \Phi \tilde \mD 
= \pp (T_1 \pa_y + T_0 + T_{-1} \pa_y\inv + + T_{-2} \pa_y^{-2} + \mR_4) \pp,
\end{equation} 
where the coefficients $T_k$ are 
\begin{align}
T_1 & = Q f\zero, 
&
T_{-1} & = Q f\due + S f\uno - c_{-1}\,f\zero, \notag
\\
T_0 & = Q f\uno + S f\zero, \label{T10-1} 
&
T_{-2} & = Q f\tre + S f\due - c_{-1}\,f\uno - c_{-2}\,f\zero, 
\end{align}
$Q,S$ mean 
\[
Q f := 2i\mu_2 f_y + (a_{76}\,-\nu)\, f, \quad
S f := (\mL_3 - \mR_3 - c_0)f 
= \om f_\t + i \mu_2 f_{yy} + a_{76} f_y + (a_{98} - c_0) f,
\]  
and the rest $\mR_4$ is the sum $\mR_3 \pp \Phi - a_9 \Pi_C \Phi$ $+$  terms of order $\pa_y^{-3}$ $+$ other regularizing terms that
\begin{itemize}
\item[(a)] contain a commutator $[g,\mH]$, where $g \in \{a_j,\a\k,\b\k: j=6,7,8,9, \ k=0,1,2,3\}$; or
\item[(b)] contain $\pepe$.
\end{itemize}
The complete formula for $\mR_4$ is in Appendix \ref{sec:appendix proofs}.
For example, typical terms are
\[
\pepe \b\zero \pa_y^2, \quad 
a_6 \pepe \b\uno_y \pa_y^{-1}, \quad
[a_6,\mH] \a\zero_y, \quad
[\b\uno,\mH]\pa_y.
\]

Now we choose $\nu_i, \a\k, \b\k$ such that all $T_n$, $n=1,0,-1,-2$, vanish. Every $T_n$ is an operator of the form 
$T_n h = p_n h + \mH (q_n h)$ for some functions $p_n(\t,y), q_n(\t,y)$. 
Thus $T_n = 0$ if 
\begin{equation} \label{system Tn}
p_n = 0, \quad q_n = 0.
\end{equation} 
To solve \eqref{system Tn}, which is a system of two equations in the real-valued unknowns $\a\k, \b\k$, we use complex notation again.  
Consider the complex-valued unknown $f\k = \a\k + i \b\k$, where now $i$ is the standard imaginary unit of $\C$. 
Then the real system \eqref{system Tn} is equivalent to the complex ODE 
$Q f\zero = 0$ for $n=1$, and similar complex equations for $n=0,-1,-2$, according to \eqref{T10-1}. 
Hence we look for complex-valued solutions $f\k$ of the four complex equations $T_n=0$, $n=1,0,-1,-2$.

\medskip

\emph{Reduction of $T_1$. --- } 
Let 
\[
a_{76}^E(\t,y) := a_{76}(\t,y) - \mu_1 = a_7(\t,y) - \mu_1 + i a_6(\t,y).
\]
Remember that $a_7 - \nu, a_6 \in Z_E$ (see \eqref{trucco magico},\eqref{a7 mu1}).
$T_1 = 0$ if
\begin{equation} \label{complex ode T1}
Q f\zero = 2i \mu_2 f\zero_y + a_{76}^E(\t,y) \,f\zero = 0.
\end{equation}
The solutions of \eqref{complex ode T1} are the exponentials $f\zero = \exp(\ph)$, where $\ph(\t,y)$ satisfies 
\begin{equation} \label{reason}
2 i \mu_2 \ph_y + a_{76}^E(\t,y) = 0.
\end{equation}
\eqref{reason} determines the $Z_E$-component of $\ph$, 
\begin{equation*} 
(\Pi_E \ph)(\t,y) = \frac{i}{2\mu_2}\,(\pa_y\inv a_{76}^E)(\t,y) 
= -\frac{1}{2\mu_2}\, (\pa_y\inv a_6)(\t,y) + i \frac{1}{2\mu_2}\, (\pa_y\inv \Pi_E a_7)(\t,y).
\end{equation*}

\medskip

\emph{Reduction of $T_0$. --- } 
Since $f\zero = \exp(\ph)$, 
\begin{equation} \label{S f zero}
S f\zero = f\zero g\zero, \quad 
 g\zero := \om \ph_\t + i \mu_2 (\ph_y^2 + \ph_{yy}) + a_{76} \ph_y + (a_{98}-c_0).
\end{equation}
Moreover
\[
i \mu_2 \ph_y^2 + a_{76} \ph_y = \frac{i}{4\mu_2}\, (a_{76}^E)^2 + \frac{i}{2\mu_2}\, \nu\, a_{76}^E
\]
by \eqref{reason} and because $a_{76} = a_{76}^E + \nu$.
Since $Q f\zero = 0$, we solve the equation $T_0 = 0$ by variation of constants: $f\uno = \eta\uno f\zero$ is a solution of $T_0 = Q f\uno + S f\zero = 0$ if $\eta\uno$ solves
\begin{equation} \label{eta uno y}
2i \mu_2 \, \eta\uno_y + g\zero = 0. 
\end{equation}
\eqref{eta uno y} has a periodic solution $\eta\uno$ if $g\zero \in Z_E$. 
The condition 
\[
\Pi_C(g\zero) = \frac{i}{4\mu_2}\, \Pi_C((a_{76}^E)^2) + \Pi_C(a_{98}) - c_0 = 0
\]
determines the constant $c_0$,
\begin{equation*} 
c_0 = \frac{i}{4\mu_2}\, \Pi_C((a_{76}^E)^2) + \Pi_C(a_{98}) \ \in \C.
\end{equation*} 
The condition 
\[
\Pi_T(g\zero) = \om (\Pi_T \ph)_\t 
+ \frac{i}{4\mu_2}\, \Pi_T((a_{76}^E)^2) + \Pi_T(a_{98}) = 0 
\]
determines the $Z_T$-component of $\ph$, 
\begin{equation*} 
(\Pi_T \ph)(\t) = 
- \frac{i}{4\mu_2 \om}\, (\pa_\t\inv \Pi_T (a_{76}^E)^2)(\t) 
- \frac{1}{\om} \, (\pa_\t\inv \Pi_T a_{98})(\t) \ \in Z_T.
\end{equation*}
So $g\zero \in Z_E$, \eqref{eta uno y} can be solved, and the $Z_E$-component of $\eta\uno$ is determined, 
\begin{equation} \label{eta uno E}
(\Pi_E \eta\uno)(\t,y)  = \frac{i}{2\mu_2}\, (\pa_y\inv g\zero)(\t,y) \ \in Z_E.
\end{equation} 

\medskip

\emph{Reduction of $T_{-1}$. --- } 
Since $f\uno = \eta\uno f\zero$, $S f\zero = f\zero g\zero$, by \eqref{reason} and the definition of $S$,
\[
S f\uno - c_{-1}\,f\zero
= \eta\uno S f\zero  
+ \eta\uno_y \big[ 2i\mu_2 f\zero_y + a_{76} f\zero \big] 
+ f\zero \big[ \om \eta\uno_\t + i \mu_2 \eta\uno_{yy} - c_{-1} \big]
\ = f\zero g\uno, 
\]
where
\begin{equation} \label{g uno}
g\uno := \eta\uno g\zero + \om \eta\uno_\t + i \mu_2 \eta\uno_{yy} + \mu_1 \eta\uno_y - c_{-1} .
\end{equation}
By variation of constants, $f\due = \eta\due f\zero$ is a solution of $T_{-1} = Q f\due + S f\uno - c_{-1}\,f\zero = 0$ if $\eta\due$ solves
\begin{equation} \label{eta due y}
2i \mu_2 \, \eta\due_y + g\uno = 0. 
\end{equation}
\eqref{eta due y} has a periodic solution $\eta\due$ if $g\uno \in Z_E$.
By \eqref{eta uno y}, $g\zero = -2i \mu_2 \,\eta\uno_y$, therefore 
\[
\eta\uno g\zero = -2i \mu_2 \, \eta\uno \,\eta\uno_y 
= -i \mu_2 \pa_y \{(\eta\uno)^2\} 
\ \in Z_E.
\]
As a consequence, the condition $g\uno \in Z_E$ determines
\begin{equation} \label{eta uno T}
\Pi_T (\eta\uno) = 0, \quad 
c_{-1} = 0.
\end{equation}
Thus \eqref{eta due y} can be solved, and the $Z_E$-component of $\eta\due$ is determined,
\begin{equation} \label{eta due E}
(\Pi_E \eta\due)(\t,y) = \frac{i}{2\mu_2}\, (\pa_y\inv g\uno)(\t,y).
\end{equation}

\medskip

\emph{Reduction of $T_{-2}$. --- } 
Since $c_{-1} = 0$, $T_{-2} = Q f\tre + S f\due - c_{-2}\,f\zero$.
By the same calculations as above,
\[
S f\due - c_{-2}\,f\zero
= \eta\due S f\zero  
+ \eta\due_y \big[ 2i\mu_2 f\zero_y + a_{76} f\zero \big] 
+ f\zero \big[ \om \eta\due_\t + i \mu_2 \eta\due_{yy} - c_{-2} \big]
\ = f\zero g\due, 
\]
where
\begin{equation} \label{g due}
g\due := \eta\due g\zero + \om \eta\due_\t + i \mu_2 \eta\due_{yy} 
+ \mu_1 \eta\due_y - c_{-2} .
\end{equation}
By variation of constants, $f\tre = \eta\tre f\zero$ is a solution of $T_{-2} = Q f\tre + S f\due - c_{-2}\,f\zero = 0$ if $\eta\tre$ solves
\begin{equation} \label{eta tre y}
2i \mu_2 \, \eta\tre_y + g\due = 0. 
\end{equation}
\eqref{eta tre y} has a periodic solution $\eta\tre$ if $g\due \in Z_E$. 
Both $(\Pi_T \eta\due) g\zero$ and $(\Pi_C \eta\due) g\zero$ belongs to $Z_E$ because $g\zero \in Z_E$. Hence 
\[
\Pi_T (\eta\due g\zero) 
= \Pi_T [ (\Pi_C \eta\due) g\zero + (\Pi_T \eta\due) g\zero + (\Pi_E \eta\due) g\zero] 
= \Pi_T [ (\Pi_E \eta\due) g\zero],
\]
and the same for $\Pi_C(\eta\due g\zero)$. 
$\Pi_E \eta\due$ is given by \eqref{eta due E}. 
The condition $\Pi_T g\due = 0$ determines 
\begin{equation} \label{eta due T}
\Pi_T \eta\due = - \frac{1}{\om}\, \pa_\t\inv \Pi_T [ (\Pi_E \eta\due) g\zero],
\end{equation}
the condition $\Pi_C g\due = 0$ determines 
\[
c_{-2} = \Pi_C[ (\Pi_E \eta\due) g\zero].
\]
Thus $g\due \in Z_E$, \eqref{eta tre y} can be solved, and the $Z_E$-component of $\eta\tre$ is determined, 
\begin{equation}  \label{eta tre E}
(\Pi_E \eta\tre)(\t,y) = \frac{i}{2\mu_2} \, (\pa_y\inv g\due)(\t,y).
\end{equation}

The only terms that have not been determined by the four equations $T_1 = 0, \ldots, T_{-2}=0$ are $\Pi_C(\ph)$, $\Pi_C(\eta\uno)$,  $\Pi_C(\eta\due)$, $\Pi_C(\eta\tre)$, and $\Pi_T(\eta\tre)$. 
Fix all of them to be $0$. 
Split real and imaginary part, 
\begin{align}
\label{formula Re ph}
& \Re(\ph) = \frac{1}{2\mu_2 \om}\,\pa_\t\inv \Pi_T[(\Pi_E a_7) a_6] 
- \frac{1}{\om} \,\pa_\t\inv \Pi_T(a_9)  
-\frac{1}{2\mu_2}\, (\pa_y\inv a_6), 
\\
\label{formula Im ph}
& \Im(\ph) = -\frac{1}{4\mu_2\om}\, \pa_\t\inv \Pi_T[(\Pi_E a_7)^2 - (a_6)^2]
- \frac{1}{\om} \,\pa_\t\inv \Pi_T(a_8)  
+ \frac{1}{2\mu_2}\, (\pa_y\inv \Pi_E a_7),
\\
\label{formula alfa beta zero}
& \a\zero = e^{\Re(\ph)}\,\cos(\Im(\ph)), \quad 
\b\zero = e^{\Re(\ph)}\,\sin(\Im(\ph)).
\end{align}
By \eqref{a 6789 parity},
\begin{equation*} 
\Re(\ph) \in X, \quad \Im(\ph) \in Y, \quad 
\a\zero \in X, \quad \b\zero \in Y.
\end{equation*}
As a consequence, $g\zero, \eta\uno, g\due, \eta\tre \in Y+iX$, 
$g\uno, \eta\due \in X+iY$, 
and  
\begin{equation*} 
\a\uno \in Y, \quad \b\uno \in X, \quad 
\a\due \in X, \quad \b\due \in Y, \quad
\a\tre \in Y, \quad \b\tre \in X.
\end{equation*}
Hence $\Phi$ preserves the parity, namely $\Phi$ maps $X \to X$ and $Y \to Y$.

By \eqref{a 6789 parity}, $(\Pi_E a_7) a_6 \in Y$, $a_9 \in Y$, therefore 
\begin{equation} \label{nu nu' 0}
\nu_0' = \Re(c_0) = 0,
\quad 
\nu_0 = \Im(c_0) = \frac{1}{4\mu_2} \, \Pi_C[(\Pi_E a_7)^2 - a_6^2] + \Pi_C(a_8). 
\end{equation}
$\nu_{-1} = \nu_{-1}' = 0$, and 
\begin{equation} \label{nu nu' -2}
\nu_{-2}' = \Re(c_{-2}) = 0,
\quad 
\nu_{-2} = \Im(c_{-2}) = \Im \{ \Pi_C[ (\Pi_E \eta\due) g\zero] \}.
\end{equation}
Put 
\[
\mu_0 := \nu_0, \quad \mu_{-2} := \nu_{-2}.
\]
Since $T_1,T_0, T_{-1},T_{-2}$ vanish, \eqref{commu pseudo} becomes $\tilde \mL_3 \tilde \Phi - \tilde \Phi \tilde \mD = \pp \mR_4 \pp$, and \eqref{mL4} holds with 
\begin{equation} \label{DR}
\mL_4 = \mD + \mR, \quad 
\mD = \om \pa_\t + \mu_2 \mH \pa_{yy} + \mu_1 \pa_y + \mu_0 \mH 
+ \mu_{-2} \mH \pa_{y}^{-2},
\quad
\mR := \tilde\Phi\inv \pp \mR_4 .
\end{equation}
If $\tilde\Phi$ is invertible, we have transformed $\tilde\mL$ into 
$\tilde\mL_4$, namely
\begin{equation} \label{catenella}
\tilde\mL = \tilde\Psi \tilde\mM \tilde\Phi \tilde\mL_4 \tilde\Phi\inv \tilde\Psi\inv,
\quad
\tilde\mL_4 
= \tilde\Phi\inv \tilde\mM\inv \tilde\Psi\inv \tilde\mL \tilde\Psi \tilde\Phi.
\end{equation}
From the formulae above, $\mu_0, \mu_{-2}, \a\k, \b\k$ are $C^1$ functions of $(u,\e)$, and the following estimates hold.

\begin{proposition} \label{prop:Phi mR}
Let $K > 0$. There exists $\e_0 \in (0,1)$, depending on $K$, 
such that, if $\e \in (0,\e_0)$, 
$\| u \|_{19} \leq K$, 
and $\|u\|_4,\e_0$ satisfy \eqref{pallata}, 
then all the following inequalities hold. 
\begin{align} 
\label{stima mu0}
|\mu_0| 
& \leq \e^3 C(K), 
&
|\pa_u \mu_0[h]| 
& \leq \e^4 C(K) \|h\|_5, 
&
|\pa_\e \mu_0| 
& \leq \e^2 C(K),
\\
\label{stima mu-2}
| \mu_{-2}| 
& \leq \e^4 C(K), 
&
| \pa_u \mu_{-2}[h]| 
& \leq \e^6 C(K) \|h\|_{12}, 
&
| \pa_\e \mu_{-2}| 
&\leq \e^3 C(K).
\end{align}
The operator $\tilde\Phi : Z_0 \to Z_0$ is invertible, and maps $X_0 \to X_0$ and $Y \to Y$. 
$\tilde\Phi, \tilde\Phi\inv$ satisfy
\begin{align} 
\label{stima Phi-I}
\| (\tilde\Phi - I) f \|_s + \| (\tilde\Phi\inv - I) f \|_s  
& \leq \e^2 C(s,K) (\| f \|_s + \|u\|_{s+12} \| f \|_2)
\quad \forall f \in Z_0,
\end{align}
for all $2 \leq s \leq r-7$.
The derivatives of $\tilde\Phi f, \tilde\Phi\inv f$ with respect to $u$ in the direction $h$ and with respect to $\e$ satisfy
\begin{align} 
\label{stima Phi der u}
\| \pa_u (\tilde\Phi f)[h] \|_s + \| \pa_u (\tilde\Phi\inv f)[h] \|_s  
& \leq \e^4 C(s,K) (\| f \|_s \| h \|_{14} + \| f \|_2 \| h \|_{s+12} 
+ \|u\|_{s+12} \| f \|_2 \| h \|_{14}), 
\\
\label{stima Phi der epsilon}
\| \pa_\e \tilde\Phi f \|_s + \| \pa_\e \tilde\Phi\inv f \|_s  
& \leq \e C(s,K) (\| f \|_s + \|u\|_{s+12} \| f \|_2).
\end{align}
Moreover 
\begin{align} 
\label{Phi-I for bif 1}
\| \pa_\t (\tilde\Phi-I) f \|_s 
& \leq \e^2 C(s,K) \big( \|\pa_\t f \|_s + \| f \|_s 
+ \| u \|_{s+13}(\|\pa_\tau f \|_2 + \| f \|_2 ) \big), 
\\
\label{Phi-I for bif 2}
\| \pa_y^k (\tilde\Phi-I) f \|_s 
& \leq \e^2 C(s,K) \big( \|\pa_y^k f \|_s + \| f \|_s 
+ \| u \|_{s+14}(\|\pa_y^k f \|_2 + \| f \|_2 ) \big), 
\quad k=1,2,
\end{align}
for $2 \leq s \leq r-9$, for all $f \in Z_0$. 

The operators $\tilde\Psi \tilde\Phi$, 
$\tilde\Psi \tilde\mM \tilde\Phi$, 
$\tilde\Phi\inv \tilde\Psi\inv$,
$\tilde\Phi\inv \tilde\mM\inv \tilde\Psi\inv$
are all of the type 
$I+S$, where $S$ satisfies 
\begin{equation} \label{S property}
\| Sf \|_s \leq \e^2 C(s,K) (\| f \|_{s+1} + \| u \|_{s+12} \| f \|_2),
\quad 
2 \leq s \leq r-7.
\end{equation}

The rest $\mR$ satisfies
\begin{equation} \label{stima mR}
\|\mR \pa_y^m f \|_s 
\leq \e^2 C(s,K) (\| f \|_s + \| u \|_{s+17} \| f \|_2), 
\quad 0 \leq m \leq 3, 
\quad 2 \leq s \leq r-12.
\end{equation}
\end{proposition}

\begin{proof} The proof is in Section \ref{sec:appendix proofs}.
\end{proof}

\section{\large{Inversion of the transformed linearized operator}}
\label{sec:inversion}

In view of the Nash-Moser iteration, we invert $\tilde\mL_4 = \tilde\mD + \tilde\mR$ on a subspace of Fourier-truncated functions.
Let 
\[
Z_N := \Big\{ u = \sum_{|k|\leq N} u_k \, e_k \Big\} \subset Z, \quad 
k=(l,j) \in \Z^2, \quad 
|k| = |l|+|j|, \quad 
Z_{0N} := Z_0 \cap Z_N,
\]
with $N>0$ sufficiently large to have $\bar v \in Z_N$, namely 
$\mK \subseteq [-N,N]$, where $\mK$ is defined in Section \ref{sec:bifurcation} (see Proposition \ref{prop:bif}).
Let $\Pi_N, \Pi_N^\perp$ denote the orthogonal projections onto $Z_N$ and $Z_N^\perp$ respectively. 
Let 
\[
X_{0N} := X_0 \cap Z_N, \quad 
Y_N := Y \cap Z_N, \quad  
V_{0N} := V_0 \cap Z_N, \quad 
W_N := W \cap Z_N. 
\]
$\Pi_N \tilde\mL_4 \Pi_N$ 
maps $X_{0N} \to Y_N$ because $\tilde\mL_4 : X_0 \to Y$.
Since $Z_{0N} = V_{0N} \oplus W_N$, to prove that 
$\Pi_N \tilde\mL_4 \Pi_N : X_{0N} \to Y_N$ is invertible, we project on the subspaces $V_{0N}$ and $W_N$ (Lyapunov-Schmidt decomposition, like in Section \ref{sec:Lyapunov-Schmidt}): 
given $f \in Y_N$, 
\begin{equation} \label{sistema Lyapunov-Schmidt}
\Pi_N \tilde\mL_4 \Pi_N h = f \quad \iff \quad 
\bigg\{ 
\begin{array}{l} 
\!\! 
\Pi_{V_{0N}} \tilde\mL_4 \Pi_{V_{0N}} h 
+ \Pi_{V_{0N}} \tilde\mL_4 \Pi_{W_N} h 
= \Pi_{V_{0N}} f \vspace{3pt} \\
\!\! 
\Pi_{W_N} \tilde\mL_4 \Pi_{V_{0N}} h 
+ \Pi_{W_N} \tilde\mL_4 \Pi_{W_N} h 
= \Pi_{W_N} f.
\end{array}
\end{equation}
Since $\mD$ is diagonal, $\mD$  maps $V \to V$ and $W \to W$, therefore
\begin{equation} \label{bell'off-diag}
\Pi_{V} \tilde\mL_4 \Pi_{W} = \Pi_{V} \tilde\mR \Pi_{W}, \quad 
\Pi_{W} \tilde\mL_4 \Pi_{V} = \Pi_{W} \tilde\mR \Pi_{V}.
\end{equation}

\begin{lemma}[Inversion on $V_{0N}$]
\label{lemma:inversione su VN}
Let $K > 0$. There exists $\e_0 \in (0,1)$, depending on $K$, 
such that, if $\e \in (0,\e_0)$, 
$\| u \|_{19} \leq K$, 
and $\|u\|_4,\e_0$ satisfy \eqref{pallata}, 
then 
\[
\Pi_{V_{0N}} \tilde\mL_4 \Pi_{V_{0N}} : V_{0N} \cap X_0 \to V_{0N} \cap Y
\]
is invertible, with 
\begin{equation} \label{inv VLV}
\| (\Pi_{V_{0N}} \tilde\mL_4 \Pi_{V_{0N}})\inv h \|_s \leq 
\frac{C(s,K)}{\e^2} \, (\|h\|_{s-1} + \|u\|_{s+13}\, \|h\|_2), 
\quad 3 \leq s \leq r-8.
\end{equation} 
\end{lemma}

\begin{proof}
$\tilde\mL_4 = \tilde \Phi\inv \tilde \mL_3 \tilde \Phi$
(see \eqref{mL4}). 
Split $\mL_3 = L + \e^2 A + \e^3 B$, where 
\begin{gather*}
L = \pa_\t + \pa_{yy} \mH, \quad 
Ah = 3\pa_\t h + \pa_y (3\bar v^2 h),
\\
B = \e^{-3} \{ (\mu_2-1) \pa_{yy} \mH + a_6 \, \pa_y \mH 
+ (a_7 - \e^2 3\bar v^2) \, \pa_y + a_8\, \mH + (a_9 - \e^2(3\bar v^2)_y)\, + \mR_3 \}.
\end{gather*}
By \eqref{stima mu2},\eqref{a 6789 epsilon},\eqref{stima mR123}, 
\begin{equation} \label{stima B}
\| Bh \|_s \leq C(s,K) \big( \| h_{yy} \|_s + \|h_y\|_s + \|h\|_s 
+ \|u\|_{s+7} (\|h_y\|_0 + \|h\|_0)  \big), 
\quad 2 \leq s \leq r-3. 
\end{equation}
Let $S_i : Z_0 \to Z_0$, 
$S_1 := \e^{-2} (\tilde\Phi - I)$, 
$S_2 := \e^{-2} (\tilde\Phi\inv - I)$
(recall that $\pp = I$ on $Z_0$). 
Since $\Pi_V L = L \Pi_V = 0$, 
\begin{align}
\Pi_{V_{0N}} \tilde\mL_4 \Pi_{V_{0N}}  
& = \Pi_{V_{0N}} \tilde \Phi\inv \tilde \mL_3 \tilde \Phi \Pi_{V_{0N}}
= \Pi_{V_{0N}} (I + \e^2 S_2) \pp (L+\e^2 A + \e^3 B) \pp (I + \e^2 S_1) \Pi_{V_{0N}}
\notag 
\\ & 
= \e^2 \Pi_{V_{0N}} (A + \e B_1) \Pi_{V_{0N}},
\label{beef 0.5}
\end{align}
where 
\[
B_1 = \e S_2 \pp L \pp S_1 
+ \e S_2 \pp A 
+ \e A \pp S_1 
+ \e^3 S_2 \pp A \pp S_1 
+ \tilde\Phi\inv \pp B \pp \tilde\Phi.
\]
By Proposition \ref{prop:bif}, $\Pi_{V_{0N}} A \Pi_{V_{0N}} : V_{0N} \cap X_0 \to V_{0N} \cap Y$ is invertible, with 
\begin{equation} \label{beef 1}
\| (\Pi_{V_{0N}} A \Pi_{V_{0N}})\inv h \|_s \leq C \|h\|_{s-1} \quad \forall h \in V_{0N} \cap Y, \quad \forall s \geq 0,
\end{equation}
where $C>0$ depends only on the set $\mK$. 
By \eqref{stima Phi-I},\eqref{Phi-I for bif 1},\eqref{Phi-I for bif 2}, 
for $2 \leq s \leq r-9$, 
\begin{align*}
& \| S_1 h \|_s + \| S_2 h \|_s  
\leq C(s,K) (\| h \|_s + \|u\|_{s+12} \| h \|_2),
\\
& \| \pa_{\cdot} S_1 h \|_s 
\leq C(s,K) \big( \| \pa_{\cdot} h \|_s + \| h \|_s 
+ \| u \|_{s+14}(\|\pa_{\cdot} h \|_2 + \| h \|_2 ) \big), 
\quad \pa_{\cdot} = \pa_\t, \pa_y, \pa_{yy},
\end{align*}
for all $h \in Z_0$. Then, since $L = \pa_\t + \mH \pa_y^2$, 
$Ah = 3\pa_\t h +  3\bar v^2 \pa_y h + (3 \bar v^2)_y h$, 
and by \eqref{stima B}, 
\begin{equation} \label{beef 2}
\| \Pi_{V_{0N}} B_1 \Pi_{V_{0N}} h \|_s 
\leq C(s,K) (\|h\|_{s+1} + \|u\|_{s+14}\, \|h\|_3), 
\quad 2 \leq s \leq r-9,
\end{equation} 
because $\|\pa_y^2 h \|_s = \| \mH \pa_y^2 h \|_s = \|\pa_\t h\|_s \leq \| h \|_{s+1}$ for all $h \in V$. 
Thus, by \eqref{beef 1}, \eqref{beef 2},
\[
\| (\Pi_{V_{0N}} B_1 \Pi_{V_{0N}}) (\Pi_{V_{0N}} A \Pi_{V_{0N}})\inv h \|_s 
\leq C(s,K) (\|h\|_{s} + \|u\|_{s+14}\, \|h\|_2), 
\quad 
2 \leq s \leq r-9,
\]
for all $h \in V_{0N} \cap Y$. Since $B_1$ maps $X$ into $Y$, 
$B_2 := (\Pi_{V_{0N}} B_1 \Pi_{V_{0N}}) (\Pi_{V_{0N}} A \Pi_{V_{0N}})\inv$ maps  $Y$ into $Y$.
By standard Neumann series with tame estimates 
(see Lemma \ref{lemma:tame estimates for operators}), 
$I + \e B_2$ 
is invertible as an operator of $V_{0N} \cap Y$ onto itself, with 
\begin{equation} \label{beef 3}
\| ( I + \e B_2)\inv h \|_s 
\leq 
C(s,K) (\|h\|_{s} + \|u\|_{s+14}\, \|h\|_2), 
\quad 
2 \leq s \leq r-9,
\end{equation} 
provided that $\e C(K) < 1/2$, for some $C(K)>0$ depending on $K, K_{g,r},\|\bar v\|_{19}$. 
By \eqref{beef 1} and \eqref{beef 3}, 
$\Pi_{V_{0N}} (A + \e B_1) \Pi_{V_{0N}}$  
$= (I + \e B_2)\inv (\Pi_{V_{0N}} A \Pi_{V_{0N}})$ 
$: X_0 \cap V_{0N} \to Y \cap V_{0N}$ is invertible, with
\begin{equation*} 
\| \{ \Pi_{V_{0N}} (A + \e B_1) \Pi_{V_{0N}}\}\inv h \|_s 
\leq 
C(s,K) (\|h\|_{s-1} + \|u\|_{s+13}\, \|h\|_2), 
\quad 
3 \leq s \leq r-8.
\end{equation*} 
By \eqref{beef 0.5} the thesis is proved.
\end{proof}

By Lemma \ref{lemma:inversione su VN}, the $V_{0N}$-equation of system  \eqref{sistema Lyapunov-Schmidt} can be solved for $\Pi_{V_{0N}} h$,
\begin{equation} \label{hV}
\Pi_{V_{0N}} h = (\Pi_{V_{0N}} \tilde\mL_4 \Pi_{V_{0N}})\inv 
[ \Pi_{V_{0N}} f - \Pi_{V_{0N}} \tilde\mL_4 \Pi_{W_N} h].
\end{equation} 
Substituting $\Pi_{V_{0N}} h$, and using \eqref{bell'off-diag}, the $W_N$-equation of system \eqref{sistema Lyapunov-Schmidt} becomes
\begin{equation} \label{WN-equation}
\mA (\Pi_{W_N} h) = f_1, 
\end{equation}
where 
\begin{align} 
\label{def mA}
\mA & := \Pi_{W_N} \tilde\mL_4 \Pi_{W_N} - 
(\Pi_{W_N} \tilde\mR \Pi_{V_{0N}}) 
(\Pi_{V_{0N}} \tilde\mL_4 \Pi_{V_{0N}})\inv 
(\Pi_{V_{0N}} \tilde\mR \Pi_{W_N}), 
\\
\label{def f1}
f_1 & := \Pi_{W_N} f - (\Pi_{W_N} \tilde\mR \Pi_{V_{0N}}) 
(\Pi_{V_{0N}} \tilde\mL_4 \Pi_{V_{0N}})\inv \Pi_{V_{0N}} f.
\end{align}
$\tilde\mL_4 = \mD + \tilde\mR$, 
where 
$\mD = \om \pa_\t + \mu_2 \mH \pa_{yy} + \mu_1 \pa_y + \mu_0 \mH + \mu_{-2} \mH \pa_y^{-2}$, which is \eqref{DR}. 
In the basis $\{ e^{i(l\t+jy)} \}_{l,j}$, 
$\mD$ is diagonal with eigenvalues 
\begin{equation}  \label{eigenvalues}
\lm_{l,j} = \lm_{l,j}(u,\e) 
= i \big( \om l + \mu_2 j |j| + \mu_1 j - \mu_0 \,\sgn(j) - \mu_{-2}\, \sgn(j)(ij)^{-2} \, \big),
\end{equation}
where $\om = 1+3\e^2$ and $\mu_i(u,\e)$ are $C^1$ functions of $(u,\e)$. By \eqref{stima mu2},
\eqref{stima mu1},
\eqref{stima mu0},
\eqref{stima mu-2}, 
\begin{equation} \label{i coefficienti costanti}
|\om-1| + |\mu_2-1| + |\mu_1| + |\mu_0| + |\mu_{-2}| < 1/2
\end{equation} 
for $\e < \e_0$ sufficiently small. 
Remember the notation $\la j \ra = \max\{ 1,|j|\}$.

\begin{lemma}[Inversion on $W_N$]
\label{lemma:inversione su WN}
Let $K > 0$. There exists $\e_0 \in (0,1)$, depending on $K$, with the following property. 
Let $\e \in (0,\e_0)$, 
$\| u \|_{19} \leq K$, 
and assume that $\|u\|_4,\e_0$ satisfy \eqref{pallata}. 
Let 
\begin{equation}  \label{dioph base}
|\lm_{l,j}(u,\e)| > \frac{1}{2 \la j \ra^3} 
\quad \forall (l,j) \in \mW_N,
\end{equation}
where 
\[
\mW_N := \{(l,j) \in \mW : |j| \leq N\} 
= \{(l,j) \in \Z^2 : l+j|j| \neq 0, \ \ |j| \leq N\}.
\]
Then $\mA : X_0 \cap W_N \to Y \cap W_N$ is invertible, with 
\begin{equation} \label{inv mA}
\| \mA\inv h \|_s 
\leq C(s,K) (\|h\|_{s+3/2} + \| u \|_{s+16+3/2} \| h \|_2), 
\quad 3/2 \leq s \leq r-12-3/2.
\end{equation}
\end{lemma}

\begin{proof} 
Since $\tilde\mL_4 = \tilde\mD + \tilde\mR$, 
we have $\mA = \mD_{W_N} + \mR_{W_N}$, where 
\[
\mD_{W_N} := \Pi_{W_N} \mD \Pi_{W_N}, \quad  
\mR_{W_N} := \Pi_{W_N} \tilde\mR \Pi_{W_N} - 
(\Pi_{W_N} \tilde\mR \Pi_{V_{0N}}) 
(\Pi_{V_{0N}} \tilde\mL_4 \Pi_{V_{0N}})\inv 
(\Pi_{V_{0N}} \tilde\mR \Pi_{W_N}).
\]
Like $\mA$, also $\mD_{W_N}$ and $\mR_{W_N}$ map $X$ into $Y$. 
$\mD_{W_N} : W_N \to W_N$ is invertible because $\lm_{l,j} \neq 0$ for all $(l,j) \in \mW_N$. 
Let 
\[
\mU := \pa_{y}^3 + \Pi_T + \Pi_C, \quad 
\mU e^{i(l\t+jy)} = \mU_j \, e^{i(l\t+jy)}, \quad 
\mU_j = (ij)^3 \ \  \forall j \neq 0, \quad 
\mU_0 = 1.
\] 
$|\lm_{l,j}| |\mU_j| > 1/2$ for every $(l,j) \in \mW_N$ 
because $|\mU_j| = \la j \ra^3$.
As a consequence, 
\[
\| \mU\inv \mD_{W_N}\inv h \|_s \leq 2 \|h\|_s \quad \forall h \in W_N, \quad \forall s \geq 0.
\]
By \eqref{stima mR} and \eqref{inv VLV},
\[
\| \mR_{W_N} \mU h \|_s \leq 
\| \mR_{W_N} \pa_y^3 h \|_s + \| \mR_{W_N} (\Pi_T + \Pi_C) h \|_s 
\leq \e^2 C(s,K) (\|h\|_{s} + \| u \|_{s+16} \| h \|_2)
\]
for $3 \leq s \leq r-12$, whence 
\[
\| \mR_{W_N} \mD_{W_N}\inv h \|_s 
= \| (\mR_{W_N} \mU)(\mU\inv \mD_{W_N}\inv) h \|_s 
\leq \e^2 C(s,K) (\|h\|_{s} + \| u \|_{s+16} \| h \|_2), 
\quad 3 \leq s \leq r-12.
\]
For $s=3$, 
$\| \mR_{W_N} \mD_{W_N}\inv h \|_3 
\leq \e^2 C(K) \|h\|_3$. 
By Lemma \ref{lemma:tame estimates for operators}, 
$I+ \mR_{W_N} \mD_{W_N}\inv$ is invertible on $W_N$, with 
\[
\| (I+ \mR_{W_N} \mD_{W_N}\inv)\inv h \|_s 
\leq C(s,K) (\|h\|_{s} + \| u \|_{s+16} \| h \|_2), 
\quad 3 \leq s \leq r-12,
\]
if $\e^2 C(K) < 1/2$. Therefore $\mA = (I + \mR_{W_N} \mD_{W_N}\inv) \mD_{W_N}$ is also invertible. 
Now $\| \mD_{W_N}\inv h \|_s \leq C \| h \|_{s+3/2}$ 
because, for indices $(l,j) \in \mW$ such that $|\lm_{l,j}| < 1$, one has $|j|^2 \leq C|l|$ by the triangular inequality and \eqref{i coefficienti costanti}, 
so that $1/|\lm_{l,j}| \leq 2 \la j \ra^3 \leq C \la l \ra^{3/2}$. 
Hence \eqref{inv mA} follows.
\end{proof}

Remember the definition $P_\e := \e^2 \Pi_V + \Pi_W$.

\begin{lemma}[Inversion of $\Pi_N \tilde \mL_4 \Pi_N$]
\label{lemma:inversione mL4}
Assume the hypotheses of lemmata \ref{lemma:inversione su VN} and \ref{lemma:inversione su WN}.
Then for every $f \in Y_N$ there exists a unique $h \in X_{0N}$ such that $\Pi_N \tilde\mL_4 \Pi_N h = f$. The inverse operator $(\Pi_N \tilde\mL_4 \Pi_N)\inv$ maps $Y_N 
\to X_{0N}$, with
\begin{gather}
\label{inversa con eps-2}
\|(\Pi_N \tilde\mL_4 \Pi_N)\inv f \|_s 
\leq \e^{-2} C(s,K) ( \|f\|_{s+3/2} + \| u \|_{s+17+3/2} \| f \|_2),
\\
\label{inversa ottima}
\|(\Pi_N \tilde\mL_4 \Pi_N)\inv P_\e f \|_s 
+ \|P_\e (\Pi_N \tilde\mL_4 \Pi_N)\inv f \|_s 
\leq C(s,K) ( \|f\|_{s+3/2} + \| u \|_{s+17+3/2} \| f \|_2),
\end{gather}
$3/2 \leq s \leq r-12-3/2$.
\end{lemma}

\begin{proof} 
Use \eqref{sistema Lyapunov-Schmidt}, \eqref{hV}, \eqref{WN-equation}, 
\eqref{def mA}, \eqref{def f1}, \eqref{inv VLV} and \eqref{inv mA}.
\end{proof}

\begin{lemma}[Derivatives of $(\Pi_N \tilde\mL_4 \Pi_N)\inv$]
\label{lemma:mL4 der}
Let $K > 0$. There exists $\e_0 \in (0,1)$, depending on $K$, with the following property. 

Let $\e \in (0,\e_0)$, 
$\| u \|_{22} \leq K$, 
assume that $\|u\|_4,\e_0$ satisfy \eqref{pallata}, and that \eqref{dioph base} holds.
Then, for $2 \leq s \leq r-18$, 
\begin{align*}
\| \pa_u (\Pi_N \tilde\mL_4 \Pi_N)\inv [h] f \|_s 
& \leq \e^{-1} C(s,K)
\big( \|f\|_{s+6} \| h \|_{14} + \| f \|_8 ( \| h \|_{s+16} + \| u \|_{s+23} \| h \|_{14}) \big),
\\
\|\pa_\e (\Pi_N \tilde\mL_4 \Pi_N)\inv f \|_s 
& \leq \e^{-3} C(s,K) (\|f\|_{s+6} +  \| u \|_{s+23} \| f \|_8),
\\
\| \pa_u (\Pi_N \tilde\mL_4 \Pi_N)\inv [h] P_\e f \|_s 
& + \| P_\e \pa_u (\Pi_N \tilde\mL_4 \Pi_N)\inv [h] f \|_s 
\\ 
& \leq \e C(s,K) 
\big( \|f\|_{s+6} \| h \|_{14} + \| f \|_8 ( \| h \|_{s+16} + \| u \|_{s+23} \| h \|_{14}) \big),
\\
\| \{ \pa_\e (\Pi_N \tilde\mL_4 \Pi_N)\inv \} P_\e f \|_s 
& + \| P_\e \{ \pa_\e (\Pi_N \tilde\mL_4 \Pi_N)\inv \} f \|_s 
\leq \e^{-1} C(s,K) (\|f\|_{s+6} +  \| u \|_{s+23} \| f \|_8).
\end{align*}
\end{lemma}

\begin{proof}[Proof of Lemma \ref{lemma:mL4 der}] 
By Proposition \ref{prop:a12345}, for all $0 \leq s \leq r$,
\begin{align*}
\| \tilde\mL f \|_s 
& \leq C(s,K) ( \| f \|_{s+2} + \| u \|_{s+4} \| f \|_2),
\\
\| \pa_u \tilde\mL[h] f \|_s 
& \leq \e^3 C(s,K) \big( \| f \|_{s+2} \| h \|_4 + \| f \|_2 
(\| h \|_{s+4} + \| u \|_{s+4} \| h \|_4) \big),
\\
\| \pa_\e \tilde\mL f \|_s 
& \leq \e C(s,K) (\| f \|_{s+2} + \| u \|_{s+4} \| f \|_2).
\end{align*}
Hence, from formula \eqref{catenella}, using the estimates 
\eqref{stima Psi der u},
\eqref{stima Psi der epsilon},
\eqref{stima mM der u},
\eqref{stima mM der epsilon},
\eqref{stima Phi der u},
\eqref{stima Phi der epsilon}
for $\tilde\Phi,\tilde\Psi,\tilde\mM$ and their inverse, 
\begin{align*}
\| \tilde\mL_4 f \|_s 
& \leq C(s,K) ( \| f \|_{s+2} + \| u \|_{s+14} \| f \|_2),
\\
\| \pa_u \tilde\mL_4[h] f \|_s 
& \leq \e^3 C(s,K) \big( \| f \|_{s+3} \| h \|_{14} + \| f \|_5 
(\| h \|_{s+14} + \| u \|_{s+15} \| h \|_{14}) \big),
\\
\| \pa_\e \tilde\mL_4 f \|_s 
& \leq \e C(s,K) (\| f \|_{s+3} + \| u \|_{s+15} \| f \|_5),
\end{align*}
for $2 \leq s \leq r-10$.
The Lemma follows from formula \eqref{formula der lm inv} and Lemma \ref{lemma:inversione mL4}.
\end{proof}

\subsection{\normalsize{Further estimates}}

In this section we collect some tame estimates that will be used in the Nash-Moser iteration.

\begin{lemma}[Tame estimates for $F$] \label{lemma:tame per F}
$(i)$ There exists $\e_0 \in (0,1)$, depending only on $\| \bar v_1 \|_5$, such that 
\begin{gather}  \label{stima v2}
\e \| \bar v_1 \|_4 + \e^2 \| \bar v_2 \|_4 < \d_0, \quad 
\| \bar v_2(\e) \|_s \leq C(s), \quad  
\| \pa_\e \bar v_2(\e) \|_s \leq \e^{-1} C(s), 
\\
\label{F(u0)}
\| F(\bar v_2(\e),\e) \|_s \leq \e C(s), \quad 
\| \pa_\e \{ F(\bar v_2(\e),\e) \} \|_s \leq C(s),
\end{gather}
for every $\e \in (0,\e_0)$, $2 \leq s \leq r$.

$(ii)$ Assume that $\e_0, u,h$ satisfy $\e_0 \| \bar v_1 \|_4 + \e_0^2 (\| u \|_4 + \| h \|_4) < \d_0$ ($\d_0$ is the universal constant of \eqref{pallata}), and $\| u \|_4 + \| h \|_4 \leq K$. 
Let 
\begin{equation} \label{def Q}
Q(u,h,\e) := F(u+h,\e) - F(u,\e) - \pa_u F(u,\e)[h].
\end{equation}
Then, for $2 \leq s \leq r$, $\e \in (0,\e_0)$, 
\begin{equation} \label{stima Q generica}
\| Q(u,h,\e) \|_s \leq C(s,K) \| h \|_4 (\|h\|_{s+2} + \|u\|_{s+2} \|h\|_4). 
\end{equation}

$(iii)$ Assume that $\e_0 \| \bar v_1 \|_4 + \e_0^2 \| u \|_4 < \d_0$, namely \eqref{pallata}, and $\| u \|_4 \leq K$.
Then 
\begin{align} 
\label{tame F basic}
\| F(u,\e) \|_s 
& \leq C(s,K) (1 + \|u\|_{s+2}), 
\\
\| \pa_u F(u,\e)[h] \|_s 
& \leq C(s,K) (\| h \|_{s+2} + \| u \|_{s+2} \| h \|_4), 
\\
\| \pa_\e F(u,\e)[h] \|_s 
& \leq \e^{-1} C(s,K) (1 + \| u \|_{s+2}), 
\end{align}
for all $2 \leq s \leq r$, $\e \in (0,\e_0)$.
\end{lemma}

\begin{proof} In Section \ref{sec:appendix proofs}. 
\end{proof}

\begin{remark}
Estimate \eqref{stima Q generica} actually holds with an additional factor $\e$ on the right-hand side. However, this makes no essential difference in our iteration proof below.
\end{remark}

\begin{lemma} \label{lemma:h<F(u)}
Assume the hypotheses of Lemma \ref{lemma:mL4 der}.
Then 
\begin{equation} \label{h<F(u)}
\| \tilde\Psi \tilde\Phi 
(\Pi_N \tilde\mL_{4} \Pi_N)\inv \Pi_N \tilde\Phi\inv \tilde\mM\inv \tilde\Psi\inv P_\e f \|_s 
\leq C(s,K) (\| f \|_{s+5/2} + \| u \|_{s + 17 + 5/2} \| f \|_2)
\end{equation}
for $2 \leq s \leq r-12-3/2$.
\end{lemma}

\begin{proof}[Proof of Lemma \ref{lemma:h<F(u)}]
By \eqref{stima Psi} and \eqref{stima Phi-I}, the term on the left-hand side in \eqref{h<F(u)} is 
\[
\leq C(s,K) 
\big( \| (\Pi_N \tilde\mL_{4} \Pi_N)\inv \Pi_N \tilde\Phi\inv \tilde\mM\inv \tilde\Psi\inv P_\e f \|_s 
+ \| u \|_{s+12} \| (\Pi_N \tilde\mL_{4} \Pi_N)\inv \Pi_N \tilde\Phi\inv \tilde\mM\inv \tilde\Psi\inv P_\e f \|_2 \big)
\]
for $2 \leq s \leq r-7$. 
Write $\tilde\Phi\inv \tilde\mM\inv \tilde\Psi\inv$ as $I + S$, where 
$S$ satisfies \eqref{S property}.
Since $\Pi_N P_\e = P_\e \Pi_N$, 
\[
(\Pi_N \tilde\mL_{4} \Pi_N)\inv \Pi_N \tilde\Phi\inv \tilde\mM\inv \tilde\Psi\inv P_\e f 
= 
(\Pi_N \tilde\mL_{4} \Pi_N)\inv P_\e \Pi_N f 
+
(\Pi_N \tilde\mL_{4} \Pi_N)\inv \Pi_N S P_\e f,
\]
then use \eqref{inversa ottima} 
for $(\Pi_N \tilde\mL_{4} \Pi_N)\inv P_\e \Pi_N f$, 
and use \eqref{inversa con eps-2}, \eqref{S property} 
for $(\Pi_N \tilde\mL_{4} \Pi_N)\inv \Pi_N S P_\e f$.
\end{proof}

\section{\large{Nash-Moser iteration and Cantor set of parameters}}
\label{sec:iteration}

Let 
\begin{equation} \label{Nn}
\chi := 3/2, \quad 
\para > 0, \quad 
N_n := \exp( \para \chi^n ), \quad 
n \in \N,
\end{equation}
with $N_0 = \exp(\para)$ sufficiently large to have $\mK \subseteq [-N_0,N_0]$ ($\mK$ is defined in Section \ref{sec:bifurcation}). 
Consider the corresponding increasing sequence of finite-dimensional subspaces $Z_n := Z_{N_n}$, with respective projections $\Pi_n := \Pi_{N_n}$.
For all $s,\a \geq 0$, $\Pi_n$ enjoys the smoothing properties
\begin{align} 
 \| \Pi_n u \|_{s+\a} & \,\leq\, N_n^\a \| u \|_s
\quad \forall u \in H^s,
\label{S1} \\
 \| \Pi_n^\perp u \|_s  & \,\leq\, N_n^{-\a} \, \| u \|_{s+\a} 
\quad \forall u \in H^{s+\a},
\label{S2} 
\end{align}
where $\Pi_n^\perp = I - \Pi_n$. 
Note that \eqref{S1}, \eqref{S2} hold even if $N_n > 0$ is not an integer number.
 
In the previous sections we have proved the transformation
\begin{equation} \label{conj completa}
F'(u,\e) = P_\e\inv \mL(u,\e) = P_\e\inv \tilde\mL(u,\e) 
= P_\e\inv \tilde\Psi \tilde\mM \tilde\Phi \tilde\mL_4 \tilde\Phi\inv \tilde\Psi\inv
\end{equation}
where $\tilde\Psi, \tilde\mM, \tilde\Phi, \tilde\mL_4$ all depend on $(u,\e)$.
Following a suitable Nash-Moser scheme, we construct a sequence $(u_n) \subset C^\infty(\T^2)$ of $\e$-dependent trigonometric polynomials by setting $u_0 := \bar v_2$ as defined in Section \ref{sec:bifurcation}, $h_0 := 0$, and 
\begin{equation} \label{sequence un}
u_{n+1} := u_n + h_{n+1}, \quad 
h_{n+1} := - \Pi_{n+1} \tilde\Psi_n \tilde\Phi_n 
(\Pi_{n+1} \tilde\mL_{4,n} \Pi_{n+1})\inv \Pi_{n+1} \tilde\Phi_n\inv \tilde\mM_n\inv \tilde\Psi_n\inv P_\e F(u_n),
\end{equation}
provided that the inverse operator $\mI_n := (\Pi_{n+1} \tilde\mL_4(u_n) \Pi_{n+1})\inv$ is well defined on $Z_{n+1}$. 
The notation in \eqref{sequence un} means
\[
\tilde\mL_{4,n} := \tilde\mL_4(u_n) = \tilde\mL_4(u_n(\e),\e), \quad 
\Psi_n := \Psi(u_n) = \Psi(u_n(\e),\e), 
\]
and similarly for $\tilde\mM, \tilde\Phi$. 
Also, $\mL_{4,n} = \mD_n + \mR_n$. 
We omit to write explicitly the dependence on $\e$ only to shorten the notation.
At a first glance, \eqref{sequence un} could seem an unusual and excessively complicated Nash-Moser scheme. However, in some sense it is ``the most natural'' for the present problem, as the ``normal form'' for the linearized operator is given by $\mL_{4,n} = \mD_n + \mR_n$, 
therefore it is natural to impose Diophantine conditions on the eigenvalues of $\mD_n$ and to insert smoothing operators $\Pi_n$ before and after it.

With $h_{n+1}$ defined by \eqref{sequence un}, one has 
$h_{n+1} = - \Pi_{n+1} \tilde\Psi_n \tilde\Phi_n \mI_n \Pi_{n+1} c_n$, 
\begin{equation} \label{rn}
F(u_n) + F'(u_n) h_{n+1} 
= r_{n} 
:= P_\e\inv \tilde\Psi_n \tilde\mM_n \tilde\Phi_n 
\big\{ \Pi_{n+1}^\perp c_n - \Pi_{n+1}^\perp\, \tilde\mR_n \Pi_{n+1} \mI_n \Pi_{n+1} c_n  + \tilde\mL_{4,n} b_n \big\}
\end{equation}
where
\begin{equation*} 
c_n := \tilde\Phi_n\inv \tilde\mM_n\inv \tilde\Psi_n\inv P_\e F(u_n),
\qquad
b_n := \tilde\Phi_n\inv \tilde\Psi_n\inv \Pi_{n+1}^\perp \, \tilde\Psi_n \tilde\Phi_n \mI_n \Pi_{n+1} c_n.
\end{equation*}
\eqref{rn} follows directly from \eqref{sequence un}, 
and is proved in Section \ref{sec:appendix proofs}. 
Hence 
\begin{equation} \label{F=r+Q}
F(u_{n+1}) = r_n + Q(u_n,h_{n+1}),
\end{equation}
where $Q$ is defined in \eqref{def Q}.

By Lemma \ref{lemma:inversione mL4}, $\Pi_{n+1} \tilde\mL_4(u_n) \Pi_{n+1}$ is invertible if the eigenvalues $\lm_{l,j}(u_n,\e)$ of $\mD_n$ satisfy the Diophantine condition \eqref{dioph base} for $u=u_n$ and $N=N_{n+1}$. Let $\mW_n := \mW_{N_n}$. 
Define recursively the set of the ``good'' parameters $\e$, those for which \eqref{dioph base} holds: let 
$\mG_0 := (0,\e_0)$, and define 
\begin{equation}  \label{An+1}
\mG_{n+1} := 
\Big\{ \e \in \mG_n : |\lm_{l,j}(u_n,\e)| > \frac{1}{2 \la j \ra^3} 
\quad \forall (l,j) \in \mW_{n+1} \Big\}, 
\quad n \geq 0.
\end{equation}
$\mG_{n}$ is the set of the parameters $\e$ for which  $(u_k,h_k,A_k,\mG_k)$ can be defined recursively for $k=0,\ldots,n$. 
On the contrary, after constructing $(u_k,h_k,A_k,\mG_k)$ for $k \leq n$, 
\begin{equation*} 
\mB_{n+1} := \mG_{n} \setminus \mG_{n+1} 
\end{equation*}
is the set of the ``bad'' parameters $\e$ for which the Diophantine condition \eqref{dioph base} on the eigenvalues $\lm_{l,j}(u_n,\e)$ is violated on $|l|+|j| \leq N_{n+1}$, 
the inverse of $(\Pi_{n+1} \mL_4(u_n) \Pi_{n+1})$ is not well-defined, 
$h_{n+1}$ cannot be defined by \eqref{sequence un}, 
and the recursive construction stops.
Therefore at the $n$-th step we eliminate the bad set $\mB_{n+1}$, and restrict the parameter set to the subset $\mG_{n+1} \subseteq \mG_{n}$. 
For convenience, put $\mB_0 := \emptyset$.

\begin{proposition}[Nash-Moser induction and measure estimate for the  parameter set]
\label{prop:NM}
There exist universal constants $r_0, s_0 > 0$ and constants $C, C', c_0, \para, \parb, \e^*_0 >0$ depending only on $\bar v_1, K_{g,r_0}$ such that 
if $\mG_0 = (0,\e_0)$, $\e_0 \leq \e^*_0$, $r \geq r_0$, and $\para$ defines $N_n$ in \eqref{Nn}, 
then the following induction hold.

Let $(P_n)=\{(P_n)(i), (P_n)(ii)\}$, $n \geq 1$, be the following set of statements.

\begin{itemize} 

\item 
$(P_n)(i)$. 
$\mG_{n}$ is an open set. 
The Lebesgue measure of $\mB_{n}$ satisfies $|\mB_{n}| \leq \e_0^2 C b_n$, 
where the sequence $(b_n)$ satisfies $\sum_{n=0}^\infty b_n = C' < \infty$. 

\item 
$(P_n)(ii)$. 
For every $\e \in \mG_n$, $h_n(\e) \in Z_n$ is well-defined. 
$h_n : \mG_{n} \to Z_n$, $\e \mapsto h_n(\e)$ 
is of class $C^1$ as a function of $\e$, with
\begin{gather}  \label{stima hk}
\| h_n(\e) \|_{s_0} < \exp(- \parb \chi^n), 
\quad 
\| \partial_\e h_n(\e) \|_{s_0} \leq \e^{-1} \exp(- \parb \chi^n).
\end{gather}
\end{itemize}

$(P_1)$ holds. 
If $(P_n)$ holds, then, using \eqref{sequence un},\eqref{An+1} to define $h_{n+1}$ and $\mG_{n+1}$, $(P_{n+1})$ also holds.

As a consequence, the Cantor set $\mG_\infty := \bigcap_{n \geq 0} \mG_n \subset (0,\e_0)$ has Lebesgue measure 
\[
| \mG_\infty | \geq \e_0 (1 - \e_0 C).
\]
For every $\e \in \mG_\infty$, the sequence $(u_n(\e))$ converges in $H^{s_0}(\T^2)$ to a limit $u_\infty(\e)$, which solves 
\[
F(u_\infty(\e),\e) = 0.
\]
Moreover, $u_\infty(\e) \in H^s(\T^2)$ for every $s$ in the interval $s_0 \leq s < (r+c_0)/2$.  

If $g_i$, $i=0,1,2$ in \eqref{g1 g2},\eqref{g0} is of class $C^\infty$, then also $u_\infty(\e) \in C^\infty(\T^2)$. 

$s_0$, $r_0$ and $c_0$ can be explicitly calculated: $s_0 = 22$, $c_0 = 28$; for $r_0$ see \eqref{parametri 800} and below.
\end{proposition}

We split the proof of Proposition \ref{prop:NM} into two parts: 
the Nash-Moser sequence $(P_n)(ii)$ with its regularity in subsection \ref{subsec:NM}, then the measure estimate $(P_n)(i)$ for the parameter set in subsection \ref{subsec:measure}

\subsection{Proof of the Nash-Moser iteration} \label{subsec:NM}

\emph{First step}. Let us prove $(P_1)(ii)$. 
For $\e \in \mG_1$, \eqref{sequence un} defines $h_1 = h_1(\e)$.  
By \eqref{stima v2}, the condition \eqref{pallata} holds. 
By \eqref{stima v2}, if $22 \leq r$, then 
$\| \bar v_2(\e) \|_{22} \leq C$ for all $\e \in (0,\e_0)$, for some constant $C$. Take this constant $C$ as the ``$K$'' in all the lemmata of the previous sections, so that the assumption $K \geq \| u \|_{22}$ is satisfied for $u = u_0 = \bar v_2(\e)$, for all $\e \in (0,\e_0)$.  
In this way, to indicate the dependence on $K$ in all the constants $C(s,K)$ is redundant, and we simply write $C(s,K) = C(s)$. 
By \eqref{sequence un}, \eqref{h<F(u)}, \eqref{stima v2} and \eqref{F(u0)},
\[
\| h_{1} \|_s 
= \| \tilde\Psi_0 \tilde\Phi_0 \mI_0 \Pi_{1} c_0 \|_{s} 
\leq C(s) \big( \| F(u_0) \|_{s+5/2} 
+ \| u_0 \|_{s+17+5/2} \| F(u_0) \|_2 \big)
\leq \e C(s)
\]
if $s+17+5/2 \leq r$.
Hence the first inequality in $(P_1)(iii)$ holds if 
\begin{equation} \label{parametri 00}
\e_0 C(s) \leq \exp(- \parb \chi).
\end{equation}
$\pa_\e h_1$ is obtained by differentiating every term in formula \eqref{sequence un} with respect to $\e$ and applying the estimates for $\pa_\e \tilde\Psi$, $\pa_\e \tilde\Phi$, $\pa_\e \{ (\Pi_1 \tilde\mL_4(u_0(\e),\e) \Pi_1)\inv \}$, etc; using \eqref{stima v2} for $\pa_\e \bar v_2$, and \eqref{F(u0)} for $\pa_\e \{ F(\bar v_2(\e),\e) \}$, we get
\[
\| \pa_\e h_1(\e) \|_{s} \leq C(s)
\]
for $\e \in (0,\e_0)$, $s +17 + 5/2 \leq r$. 
Therefore the second inequality in $(P_1)(iii)$ holds if \eqref{parametri 00} holds (with a possibly different constant $C(s)$, as usual).

\emph{Inductive step}. 
Now assume that $(P_n)$ holds, $n \geq 1$, and prove $(P_{n+1})(ii)$. 
By \eqref{stima hk}, 
\begin{equation} \label{sommina un s}
\| u_n \|_{s} 
\leq 
\| u_0 \|_{s} + \sum_{k=1}^n \| h_k \|_{s} 
\leq 
\| \bar v_2 \|_{s} + C(\parb),
\quad 
C(\parb) := \sum_{k=1}^\infty \exp(- \parb \chi^k).
\end{equation}
Note that $C(\parb)$ is independent on $n$, it is decreasing as a function of $\parb$, and 
$C(\parb) \to 0$ as $\parb \to +\infty$. 
Hence, for $s \geq 22$, $\| u_n \|_{22} \leq \| \bar v_2 \|_{22} + C(\parb) 
\leq 2 \| \bar v_2 \|_{22} = C$ for all $\e \in (0,\e_0)$ 
if 
\begin{equation} \label{parametri 10}
\parb \geq C,
\end{equation}
for some $C>0$. 
As in the previous step, take this constant $C$ as the ``$K$'', and replace $C(s,K)$ with $C(s)$ in all the lemmata of the previous sections. 
Moreover, \eqref{pallata} is satisfied for $u=u_n$ if $\e_0$ is sufficiently small, independently on the parameters. 
Also, $\| u_n \|_s \leq C(s)$.

By \eqref{sequence un}, \eqref{S1} and \eqref{h<F(u)}, 
for $\a \geq 0$, $2 \leq s-\a \leq r-12-3/2$,  
\begin{align} \notag
\| h_{n+1} \|_s 
& \leq N_{n+1}^\a \| \tilde\Psi_n \tilde\Phi_n \mI_n \Pi_{n+1} c_n \|_{s-\a} 
\\
& \leq N_{n+1}^\a C(s-\a) ( \| F(u_n) \|_{s-\a+5/2} 
+ \| u_n \|_{s-\a+17+5/2} \| F(u_n) \|_2 ).
\label{radio 1}
\end{align}
Take $\a := 17 + 5/2$, and denote $s':= s - 17$. Since $s' \geq 2$, 
\[
\| h_{n+1} \|_s 
\leq \eqref{radio 1} 
\leq N_{n+1}^\a C(s)(\| F(u_n) \|_{s'} + \| u_n \|_{s} \| F(u_n) \|_2) 
\leq N_{n+1}^\a C(s) \| F(u_n) \|_{s'} 
\]
because $\| u_n \|_s \leq C(s)$ by \eqref{sommina un s}. 
By \eqref{F=r+Q}, $F(u_n) = r_{n-1} + Q(u_{n-1}, h_n)$. Therefore
\begin{gather} \label{12 12}
\| h_{n+1} \|_s \leq A_r + A_Q,  \qquad 
A_r := N_{n+1}^\a C(s) \| r_{n-1} \|_{s'}, 
\quad
A_Q := N_{n+1}^\a C(s) \| Q(u_{n-1}, h_n) \|_{s'}.
\end{gather}
By \eqref{rn}, $r_{n-1}$ is the sum of 3 terms, say (I)+(II)+(III). 
The first one is 
\[
\text{(I)} = P_\e\inv \tilde\Psi_{n-1} \tilde\mM_{n-1} \tilde\Phi_{n-1} \Pi_n^\perp 
\tilde\Phi_{n-1}\inv \tilde\mM_{n-1}\inv \tilde\Psi_{n-1}\inv P_\e F(u_{n-1}).
\]
Using \eqref{S property}, like in the proof of Lemma \ref{lemma:h<F(u)}, no negative power of $\e$ appears in the estimate of (I). Using \eqref{S2} to deal with $\Pi_n^\perp$, for $\b \geq 0$, 
$ 2 \leq s'+\b \leq r-8$, one has 
\[
\| \text{(I)} \|_{s'} 
\leq C(s+\b) N_n^{-\b} 
( \| F(u_{n-1}) \|_{s'+\b+2} +  \| u_{n-1} \|_{s'+\b+13} \| F(u_{n-1}) \|_{2} ).
\]
The same argument applies to (II) and (III), whence
\[
\| r_{n-1} \|_{s'} 
\leq C(s'+\b) N_n^{-\b} 
( \| F(u_{n-1}) \|_{s'+\b+8} +  \| u_{n-1} \|_{s'+\b+19} \| F(u_{n-1}) \|_{2} ),
\]
$ 2 \leq s'+\b \leq r-16$. Applying \eqref{tame F basic}, 
\begin{equation} \label{radio 2}
\| r_{n-1} \|_{s'} 
\leq C(s'+\b) N_n^{-\b} (1 + \| u_{n-1} \|_{s'+\b+19})
= C(s+\b) N_n^{-\b} (1 + \| u_{n-1} \|_{s+\b+2}).
\end{equation}
Now estimate the ``high norm'' $B_k := \| h_{k} \|_{s+\b+2}$. 
To each $k=0,\ldots,n$, apply \eqref{radio 1} with $s+\b+2$ instead of $s$, 
and use \eqref{tame F basic}: 
for $2 \leq (s+\b+2)-\a \leq r-12-3/2$,
\begin{align} \label{radio 3}
\notag 
\| h_{k+1} \|_{s+\b+2}  
& \leq N_{k+1}^{\a} C(s+\b+2-\a) ( \| F(u_k) \|_{s+\b+2-\a+5/2} 
+ \| u_k \|_{s+\b+2-\a+17+5/2} \| F(u_k) \|_2 )
\\
& \leq N_{k+1}^{\a} C(s+\b) ( 1 + \| u_k \|_{s+\b+2})
\end{align}
where, as above, $\a := 17+5/2$. 
For \eqref{stima v2}, $\| u_0 \|_{s+\b+2} \leq C(s+\b)$ if $s+\b+2 \leq r$. 
Then, by \eqref{radio 3}, 
$B_1 = \| h_{1} \|_{s+\b+2} \leq N_{1}^{\a} C(s+\b)$, and 
\begin{equation}
\label{Bk+1}
B_{k+1} 
\leq N_{k+1}^{\a} C(s+\b) \Big( 1 + \| u_0 \|_{s+\b+2} + \sum_{j=1}^k \| h_j \|_{s+\b+2} \Big)
\leq N_{k+1}^{\a} C(s+\b) \Big( 1 + \sum_{j=1}^k B_j \Big)
\end{equation}
for $1 \leq k \leq n$. 
By \eqref{Nn}, this implies that 
\begin{equation} \label{Bn exp}
\| h_{k} \|_{s+\beta+2} = B_{k} \leq \exp(\lmcioeparb \chi^{k}), 
\end{equation}
$k=1,\ldots,n+1$. 
For, by induction: \eqref{Bn exp} holds for $k=1$ if 
$C(s+\b) \exp[(\para \a - \lmcioeparb)\chi]  \leq 1$, 
namely if $(\lmcioeparb - \para \a)$ is larger than some constant depending on $(s+\b)$. 
Suppose that \eqref{Bn exp} holds for all $j \in [1,k]$, $k \geq 1$. 
For $\lmcioeparb \geq 1$, 
\[
1 + \sum_{j=1}^k \exp(\lmcioeparb \chi^{j}) \leq C \exp(\lmcioeparb \chi^{k}), 
\quad \forall k \in \N,
\]
for some universal constant $C$.  
Then, by \eqref{Bk+1}, \eqref{Bn exp} also holds for $k+1$ if 
$C(s+\b) \exp[\chi^k (\para \a \chi - \lmcioeparb \chi + \lmcioeparb)]$ $ \leq 1$,
namely if 
\begin{equation} \label{parametri 80}
\lmcioeparb - 3 \para \a \geq C(s+\b)
\end{equation} 
for some $C(s+\b)>0$, and \eqref{Bn exp} is proved. 
Thus $\| u_{n-1} \|_{s+\b+2} \leq C(s+\b) \exp(\lmcioeparb \chi^{n-1})$, and, by \eqref{radio 2},
\begin{equation*} 
\| r_{n-1} \|_{s'} \leq C(s+\b) \exp[ \chi^{n-1}(\lmcioeparb - \b \para \chi)], 
\quad 
A_r \leq C(s+\b) \exp[ \chi^{n-1} (\lmcioeparb + \a \para \chi^2 - \b \para \chi)].
\end{equation*}
As a consequence, $A_r \leq \frac{1}{2} \exp(-\parb \chi^{n+1})$ if
\begin{equation} \label{parametri 100}
\para (\b \chi - \a \chi^2) - \lmcioeparb (1 + \chi^2) \geq C(s+\b)
\end{equation}
for some $C(s+\b) > 0$. 

Estimate $A_Q$. 
Since $\| u_{n-1} \|_{s'+2} = \| u_{n-1} \|_{s-15} \leq C(s)$, by \eqref{stima Q generica} we have $A_Q \leq N_{n+1}^\a C(s) \| h_n \|_{s}^2$. 
This is $\leq \frac{1}{2} \exp(-\parb \chi^{n+1})$ if 
\begin{equation} \label{parametri 90}
\parb - 3 \a \para \geq C(s)
\end{equation} 
for some $C(s) > 0$. 
Now fix 
\begin{equation} \label{parametri 800}
\parb := (3\a +1) \para, \quad 
\b := [\a \chi^2 + (1+\chi^2)(3\a+1)] \chi\inv.
\end{equation} 
Since $\chi=3/2$ and $\a = 17+5/2$, $\b$ is a universal constant, 
and the constants $C(s+\b)$ can be written as $C(s)$. 
Fix $\para \geq C(s)$ sufficiently large to satisfy \eqref{parametri 80}, \eqref{parametri 100}, \eqref{parametri 90} and \eqref{parametri 10}.
Then fix $\e_0 \leq C(s)$ sufficiently small to satisfy \eqref{parametri 00}.  
All the above conditions on $s$ hold if 
\[
22 \leq s \leq r-2-\b.
\]
Hence the minimal value for $r$ is $r_0 := 24 + \b$. Put $s_0 := 22$. 
For $s = s_0 = 22$ and $r = r_0$, all the above constants that depend on $s$ and $K_{g,r}$ become constants depending only on $K_{g,r_0}$. 
With this choice of parameters, the first estimate of $(P_{n+1})(iii)$ is proved.

The second estimate of $(P_{n+1})(iii)$ can be proved by the same arguments. Observe that in every estimate for $\pa_\e$ there is an additional factor $1/\e$: indeed, terms like $\e^p$ or $P_\e$, after being differentiated, have one degree less as powers of $\e$. 
Terms like $F(u_n,\e)$, $\tilde\Psi(u_n,\e)$, \ldots, after being differentiated with respect to $\e$, contain also terms like $\pa_u F(u_n,\e)[\pa_\e u_n]$, $\pa_u \tilde\Psi(u_n,\e)[\pa_\e u_n]$, \ldots,  and the loss of one degree as a power of $\e$ comes from \eqref{stima hk}. 
The estimates for $\pa_u$ and $\pa_\e$ of all the terms are given in the previous sections (and remind formula \eqref{formula F} for $F(u,\e)$).

\medskip

For each $\e$ for which the sequence $(u_n(\e))$ can be constructed, 
by \eqref{stima hk} $u_n = u_0 + \sum_{k=1}^n h_k$ 
is a Cauchy sequence in $H^{s_0}(\T^2)$, 
therefore $u_n(\e)$ converges in $H^{s_0}$ to some limit $u_\infty(\e) \in H^{s_0}$ as $n \to \infty$.  
Since the map $H^{s_0} \to H^{s_0-2}$, $u \mapsto F(u,\e)$ is continuous, 
$\| F(u_n,\e) - F(u_\infty, \e)\|_{s_0-2} \to 0$. 
On the other hand, we have proved that 
\[
\| F(u_n,\e) \|_{s'} \leq \| r_{n-1} \|_{s'} +  \| Q(u_{n-1}, h_n) \|_{s'} 
= C(s_0) N_{n+1}^{-\a} (A_r + A_Q) 
\leq C(s_0) N_{n+1}^{-\a} \exp(-\parb \chi^{n+1}) \to 0 
\]
as $n \to \infty$, where $s' = s_0 - 17 = 5$. Thus $F(u_\infty,\e) = 0$.

\medskip

Now let $22 = s_0 < s_1 < s_2$, with $s_1 = \lm s_0 + (1-\lm) s_2$, and $\lm \in (1/2 , 1)$. 
Apply \eqref{radio 3} with $s_2$ instead of $s+\b+2$: 
for $s_2 - \a \leq r - 12 - 3/2$ we get 
\[
\| h_{k+1} \|_{s_2} \leq N_{k+1}^\a C(s_2) (1 + \| u_k \|_{s_2}) \quad \forall k \geq 0,
\]
for some constant $C(s_2)$ depending on $s_2$. 
For \eqref{stima v2}, $\| u_0 \|_{s_2} \leq C(s_2)$ if $s_2 \leq r$. 
Then the ``very high norms'' $B_k' := \| h_k \|_{s_2}$ satisfy
$B_1' = \| h_{1} \|_{s_2} \leq N_{1}^{\a} C(s_2)$, and 
\begin{equation*} 
B_{k+1}' \leq N_{k+1}^{\a} C(s_2) \Big( 1 + \sum_{j=1}^k B_j' \Big), \quad k \geq 1.
\end{equation*}
Therefore there is a constant $K(s_2)$ such that 
\begin{equation} \label{Bn exp s2}
\| h_{k} \|_{s_2} = B_{k}'  \leq  K(s_2) \exp(\lmcioeparb \chi^{k}), 
\quad 
k \geq 1.
\end{equation}
Let us prove \eqref{Bn exp s2}. 
Since $\parb - 3 \a \para > 0$, where $\para, \parb$ have been fixed above, 
the inductive step $(k \Rightarrow k+1)$ holds for all $k \geq k_0(s_2)$, for some $k_0(s_2)$ depending on $s_2$ which is sufficiently large. Note that the constant $K(s_2)$ have no role in the inductive step. 
Then choose $K(s_2) := \max \{\| h_{k} \|_{s_2} \exp(-\lmcioeparb \chi^{k}) : 1 \leq k \leq k_0(s_2)\}$, so that \eqref{Bn exp s2} holds for all $k \geq 1$. 
Now, by \eqref{interpolation GN}, \eqref{Bn exp s2} and \eqref{stima hk}, 
\[
\| h_k \|_{s_1} \leq 2 \| h_k \|_{s_0}^\lm  \| h_k \|_{s_2}^{1-\lm} 
\leq 2 K(s_2)^{1-\lm} \exp(- \lm \parb \chi^k) \exp((1-\lm) \parb \chi^k) 
= C(s_2,\lm) \exp((1-2\lm) \parb \chi^k),
\]
and the series $\sum_{k \geq 1} \exp((1-2\lm) \parb \chi^k)$ converges because $(1-2\lm) < 0$. 
This implies that $\| u_\infty \|_{s_1} \leq \| u_0 \|_{s_1} + \sum_{k \geq 1} \| h_k \|_{s_1} 
< \infty$.
Since $s_1 < (s_0 + s_2)/2$ and $s_2 < r-12-3/2 + \a$, $\a = 17 + 5/2$, this argument holds if 
\[
s_1 < \frac{r + 28}{2}\,.
\]
If $g_i$, $i=0,1,2$ that defines the nonlinearity $\mN$ is of class $C^\infty$, 
then there is no upper bound for $s_1$, and the argument applies for every $s_1 \geq s_0$, whence $u_\infty \in C^\infty$.

\subsection{Proof of the measure estimate} \label{subsec:measure}

$\mG_0 = (0,\e_0)$, $\mB_{0} = \emptyset$. Let us estimate $\mG_{n+1}, \mB_{n+1}$, $n \geq 0$. 

The set $\mG_{n+1}$ is defined by \eqref{An+1}. 
$u_n(\e)$ is a $C^1$ function of $\e$, and $\mu_k(u,\e)$, $k=2,1,0,-2$ is a $C^1$ function of $(u,\e)$. Therefore each eigenvalue $\lm_{l,j}(u_n(\e),\e)$ is $C^1$ in $\e$. 
$\mB_{n+1}$ is the union  
\begin{equation} \label{union}
\mB_{n+1} = \bigcup_{(l,j) \in \mW_{n+1}} \Om_{l,j}^{n},
\qquad 
\Om_{l,j}^{n} := \Big\{ \e \in \mG_n : |\lm_{l,j}(u_n,\e)| \leq  \frac{1}{2 \la j \ra^3} \Big\}.
\end{equation}
Write the eigenvalues $\lm_{l,j}(u_n(\e),\e)$ as 
\[
\lm_{l,j}(u_n(\e),\e) = i \om \big( l + p_j^n(\e) \big),
\]
\[
p_j^n(\e) := \frac{\mu_2(u_n(\e),\e)}{1+3\e^2}\,j|j| 
+ \frac{\mu_1(u_n(\e),\e)}{1+3\e^2}\,j 
+ \frac{- \mu_0(u_n(\e),\e)}{1+3\e^2}\,\sgn(j) 
+ \frac{\mu_{-2}(u_n(\e),\e)}{1+3\e^2}\,\frac{\sgn(j)}{j^2}
\]
(where we mean $\sgn(j) j^{-2} = 0$ for $j=0$).
Since $\om = 1+3\e^2 > 1$, $|\lm_{l,j}(u_n(\e),\e)| \geq |l + p_j^n(\e)|$, and 
\begin{equation} \label{subset Om}
\Om_{l,j}^n \subseteq \tilde{\Om}_{l,j}^n 
:= \Big\{ \e \in \mG_n : |l + p_j^n(\e)| \leq  \frac{1}{2 \la j \ra^3} \Big\} \qquad \forall (l,j) \in \mW_{n+1}.
\end{equation}
For $j=0$, $p_j^n(\e) = p_0^n(\e) = 0$, therefore $\tilde{\Om}_{l,0}^n = \emptyset$ for all $l \neq 0$. 
The pair $(l,j)=(0,0)$ does not belong to $\mW_{n+1}$, 
hence the case $j=0$ gives no contribution to the union \eqref{union}.
So let $j \neq 0$.
\[
\frac{\mu_2(u_n(\e),\e)}{1+3\e^2} \, = 1 - 3\e^2 + O(\e^3), \quad 
\frac{\mu_1(u_n(\e),\e)}{1+3\e^2} \, = 3b \e^2 + O(\e^3), \quad 
\frac{\mu_k(u_n(\e),\e)}{1+3\e^2}\, = O(\e^3), \quad k=0,-2,
\]
where $b := \Pi_C(\bar v_1^2)$, and the precise meaning of $O(\e^3)$ is given by  
\eqref{stima mu2}, \eqref{stima mu1}, \eqref{stima mu0}, \eqref{stima mu-2}.
Therefore 
\[
p_j^n(\e) = j|j| (1 + \e^2 r_j^n(\e)), \quad 
r_j^n(\e) := \frac{1}{\e^2} \Big( \frac{p_j^n(\e)}{j|j|}\,-1 \Big)
= -3 + \frac{3b}{|j|} + O(\e).
\]
$|r_j^n(\e)| \leq C$ for some $C>0$ independent of $j,n,\e$. 
Also, by Proposition \ref{prop:bif}, 
\[
| b-|j| | \geq \d |j|, \quad 
\Big| -3 + \frac{3b}{|j|} \Big| \geq 3\d  
\quad \forall j \in \N, \ j \neq 0.
\]
As a consequence, 
\[
2\d \, \leq |r_j^n(\e)| \leq C
\]
for $\e < \e_0$ sufficiently small to have 
$|r_j^n(\e) + 3 - 3b/|j|| \leq \d$.
Suppose that $\e \in \tilde{\Om}_{l,j}^n \neq \emptyset$. 
Then, by the triangular inequality, 
\begin{equation} \label{tria}
|l+j|j|| 
\leq |l + p_j^n(\e)| + |- p_j^n(\e) + j|j|\,|
\leq \frac{1}{2\la j \ra^3}\, + \e^2 |j|^2 |r_j^n(\e)| 
\leq \frac{1}{2}\, +  C \e^2|j|^2. 
\end{equation}
$|l+j|j|| \geq 1$ because $l+j|j|$ is a nonzero integer. 
Thus we have a ``cut-off'': if $\tilde{\Om}_{l,j}^n \neq \emptyset$, then $1 \leq 1/2 + C \e^2|j|^2$, and
\begin{equation} \label{cut-off}
C \leq \e |j| \leq \e_0 |j|,
\end{equation}
for some $C>0$. 
Moreover, by \eqref{tria}, $l$ belongs to the interval 
\begin{equation} \label{interval lj}
-j|j| - 1/2 - C \e_0^2 |j|^2 \leq l \leq -j|j| + 1/2 + C \e_0^2 |j|^2.
\end{equation}
As a consequence, for any fixed $j$ with $|j| \geq C/\e_0$, 
the number of integers $l$ such that $\tilde{\Om}_{l,j}^n \neq \emptyset$ does not exceed the number of integers $l$ in the interval \eqref{interval lj}, namely
\begin{equation} \label{number of l}
\sharp\{ l : \tilde{\Om}_{l,j}^n \neq \emptyset \} 
\, \leq \, 2(1/2 + C \e_0^2 |j|^2) + 1 
\, \leq \, C' \e_0^2 |j|^2
\end{equation}
because $2 \leq C \e_0^2 |j|^2$ by \eqref{cut-off} (and the number of integers in an interval $[a,b]$ is at most $(b-a+1)$).
By \eqref{subset Om}, \eqref{number of l} implies that $\mB_{n+1}$ is the union of a finite number of closed sets, hence $\mG_{n+1}$ is open.

From the chain rule, 
\eqref{stima mu2}, \eqref{stima mu1}, \eqref{stima mu0}, \eqref{stima mu-2},
and $\| \pa_\e u_n(\e) \|_{12} \leq \e\inv C$ (which follows from \eqref{stima hk}), 
\[
\pa_\e p_j^n(\e) 
= j|j| \e \Big( -6 + \frac{6b}{|j|} + O(\e) \Big).
\]
Hence, for any fixed $j$, the sign of $\pa_\e p_j^n(\e)$ is the sign of $j ( -1 + b/|j|)$, which is constant with respect to $\e$. 
By \eqref{cut-off}, 
\[
|\pa_\e p_j^n(\e)| = |j|^2 \e \Big| -6 + \frac{6b}{|j|} + O(\e) \Big| 
\geq |j|^2 \e \d \geq C|j|
\]
if $\e_0$ is sufficiently small.
So $p_j^n$ is strictly monotone as a function of $\e$, and, as a consequence, $\tilde{\Om}_{l,j}^n$ is an interval, say $[\e_1,\e_2]$. 
If $p_j^n$ is increasing, then  
\[
\frac{1}{|j|^3} 
\geq p_j^n(\e_2) - p_j^n(\e_1)
= \int_{\e_1}^{\e_2} \pa_\e p_j^n(\e)\,d\e 
\geq C|j| (\e_2-\e_1) 
= C |j| |\tilde{\Om}_{l,j}^n|,
\]
and analogous calculation if $p_j^n$ is decreasing. 
Thus 
\begin{equation} \label{measure Om ljn}
|\tilde{\Om}_{l,j}^n| 
\leq  \frac{C}{|j|^4}\,.
\end{equation}
Also, $|\Om_{l,j}^n| \leq |\tilde{\Om}_{l,j}^n|$ because 
$\Om_{l,j}^n \subseteq \tilde{\Om}_{l,j}^n$.

Now split the union \eqref{union} into two parts, the union over the ``old'' indices $(l,j) \in \mW_{n+1} \cap \mW_n = \mW_n$ and the one over the ``new'' indices $(l,j) \in \mW_{n+1} \setminus \mW_n$. 
By \eqref{number of l} and \eqref{measure Om ljn}, the Lebesgue measure of the union over the new indices is
\[
\Big| \bigcup_{\text{new}} 
 \Om_{l,j}^n \Big| 
\leq \sum_{\text{new}} 
 |\Om_{l,j}^n| 
\leq \sum_{N_n < |j| \leq N_{n+1} } \frac{C}{|j|^4}\,\e_0^2 |j|^2
= C \e_0^2 \sum_{N_n < |j| \leq N_{n+1} } \frac{1}{|j|^2} 
= C \e_0^2 \, c_{n+1}, 
\]
where 
\[
c_0 := \sum_{1 \leq |j| \leq N_0} \frac{1}{|j|^2}\,, \quad 
c_{n+1} := \sum_{N_n < |j| \leq N_{n+1} } \frac{1}{|j|^2}\,, \quad  
\text{and} \quad 
\sum_{n=0}^\infty c_n 
= \sum_{|j|=1}^\infty \frac{1}{|j|^2} 
= C < \infty.
\]
For old indices, let $\e \in \tilde{\Om}_{l,j}^n$, with $(l,j) \in \mW_n$. 
By the triangular inequality, $u_n = u_{n-1} + h_n$, and estimates \eqref{stima mu2}, \eqref{stima mu1}, \eqref{stima mu0}, \eqref{stima mu-2} for $\pa_u \mu_k(u,\e)$,  
\[
|l + p_j^{n-1}(\e)| 
\leq |l + p_j^{n}(\e)| + |p_j^{n}(\e) - p_j^{n-1}(\e)|
\leq \frac{1}{2|j|^3}\, + C \e^4 |j|^2 \|h_n(\e)\|_{12}.
\]
Since $\tilde{\Om}_{l,j}^n \subseteq \mG_n$, and $(l,j) \in \mW_n$, 
\[
\tilde{\Om}_{l,j}^n \subseteq \Big\{ \e \in \mG_n : 
\frac{1}{2|j|^3}\, < |l + p_j^{n-1}(\e)| 
\leq \frac{1}{2|j|^3}\, + C \e^4 |j|^2 \|h_n(\e)\|_{12} \Big\}.
\]
As above, $p_j^{n-1}$ is strictly monotone as a function of $\e$,  
$|\pa_\e p_j^{n-1}(\e)| \geq C|j|$, and 
$\| h_n(\e) \|_{12} \leq \exp(-\parb \chi^n)$ by \eqref{stima hk}.
Hence
\[
|\tilde{\Om}_{l,j}^n| 
\leq  C \e_0^4 |j|^2 \exp(-\parb \chi^n) \, \frac{1}{|j|}\,
\leq  C \e_0^4 N_n \exp(-\parb \chi^n)
\]
because $|j| \leq N_n$. 
By \eqref{number of l} and \eqref{Nn}, the Lebesgue measure of the union over the old indices is then
\[
\Big| \bigcup_{\text{old}} \Om_{l,j}^n \Big| 
\leq \sum_{\text{old}} |\Om_{l,j}^n| 
\leq C \e_0^4 \sum_{|j| \leq N_n} N_n^3 \exp(-\parb \chi^n) 
\leq C \e_0^4 N_n^4 \exp(-\parb \chi^n) 
=  C \e_0^4 \exp[\chi^n(-\parb + 4\para)].
\]
Since $\parb -4 \para > \para \geq 1$ by \eqref{parametri 800}, $\sum_{n=0}^\infty \exp[\chi^n(-\parb + 4\para)] = C < \infty$. 
We have proved that 
\[
|\mB_{n+1}| \leq C \e_0^2 b_{n+1}, 
\quad \sum_{n=0}^\infty b_n = C < \infty.
\]
Therefore $| \cup_{n \geq 1} \mB_n | \leq \e_0^2 C$, whence 
$|\mG_{\infty}| \geq \e_0(1 - \e_0 C)$.

\section{\large{Appendix A. Kernel properties}}
\label{Appendix A. Kernel properties}

\begin{proof}[Proof of Lemma \ref{lemma:prodotti in V}] 
1) Let $j_1,j_2$ be nonzero. 
$q_{j_1} q_{j_2} = q_{j_3} \in V$ for some $j_3 \in \Z$ 
if and only if
\[
j_1 + j_2 = j_3, \quad 
-j_1|j_1| -j_2|j_2| = -j_3|j_3| .
\]
Let $n_k := |j_k|$ and $j_k = \s_k n_k$, $\s_k \in \{1,-1\}$, $k=1,2$.
If $\s_1 = \s_2$, then 
\[
j_3 = j_1 + j_2 = \s_1 (n_1 + n_2), \quad 
j_3 |j_3| = j_1|j_1| + j_2|j_2| = \s_1 (n_1^2 + n_2^2),
\]
therefore $|j_3|^2 = (n_1 + n_2)^2 = (n_1^2 + n_2^2)$, 
and this is impossible because $n_1 n_2 > 0$. 
If $\s_1 = - \s_2$, then 
\[
j_3 = j_1 + j_2 = \s_1 (n_1 - n_2), \quad 
j_3 |j_3| = j_1|j_1| + j_2|j_2| = \s_1 (n_1^2 - n_2^2),
\]
whence 
$|n_2 - n_1| \big( n_1 + n_2 - |n_2 - n_1| \big) = 0$.
This holds only for $n_2 = n_1$. 

2) Let $j_1, j_2, j_3$ all nonzero. 
$q_{j_1} q_{j_2} q_{j_3} = q_{j_4} \in V$ for some $j_4 \in \Z$ 
if and only if 
\[
j_1 + j_2 + j_3 = j_4 , \quad 
-j_1|j_1| -j_2|j_2| - j_3|j_3| = -j_4|j_4| .
\]
Let $n_k := |j_k|$, $j_k = \s_k n_k$, $k=1,2,3,4$, with $\s_1,\s_2,\s_3 \in \{1,-1\}$ and $\s_4 \in \{1,0,-1\}$. 
If $\s_1 = \s_2 = \s_3$, then 
\[
- n_1^2 - n_2^2 - n_3^2  
+ (n_1 + n_2 + n_3)^2 = 0,
\]
which is impossible because $n_1,n_2,n_3>0$. 
If $\s_1, \s_2, \s_3$ are not all equal, say $\s_1 = \s_2 = - \s_3$, then 
\[ 
\s_4 n_4 = j_4 = j_1 + j_2 + j_3 = \s_1 (n_1 + n_2 - n_3),
\]
\[
\s_4 n_4^2 = j_4|j_4| = j_1|j_1| + j_2|j_2| + j_3|j_3| 
= \s_1 (n_1^2 + n_2^2 - n_3^2).
\]
If $j_4 = 0$, then 
\[
n_1 + n_2 = n_3, \quad 
n_1^2 + n_2^2 = n_3^2,
\]
which is impossible because $n_1 n_2 >0$. 
Thus $j_4 \neq 0$, $\s_4 \neq 0$. As a consequence, 
\[
n_1 + n_2 - n_3 = \s n_4 , \quad 
n_1^2 + n_2^2 - n_3^2 = \s n_4^2,
\quad \s := \s_1 \s_4 \in \{1,-1\}.
\]
If $\s = -1$, then 
\[
n_1 + n_2 + n_4 = n_3 , \quad 
n_1^2 + n_2^2 + n_4^2 = n_3^2,
\]
which is impossible, as already observed. Thus $\s = 1$ and 
\[
n_1 - n_3 = n_4 - n_2, \quad 
(n_1 - n_3)(n_1 + n_3) = (n_4 - n_2)(n_4 + n_2) .
\]
If $n_1 \neq n_3$, then the second equality implies $n_1 + n_3 = n_4 + n_2$. Therefore the sum of the two equalities gives
\[
n_1 = n_4, \quad n_3 = n_2,
\]
hence $j_2 + j_3 = 0$ because $\s_2 = - \s_3$. 
If, instead, $n_1 = n_3$, then also $n_2 = n_4$, and $j_1 + j_3 = 0$ because $\s_1 = - \s_3$.
\end{proof}

\section{\large{Appendix B. Tame estimates}}
\label{sec:Appendix B. Tame estimates}

In this Appendix we remind classical tame estimates for changes of variables, composition of functions and the Hilbert transform, in Sobolev class on the torus, which are used in the paper. 
For these classical estimates see also, for example: \cite{Ioo-Plo-Tol}, Appendix G; 
\cite{Hormander-geodesy}, Appendix; 
\cite{Berti-Bolle-Ck-Nodea}, section 2; 
\cite{Inci-Kappeler-Topalov}. 
Before that, remind standard Sobolev norms properties 
(Lemma \ref{lemma:standard Sobolev norms properties}) 
and tame estimates for operators 
(Lemma \ref{lemma:tame estimates for operators}).

\begin{lemma} \label{lemma:standard Sobolev norms properties}
 Let $d \in \N$, $d \geq 1$, and $s_0 > d/2$. 
There exists an increasing function $C(s)>0$, $s \geq s_0$, with the following properties. \\
$(i)$ Embedding. $\| u \|_{L^\infty} \leq C(s_0) \| u \|_{s_0}$ for all $u \in H^{s_0}(\T^d,\C)$.
\\
$(ii)$ Algebra. $\| uv \|_{s_0} \leq C(s_0) \| u \|_{s_0} \| v \|_{s_0}$ for all 
$u, v \in H^{s_0}(\T^d,\C)$.
\\
$(iii)$ Interpolation. For $0 \leq s_1 \leq s \leq s_2$, $s = \lm s_1 + (1-\lm) s_2$, 
\begin{equation} \label{interpolation GN}
\| u \|_{s} \leq 2 \| u \|_{s_1}^\lm \| u \|_{s_2}^{1-\lm} 
\quad \forall u \in H^{s_2}(\T^d,\C).
\end{equation}
For $0 \leq s_1 \leq \s_1 \leq \s_2 \leq s_2$, 
\begin{equation} \label{interpolation estremi}
\| u \|_{\s_1} \| u \|_{\s_2} \leq 4 \| u \|_{s_1} \| u \|_{s_2}
\quad \forall u \in H^{s_2}(\T^d,\C).
\end{equation}
\eqref{interpolation GN},\eqref{interpolation estremi} also hold with all $\| u \|_s$ replaced by $|u|_s$, $u\in W^{s,\infty}(\T^d)$, $s \in \N$.

\medskip

\noindent
$(iv)$ Asymmetric tame product. For $s \geq s_0$, 
\begin{equation} \label{asymmetric tame product}
\| uv \|_s \leq C(s) \|u\|_s \|v\|_{s_0} + C(s_0) \|u\|_{s_0} \| v \|_s
\quad \forall u,v \in H^s(\T^d).
\end{equation}
$(v)$ Mixed norms tame product. For $s \geq 0$, $s \in \N$,
\begin{equation} \label{mixed norms tame product}
\| uv \|_s \leq C(s) (\|u\|_s |v|_{0} + \|u\|_0 | v |_s)
\quad \forall u \in H^s(\T^d), \ v \in W^{s,\infty}(\T^d).
\end{equation}
\end{lemma} 

\begin{proof} $(iii)$: see \cite{Moser-Pisa-66}, page 269. 
$(iv)$: see the Appendix of \cite{Berti-Bolle-Procesi-AIHP-2010}. 
$(v)$: write $D^\a(uv) = \sum_{\b+\g=\a} (D^\b u)(D^\g v)$, use the elementary inequality $\|(D^\b u)(D^\g v)\|_0 \leq \|D^\b u\|_0 |D^\g v|_0$, then the interpolation $(iii)$.
\end{proof}

\begin{lemma}
\label{lemma:tame estimates for operators}
Let $0 \leq s_0 \leq s$, and $c_0, c_s > 0$.
Let $S$ be a closed linear subspace of $Z$ 
(for example, $S=Z_0$ or $S=Z_{0N} \cap Y$).
Let $T : S \cap H^{s_0} \to S \cap H^{s_0}$ be a linear operator. 

$(i)$ Tame Neumann series. Let $c_0 \leq 1/2$. Assume that 
\begin{equation} \label{Neumann tame}
\| (T-I) f \|_{s} \leq c_0 \| f \|_{s} + c_s \| f \|_{s_0}, \quad 
\| (T-I) f \|_{s_0} \leq c_0 \| f \|_{s_0}
\end{equation}
for all $f \in S \cap H^{s_0}$. 
Then $T : S \cap H^{s_0} \to S \cap H^{s_0}$ 
is invertible, with 
\begin{equation} \label{Neumann tame inv}
\| (T\inv -I) f \|_{s} 
\leq 2 c_0 \| f \|_{s} + 4 c_s \| f \|_{s_0}, \quad 
\| (T\inv -I) f \|_{s_0} \leq 2 c_0 \| f \|_{s_0}.
\end{equation}

$(ii)$ Tame derivative of the inverse with respect to a parameter. 
Let 
\begin{equation} \label{op tame inv}
\| T\inv f \|_{s} \leq c_0 \| f \|_{s} + c_s \| f \|_{s_0}, \quad 
\| T\inv f \|_{s_0} \leq c_0 \| f \|_{s_0}
\end{equation}
for all $f \in S \cap H^{s_0}$.
Assume that $T$ depends in a $C^1$ way on a parameter $\lm$ in a Banach space, and the derivative 
$(\pa_\lm T)[\hat \lm] f$ 
of $Tf$ with respect to $\lm$ in the direction $\hat \lm$ satisfies 
\begin{equation} \label{op tame der}
\| (\pa_\lm T)[\hat \lm] f \|_{s} \leq b_0 \| f \|_{s} + b_s \| f \|_{s_0}, \quad 
\| (\pa_\lm T)[\hat \lm] f \|_{s_0} \leq b_0 \| f \|_{s_0}
\end{equation}
for all $f \in S \cap H^{s_0}$,  for some constants $b_0, b_s >0$. 
Then $T\inv$ is also a $C^1$ function of $\lm$, 
\begin{gather} \label{formula der lm inv}
\pa_\lm T\inv[\hat \lm] = - T\inv (\pa_\lm T[\hat \lm])\, T\inv,
\\
\label{op tame inv der}
\| \pa_\lm T\inv[\hat \lm] f \|_{s} 
\leq (4c_0^2 b_0) \| f \|_{s} 
+ (16 c_0 b_0 c_s + 4 c_0^2 b_s) \| f \|_{s_0}, \quad 
\| \pa_\lm T\inv[\hat \lm] f \|_{s_0} 
\leq c_0^2 b_0 \| f \|_{s_0}.
\end{gather}
\end{lemma} 

\begin{proof} 
$(i)$. Let $A := I-T$. By induction, 
\[
\| A^n f \|_s \leq c_0^n \| f \|_s + c_s n c_0^{n-1} \| f \|_{s_0}, \quad
\| A^n f \|_{s_0} \leq c_0^n \| f \|_{s_0}, \quad 
n \geq 1,
\]
where $A^2 f$ means $A(Af)$ and so on.
Since $c_0 \leq 1/2$, 
\[ 
\sum_{n=1}^\infty \| A^n f \|_s 
\leq c_0 \Big( \sum_{n=0}^\infty c_0^n \Big) \| f \|_s 
+ c_s \Big( \sum_{n=1}^\infty n c_0^{n-1} \Big) \| f \|_{s_0} 
\leq 2 c_0 \| f \|_s + 4 c_s \| f \|_{s_0}.
\]
Hence, by Neumann series, $T$ is invertible, and 
$T\inv - I = \sum_{n=1}^\infty A^n$ satisfies \eqref{Neumann tame inv}.

$(ii)$ Formula \eqref{formula der lm inv} follows from differentiating the equality $T T\inv f = f$ with respect to the parameter $\lm$.
\eqref{op tame inv},\eqref{op tame der},\eqref{formula der lm inv} give \eqref{op tame inv der}.
\end{proof}

\begin{lemma}[Composition of functions]
\label{lemma:composition of functions, Moser}
$(i)$ Let $f(x,y)$ be defined for $y=(y_1,\ldots,y_m)$ in the ball 
$B_1 = \{ y \in \R^m : |y|^2 = \sum_{i=1}^m |y_i|^2 < 1\}$  
and all $x=(x_1,\ldots,x_d) \in\R^d$, and let $f$ be $2\p$ periodic in $x_1,\ldots,x_d$. 
Assume that $f$ has continuous derivatives up to order $r \geq 0$ which are bounded by $\| f \|_{C^r} < \infty$. 
Let $u \in H^r(\T^d,\R^m)$, with $u(x) \in B_1$ for all $x$. 
Let $\tilde f(u)(x) = f(x,u(x))$. 
Then 
\[
\| \tilde f(u) \|_r \leq C \|f\|_{C^r} (\|u\|_r + 1).
\]
The constant $C$ depends on $r,d,m$.

\medskip

$(ii)$ Let $f,\tilde f$ be like in ($i$), and assume that $\| \pa_y^\a f \|_{C^r} \leq K_r$ for all $|\a| \leq N+1$. 
Let $\tilde f^{(n)}(u)[h]^n$ denote the $n$-th Fr\'echet derivative of $\tilde f$ at $u$ in the direction $[h]^n = [h,\ldots,h]$. 
($\tilde f^{(n)}(u)(x)$ is simply the $n$-th Fr\'echet derivative of $f(x,y)$ with respect to the variable $y$, evaluated at the point $(x,y) = (x,u(x))$\,).
If $u,h \in H^r(\T^d,\R^m)$, with $u(x),u(x)+h(x) \in B_1$ for all $x$, then
\[
\Big\| \tilde f(u+h) - \sum_{n=0}^{N} \frac{1}{n!}\, 
\tilde f^{(n)}(u)[h]^n \Big\|_r 
\leq C K_r \, \| h \|_{L^\infty}^{N} ( \| h \|_r + \| h \|_{L^\infty} \| u \|_r).
\] 
$C$ depends on $r,d,m,N$.

\medskip

$(iii)$ Let $u \in H^{r+p}(\T^d,\R)$. Let $D^k u(x)$ be the list of all partial derivatives $\pa_x^\a u(x)$ of order $|\a|=k$. 
Let $\tilde f(u)(x) = f(x,u(x),Du(x),\ldots,D^p u(x))$, where $f$ is like in $(i)$ for a suitable $m$. 
Then
\[
\| \tilde f(u) \|_r
\leq C \| f \|_{C^r} (\|u\|_{r+p} + 1)
\]
provided $(u(x),Du(x),\ldots,D^p u(x)) \in B_1$ for all $x$.
$C$ depends on $r,d,p$. 

If, in addition, $\| \pa_y^\a f \|_{C^r} \leq K_r$ for all $|\a| \leq N+1$, then 
\begin{equation} \label{Taylor composition}
\Big\| \tilde f(u+h) - \sum_{n=0}^{N} \frac{1}{n!}\, 
\tilde f^{(n)}(u)[h]^n \Big\|_r 
\leq C K_r \, \| h \|_{W^{p,\infty}}^{N} 
( \| h \|_{r+p} + \| h \|_{W^{p,\infty}} \| u \|_{r+p}).
\end{equation}
$C$ depends on $r,d,p,N$.

\medskip

$(iv)$ The previous statements also hold when all the $L^2$-based Sobolev norms $\|u\|_r$ are replaced by the $L^\infty$-based Sobolev norms $|u|_r = \|u\|_{W^{r,\infty}} = \sum_{k \leq r} \|D^k u\|_{L^\infty}$. 
\end{lemma}

\begin{proof}
$(i)$. See \cite{Moser-Pisa-66}, section 2, pages 272--275. 
$(ii)$. Use Taylor's formula with integral rest and the inequality 
$\| \int_0^1 u(\lm,\cdot) \, d\lm \|_r^2 \leq \int_0^1 \| u(\lm,\cdot) \|_r^2 \, d\lm$, which holds for $u(\lm,x) \in H^r(\T^d_x)$, depending on the parameter $\lm$, by H\"older's inequality. 
As an alternative, see \cite{Rabinowitz-tesi-1967}, Lemma 7 in the Appendix, pages 202--203.
$(iii)$. Consider $\tilde u = (u,Du, \ldots,D^p u)$ and apply $(i),(ii)$. See also \cite{Moser-Pisa-66}, page 275. 
$(iv)$. See \cite{Hamilton}, Lemma 2.3.4, page 147 for $(i)$ in the $W^{r,\infty}$ case. 
$(ii),(iii)$ can be adapted with no difficulty (the $W^{r,\infty}$ norms satisfy the algebra and interpolation properties, which are the core of the proofs).
\end{proof}

$(iii)$ of Lemma \ref{lemma:composition of functions, Moser} is used for the nonlinearity $\mN(u)$. 
$(ii)$ is also used for $N=0$, $u=0$, mainly for $f(y)=e^y$, $f(y)=\cos(y)$, $f(y)=(1+y)^p$, $p \in \R$: 
\begin{equation} \label{tame veloce infty}
| f(h) - f(0) |_s \leq C | h |_s 
\quad \forall h \in W^{s,\infty}(\T^2,\R), \ \ |h|_0 < 1,
\end{equation}
where $C$ depends on $f$ and $s$. 

\medskip

The next lemma is also classical, see for example \cite{Hormander-geodesy}, Appendix, and \cite{Ioo-Plo-Tol}, Appendix G. 
However, in those papers it is stated slightly differently than in Lemma \ref{lemma:utile}, especially part $(i)$, therefore we prove it, adapting Lemma 2.3.6 on page 149 of \cite{Hamilton}.

\begin{lemma}[Change of variable]  \label{lemma:utile} 
Let $p:\R^d \to \R^d$ be a $2\p$-periodic function in $W^{m,\infty}$, 
$m \geq 1$, with $|Dp|_0 \leq 1/2$. Let $f(x) = x + p(x)$. Then:

$(i)$ $f$ is invertible, its inverse is $f\inv(y) = g(y) = y + q(y)$, where $q$ is periodic, $q \in W^{m,\infty}(\T^d,\R^d)$, and $|q|_m 
\leq C |p|_m$. More precisely,
\[
| q |_0 = | p |_0, \quad   
| Dq |_0 \leq 2 | Dp |_0 \leq 1, \quad   
| Dq |_{m-1} \leq C | Dp |_{m-1}.
\]
The constant $C$ depends on $d,m$. 

$(ii)$ If $u \in H^m(\T^d,\C)$, then $u\circ f(x) = u(x+p(x))$ is also in $H^m$, and, with the same $C$ as in $(i)$, 
\[
\| u \circ f \|_m \leq C (\|u\|_m + |Dp|_{m-1} \|u\|_1).
\]

$(iii)$ Part $(ii)$ also holds with $\| \ \|_k$ replaced by 
$| \ |_k$,
namely 
$| u \circ f |_m \leq C (|u|_m + |Dp|_{m-1} |u|_1)$.
\end{lemma}

\begin{proof} $(i)$. 
For every $y \in \R^d$, the map $G_y : \R^d \to \R^d$, $G_y(x) = y-p(x)$ is a contraction because $|Dp|_0 \leq 1/2$, therefore $G_y$ has a unique fixed point $x=G_y(x)$ in $\R^d$, and the inverse function $g = f\inv : \R^d \to \R^d$ is globally defined. Let $q(y) := g(y) - y$.

Since $p$ is periodic, $f(x+2\p m) = f(x) + 2\p m$ for all $m \in \Z^d$. 
Applying $g$ to this equality gives $x+2\p m = g(f(x) + 2\p m)$, namely $g(y) + 2\p m = g(y + 2\p m)$ where $y=f(x)$, and this means that $q$ is periodic.  
Hence $g$, like $f$, is also a bijection of $\T^d$ onto itself.

The identity $f(g(y))=y$ gives 
\begin{equation} \label{gamma p}
q(y) + p(y+q(y)) = 0, \quad 
q(x+p(x)) + p(x) = 0 \quad \forall x,y \in \R^d.
\end{equation}
\eqref{gamma p} implies that $|q|_0 = |p|_0$. 
By Neumann series, the matrix $Df(x) = I + Dp(x)$ is invertible for a.e. $x$, 
$(Df(x))\inv = \sum_{n=0}^\infty (-Dp(x))^n$, and $| (Df)\inv |_0 \leq 2$. 
Differentiang \eqref{gamma p},
\begin{equation} \label{D gamma}
Dq(y) = - \big[ Df(y+q(y)) \big]\inv Dp(y+q(y)) = \sum_{n=1}^\infty [-Dp(g(y))]^n ,
\end{equation}
whence $|Dq|_0 \leq 2 |Dp|_0 \leq 1$. 
Differentiating \eqref{D gamma}, 
\[ 
(D^2q)(y) = - \big[ (Df)(g(y)) \big]\inv (D^2p)(g(y))\,Dg(y)\,Dg(y),
\] 
and $|D^2 q|_0 \leq 8 |D^2 p|_0$. $(i)$ is proved for $m=1$ and $m=2$.

In general, by the ``chain rule'', the $m$-th Fr\'echet derivative of the composition of functions 
$u \circ v$ is 
\begin{equation} \label{Dm uv composition}
D^m(u\circ v)(x) = \sum_{k=1}^m \sum_{j_1+\ldots+j_k = m} C_{kj} 
(D^k u)(v(x)) \, [D^{j_1}v(x), \ldots , D^{j_k} v(x)],
\end{equation} 
where $j_1, \ldots, j_k \geq 1$, and $C_{kj}$ are constants depending on $k,j_1,\ldots,j_k$ 
(\cite{Hamilton}, page 147).
Apply \eqref{Dm uv composition} to $f \circ g$: since $f(g(y))=y$, $D^m(f\circ g)=0$ for all $m \geq 2$.
Separate $k=1$ from $k \geq 2$ in the sum \eqref{Dm uv composition} and solve for $D^m g$, 
\[
D^m g(y) = 
- Dg(y) \, \sum_{k=2}^m \, \sum_{j_1+\ldots+j_k = m} C_{kj} 
(D^k f)(g(x)) \, [D^{j_1}g(y), \ldots , D^{j_k}g(y)].
\]
$D^m g = D^m q$ and $D^k f = D^k p$ because $k,m \geq 2$. 
Since $k \geq 2$, it is $1 \leq j_i \leq m-1$ for all $i=1,\ldots,k$, because there are at least two $j_1,j_2$, each of them $\geq 1$, and $\sum j_i = m$.  
For $k=m$ one has $j_i = 1$ for all $i=1,\ldots,m$, and the corresponding term in the sum is estimated 
\[
| (D^m p) \circ g \, [Dg,\ldots,Dg] |_0
\leq | D^m p |_0 |Dg|_0^m
\leq C |Dp|_{m-1},
\]
because $|Dg|_0 = |I+Dq|_0 \leq 2$. 
For $2 \leq k \leq m-1$, at least one among $j_1, \ldots, j_k$ is $\geq 2$ (otherwise $k=m$). 
Let $\ell$ be the number of indices $j_i$ that are $\geq 2$, so that $1 \leq \ell \leq k$.
It remains to estimate
\begin{equation} \label{somma riorganizzata}
\sum_{k=2}^{m-1} \, \sum_{\ell = 1}^{k} \  \sum_{\s_1 + \ldots + \s_\ell = m-k+\ell} 
C_{k\ell\s} (D^k p)(g(y))\, [Dg(y)]^{k-\ell} [ D^{\s_1} q(y), \ldots, D^{\s_\ell} q(y)],
\end{equation}
where indices $j_i \geq 2$ have been renamed $\s_1, \ldots \s_\ell$, the number of indices $j_i = 1$ is $k-\ell$, and $D^{\s_i}g = D^{\s_i} q$ because $\s_i \geq 2$. 
Every factor $Dg$ in \eqref{somma riorganizzata} is estimated by $|Dg|_0 \leq 2$. 
For the remaining factors use the interpolation between $0$ and $m-2$, which is possible because 
$1 \leq \s_i-1 \leq m-2$, and use the formula $\s_1 + \ldots + \s_\ell = m-k+\ell$,
\begin{align*}
|(D^k p) \circ g \, (D^{\s_1} q) \ldots (D^{\s_\ell} q)|_0 
& \leq |D^{k-2} D^2 p|_0  |D^{\s_1-1} Dq|_0 \ldots |D^{\s_\ell-1} Dq|_0 
\\ 
& \leq C |D^2 p|_0^{\frac{m-2-(k-2)}{m-2}}  |D^2 p|_{m-2}^{\frac{k-2}{m-2}} 
\prod_{i=1}^\ell |Dq|_0^{\frac{m-2 -(\s_i-1)}{m-2}} |Dq|_{m-2}^{\frac{\s_i-1}{m-2}}
\\
& = C |Dq|_0^{\ell-1} (|D^2 p|_0 |Dq|_{m-2})^{1-\frac{k-2}{m-2}} 
(|D^2 p|_{m-2} |Dq|_0)^{\frac{k-2}{m-2}} 
\\ 
& \leq C|Dq|_0^{\ell-1} (|D^2 p|_0 |Dq|_{m-2} + |D^2 p|_{m-2} |Dq|_0) 
\\ 
& \leq C(|Dq|_{m-2} + |Dp|_{m-1}).
\end{align*}
Collecting all the terms in the sum, we have proved that 
\begin{equation}  \label{indu}
|D^m q|_0 \leq C (|Dp|_{m-1} + |Dq|_{m-2}).
\end{equation}
Now use the induction on $m$. We have already proved $(P_m)$ $|Dq|_{m-1} \leq C|Dp|_{m-1}$ for $m=2$. 
Assume that $(P_{m-1})$ holds. Then $(P_m)$ follows from \eqref{indu}.

\smallskip

$(iii)$ follows a similar argument, using formula \eqref{Dm uv composition} and interpolation for $W^{k,\infty}$ norms; see \cite{Hamilton}, Lemma 2.3.4, page 147.

\smallskip

$(ii)$ $\|u\circ f\|_0 \leq C \|u\|_0$, because,  
changing variable $x=g(y)$ in the integral,
\begin{equation} \label{integral trick}
\| u \circ f \|_0^2 = \int_{\T^d} |u(f(x))|^2\,dx 
= \int_{\T^d} |u(y)|^2\,|\det Dg(y)|\,dy 
\leq \| \det Dg \|_{L^\infty} \int_{\T^d} |u(y)|^2 dy 
\leq C \|u\|_0^2.
\end{equation}
The $m$-th derivative of $u \circ f$, $m \geq 1$, is given by formula \eqref{Dm uv composition}. 
The $L^2$ norm of a typical term of the sum is estimated by  
\[
\| D^k u(f(x)) \, [D^{j_1}f(x), \ldots , D^{j_k}f(x)] \|_0 
\leq \| (D^k u) \circ f \|_0 \| D^{j_1}f \|_{\linf} \ldots \| D^{j_k}f \|_\linf.
\]
$\| (D^k u) \circ f \|_0 \leq C \| D^k u \|_0 \leq C \|Du\|_{k-1}$ by \eqref{integral trick}. 
Use interpolation \eqref{interpolation GN} for $\| Du \|_{k-1}$ and interpolation with $W^{k,\infty}$ norms for all $D^{j_i-1} Df$ between $0$ and $m-1$, which is possible because 
$k-1,j_i-1$ are all in the interval $[0,m-1]$. (Remember that $Df$ is periodic, while $f$ is not). 
We get 
\[
\| D^k u \|_0 \| D^{j_1}f \|_{\linf} \ldots \| D^{j_k}f \|_\linf
\leq C \|Df\|_\linf^{k-1} (\| Du \|_{m-1} \|Df\|_\linf  + \| Du \|_0 \|Df\|_{W^{m-1,\infty}}). 
\]
Now $\|Df\|_{\linf} \leq 2$, and $\|Df\|_{W^{m-1,\infty}} \leq C(1 + \|Dp\|_{W^{m-1,\infty}})$.  
The sum gives the thesis. 
\end{proof}

The next lemma estimates the commutator of $\mH$ with multiplication operators and changes of variables that are used in the paper. 
See also \cite{Ioo-Plo-Tol}, Appendices H and I.

\begin{lemma}[Commutators of $\mH$] \label{lemma:commutators} 
1) Let $s,m_1,m_2 \in \N$, with $s \geq 2$, $m_1, m_2 \geq 0$, $m=m_1+m_2$. 
Let $f(t,x) \in H^{s+m}(\T^2,\C)$. 
Then $[f,\mH]u = f \mH u - \mH(fu)$ satisfies 
\[
\| \pa_x^{m_1} [f,\mH] \pa_x^{m_2} u \|_{s} 
\leq C(s) (\| u \|_s \|f\|_{m+2} + \| u \|_2 \|f\|_{m+s}).
\]
2) Let $a:\T \to \T$ a function, and $Au(t,x) = u(a(t),x)$. Then $[A,\mH]=0$.

\medskip

\noindent
3) There exists a universal constant $\d \in (0,1)$ with the following property. Let $s,m_1,m_2 \in \N$, $m=m_1+m_2$, 
$\b(t,x) \in W^{s+m+1,\infty}(\T^2,\R)$, with $|\b|_1 \leq \d$. Let $Bh(t,x) = h(t,x+\b(t,x))$, $h\in H^s(\T^2,\C)$. Then 
\[
\| \pa_x^{m_1} (B\inv \mH B - \mH) \pa_x^{m_2} h \|_s 
\leq C(s,m) ( | \b |_{m+1} \| h \|_s + | \b |_{s+m+1} \| h \|_0).
\]
\end{lemma}

\begin{proof} 
1) Let $u(t,x) = \sum_{k \in \Z} u_k(t)\, e^{ikx}$, $f(t,x) = \sum_{k \in \Z} f_k(t)\, e^{ikx}$, and 
\[
S = \{ (k,j) \in \Z^2 : \sgn(k) - \sgn(j) \neq 0 \}, \quad 
S(k) = \{ j \in \Z : (k,j) \in S \}.
\]
Since $\mH(e^{ikx}) = -i \,\sign(k) \, e^{ikx}$, 
\[
\pa_x^{m_1} [f,\mH] \pa_x^{m_2} u 
= \sum_{k,j\in\Z} f_{j-k}(t)\,u_k(t)\,\d(k,j)\, (ij)^{m_1} (ik)^{m_2} \, e^{ijx} 
= \sum_{(k,j)\in S} \text{(the same)},
\]
where $\d(k,j) := -i\,(\sign(k) - \sign(j))$. 
If $(k,j) \in S$, then 
\[
|k-j| = |k| + |j|, \qquad |j| \leq |j-k|, \quad |k| \leq |j-k|.
\]
Therefore $|j^{m_1} k^{m_2}| \leq |k-j|^m$. 
If $j,k$ are Fourier indices for the space and $n,l$ for the time, 
\begin{align*} 
\| \pa_x^{m_1} [f,\mH] \pa_x^{m_2}u \|_s^2 
& \leq \sum_{n,j} \Big( \sum_{l,k} |f_{(n-l,j-k)}| |j-k|^m |u_{(l,k)}| \Big)^2 \la(n,j)\ra^{2s} 
\leq \sum_{a \in \Z^2} \Big( \sum_{b \in \Z^2} |(\pa_x^m f)_{a-b}| |u_b| \Big)^2 \la a \ra^{2s}
\end{align*}
and this gives the usual tame estimate for the product $(\pa_x^m f) u$. 
The estimate holds with $\| \ \|_{s_0}$ with $s_0 > d/2 = 2/2 = 1$, so we fix $s_0=2$.

\smallskip

2) Trivially $A\mH u(t,x) = \sum_k u_k(a(t))\,(-i\,\sign k)\,e^{ikx} 
= \mH Au(t,x)$.

\smallskip

3) Following \cite{Ioo-Plo-Tol}, Appendix I, 
it is convenient to use the representation of $\mH$ as a principal value integral, 
\begin{equation} \label{mH pv}
\mH u(t,x) = \frac{-1}{2\p}\, p.v.\int_\T \frac{u(t,x')}{\tan \frac12(x-x')}\,dx'
= \frac{-1}{2\p}\, \lim_{\e \to 0^+} 
\Big\{ \int_{x-\p}^{x-\e} + \int_{x+\e}^{x+\p} \Big\} 
\frac{u(t,x')}{\tan \frac12(x-x')}\,dx'.
\end{equation} 
Let $I+\tilde\b$ be the inverse of $I+\b$, namely 
$x+\b(t,x) = y$ if and only if $x=y+\tilde\b(t,y)$.
Changing variable $x'+\b(t,x') = y'$, $dx' = (1+\tilde\b_{y'}(t,y'))\,dy'$ in the integral,
\[
B\inv \mH Bu(t,y) = \frac{1}{\p}\, p.v. \int_{-\p}^\p u(t,y') \,
\pa_{y'} \Big\{ \log \sin \Big( \frac12 
\Big[ y + \tilde\b(t,y) - y' - \tilde\b(t,y') \Big] \Big)\, \Big\} \,dy',
\]
therefore 
\begin{equation} \label{kernel representation}
(B\inv \mH B - \mH)u(t,y) = \int_\T u(t,y')\, K(t,y,y')\,dy',
\end{equation}
where the kernel $K$ is 
\[
K(t,y,y') = \frac{1}{\p} \pa_{y'} \log \Big( 
\frac{\sin \frac12[y + \tilde\b(t,y) - y' - \tilde\b(t,y')]}
{\sin \frac12(y - y')} \Big).
\]
If $\b$ is sufficiently regular, then $K$ is bounded, and the integral in \eqref{kernel representation} is no longer a singular one. 
Denote $\mR = B\inv \mH B - \mH$. Then 
\begin{align*}
\pa_y^{m_1} \mR \pa_y^{m_2} u(t,y) 
= \int_\T (\pa_{y'}^{m_2} u)(t,y') \, \pa_y^{m_1} K(t,y,y')\, dy'
= \int_\T u(t,y') \, (-1)^{m_2} \pa_{y'}^{m_2} \pa_y^{m_1} K(t,y,y')\, dy',
\end{align*}
every space derivative goes on $K$ and does not affect $u$. 
Hence 
\[
\| \mR u\|_0^2 
= \int_{\T^2} \Big| \int_\T u(t,y') K(t,y,y')\,dy' \Big|^2 dy\,dt 
\leq C \int_{\T^3} |u(t,y')|^2 |K(t,y,y')|^2 \,dy'\,dy\,dt 
\leq C |K|_{0}^2 \, \|u\|_0^2,
\]
for $\|\pa_y^s (\pa_y^{m_1} \mR \pa_y^{m_2} u)\|_0$ 
replace $K$ with $\pa_y^{s+m_1} \pa_{y'}^{m_2} K$ 
and for $\| \pa_t^s (\pa_y^{m_1} \mR \pa_y^{m_2} u)\|_0$ calculate the usual derivatives of a product. Thus
\[
\| \pa_y^{m_1} \mR \pa_y^{m_2} u \|_s 
\leq C (\| u \|_s | K |_{m} + \| u \|_0 | K |_{s+m}).
\]
Now write $K = (1/\p) \pa_{y'} \log(1+f)$, where 
\[
f(t,y,y') = \frac{\sin \frac12[y + \tilde\b(t,y) - y' - \tilde\b(t,y')]
- \sin \frac12(y - y')}{\sin \frac12(y - y')} ,
\]
and decompose $f = a b c$, 
\[
a(y,y') = \frac{\frac12(y-y')}{\sin \frac12(y-y')}, \quad 
b(t,y,y') = \frac{\tilde\b(t,y) - \tilde\b(t,y')}{y - y'}\, 
= \int_0^1 \tilde\b_{y}(t,\lm y + (1-\lm)y')\,d\lm,
\]
\[
c(t,y,y') = \int_0^1 \cos \Big(\frac{y-y'+ \lm [\tilde\b(t,y) - \tilde\b(t,y')]}{2}\Big)\, d\lm.
\]
$a \in C^\infty$ for $|y'-y|\leq \p$ (by periodicity, take $\T = [y-\p,y+\p]$ when integrating in $dy'$).
$|b|_{s} \leq C |\tilde\b|_{s+1} \leq C |\b|_{s+1}$ by Lemma \ref{lemma:utile}$(i)$.  
All the derivatives of $c$ of order $\leq s$ are bounded if $\tilde\b \in 
W^{s,\infty}$, with tame estimate 
\[
| c |_{s} \leq C(s,|\tilde\b|_0)\,(1+|\tilde\b|_{s})
\leq C(s,|\b|_0)\,(1+|\b|_{s}).
\]
As a consequence $|f|_{0} \leq 1/2$ if $|\b|_1 \leq \d$ for some universal $\d \in (0,1)$, and 
$| K |_{s} \leq C(s) | \b |_{s+1}$.
\end{proof}

\begin{remark} Inequality 1) of Lemma \ref{lemma:commutators} can also be proved in a simple way using \eqref{mH pv}, see \cite{Ioo-Plo-Tol}, Appendix H.
\end{remark}

\section{\large{Appendix C. Proofs}}
\label{sec:appendix proofs}

\begin{proof}[\emph{\textbf{Proof of Proposition \ref{prop:a12345}}}]
Apply Lemma \ref{lemma:composition of functions, Moser}$(iv)$: let $f(x,y) = \pa_{y}^\a g_i(x,y)$, $|\a|=1$. By \eqref{g order 4}, $\pa_y^\b f(x,0) = 0$ for all $|\b| \leq 2$, and, by Taylor's formula \eqref{Taylor composition} for $N=2$ (with $\tilde f$ defined as in Lemma \ref{lemma:composition of functions, Moser}), 
\begin{equation} \label{Taylor per a1}
| \tilde f(U) |_s 
= \Big| \tilde f(U) - \sum_{n=0}^2 \frac{1}{n!}\, \tilde f^{(n)}(0)[U]^n \Big|_s 
\leq C(s) |U|_2^2 |U|_{s+2} 
\leq C(s) \| U \|_4^2 \|U\|_{s+4}.
 \end{equation}
Suppose that $a_1 = (\pa_y^\a g_i)(x,U,\mH U,\ldots) 
= \tilde f(U)$, where $U = \e \bar v + \e^2 u$. 
Then \eqref{Taylor per a1} gives 
\[
|a_1|_s 
\leq C(s) \| \e \bar v + \e^2 u \|_4^2 \|\e \bar v + \e^2 u \|_{s+4}
\leq \e^3 C(s) (\| \bar v \|_4 + \e K)^2 (\| \bar v \|_{s+4} + \e 
\| u \|_{s+4}) 
\leq \e^3 C(s,K) (1+\|u\|_{s+4})
\]
because $\|u\|_4 \leq K$ and $\| \bar v \|_{s+4}$ is a certain constant $C(s)$  depending on $s$.
Also $a_2, a_4, a_3 - 3 U^2$ and $a_5 - 3(U^2)_x$ are of the type $(\pa_y^\a g_i)(x,U,\mH U,\ldots)$, therefore they satisfy the same estimate as $a_1$. 
The additional part in $a_3$ and $a_5$ comes from the cubic term $\pa_x(U^3)$ of the nonlinearity $\mN(U)$. One has 
\[
|U^2 - \e^2 \bar v^2|_s = \e^3 |2 \bar v u + \e u^2|_s 
\leq \e^3 C(s,K) |u|_s 
\leq \e^3 C(s,K) \|u\|_{s+2}
\]
because $U = \e \bar v + \e^2 u$, and the estimate for $a_3 - \e^2 3 \bar v^2$ follows. Similarly for $a_5$. 

The derivatives $\pa_u a_1$ and $\pa_\e a_1$ are obtained differentiating the equality $a_1 = (\pa_y^\a g_i)(x,U, \mH U,\ldots)$, therefore they involve $\pa_y^\b g_i$ with $|\b|=2$. Then apply Taylor's formula \eqref{Taylor composition} with $N=1$ and evaluate at $U$, as above. 
\end{proof}

\begin{remark} 
In the estimate for $\pa_u a_i$ there is a factor $\e^2$ more than in the one for $\pa_\e a_i$ because $\pa_u U[h] = \e^2 h = O(\e^2)$, while 
$\pa_\e U = \bar v + 2 \e u = O(1)$. 
The point becomes very evident in the simplest case $g(x,U,\ldots) = U^4$.
\end{remark}

\begin{proof}[\emph{\textbf{Proof of Proposition \ref{prop:Psi mM mR3}}}]

By Proposition \ref{prop:a12345}, for $s=0$ and $\e < \e_0$, 
$|a_1|_0 \leq \e^3 C(K) \leq \e_0^3 C(K) \leq 1/2$ if $\e_0$ is small enough. 
$|\int a \,dx |_s \leq 2\p |a|_s$ for all $a(t,x)$. 
Applying \eqref{tame veloce infty} with $f(y) = (1+y)^p$, $p=-1/2, -2$ 
gives
\begin{equation} \label{stima rho}
|\rho - 1|_s \leq C(s,K) |a_1|_s \leq \e^3 C(s,K) (1+\|u\|_{s+4}), 
\quad 0 \leq s \leq r.
\end{equation}
Differentiating the formula for $\rho(u,\e)$, and using estimates on $a_1$,
\begin{equation} \label{stima der rho}
|\pa_u \rho(u,\e)[h]|_s 
\leq C(s,K) (|\pa_u a_1[h]|_s + |a_1|_s |\pa_u a_1[h]|_0) 
\leq \e^4 C(s,K) (\| h \|_{s+4} + \|u\|_{s+4} \|h\|_4), 
\end{equation}
and similarly $|\pa_\e \rho(u,\e)|_s \leq \e^2 C(s,K) (1+\| u \|_{s+4})$, 
for all $0 \leq s \leq r$.

$\mu_2 = \Pi_C(\rho)$, therefore, using \eqref{stima rho} with $s=0$,  
$|\mu_2-1| = |\Pi_C(\rho-1)| \leq |\rho-1 |_0 
\leq \e^3 C(0,K) \|u\|_4 = \e^3 C(K) \leq 1/2$.
Also, 
$|\pa_u \mu_2(u,\e)[h]| = |\Pi_C( \pa_u \rho(u,\e)[h] )| 
\leq |\pa_u \rho(u,\e)[h]|_0$, then use \eqref{stima der rho} with $s=0$.
Similarly for $\pa_\e \mu_2$.

$\a$ satisfies \eqref{propoli 4}, namely $\mu_2 (1+\a') = \rho$.
Thus $\a' = \mu_2\inv [(\rho - 1)+(1-\mu_2)]$, whence 
$|\a'|_s \leq 2(|\rho-1|_s + |\mu_2-1|)$. Moreover 
$|\a|_{s+1} \leq C |\a'|_s$ because $\a \in Y$, $\a(0)=0$, 
and $|\a(t)| $ $= |\a(t) - \a(0)|$ $\leq \p |\a'|_0$ 
for all $|t| \leq \p$ (Poincar\'e inequality for odd functions).
The derivatives of $\a$ are obtained differentiating 
the equality $\mu_2 (1+\a') = \rho$.
Similar argument for $\Pi_E \b$ using \eqref{tame veloce infty}, because $\Pi_E \b_x = \rho^{1/2}(1+a_1)^{-1/2} - 1$ by \eqref{propoli 5}. 
Thus $\a(u,\e)$ and $\Pi_E \b(u,\e)$ satisfy
\begin{align}
\label{stima alfa beta E}
| \a |_{s+1} + | \Pi_E \b |_s + |\Pi_E \b_x |_s 
& \leq \e^3 C(s,K) (1+\|u\|_{s+4}), 
\\
\label{stima alfa beta E der u}
| \pa_u \a[h] |_{s+1} + | \pa_u (\Pi_E \b)[h] |_s 
& \leq \e^4 C(s,K) (\|h\|_{s+4} + \|u\|_{s+4} \| h \|_4), 
\qquad 0 \leq s \leq r,
\\
\label{stima alfa beta E der epsilon}
| \pa_\e \a |_{s+1} + | \pa_\e \Pi_E \b |_s 
& \leq \e^2 C(s,K) (1+\|u\|_{s+4}).
\end{align}

$\s$ is defined in \eqref{propoli sigma}, namely 
$\s = \Pi_{T+C} \{ \om (\Pi_E \b)_t (1+\Pi_E \b_x) + a_3 (1+\Pi_E \b_x)^2\}$. 
Since $\Pi_E \b = O(\e^3)$, the only term of order $\e^2$ in $\s$ comes from $a_3$ and it is $\e^2 \Pi_{T+C}(3\bar v^2)$.
$\bar v$ is a finite sum of $q_j$ \eqref{qj}, therefore $\Pi_T (\bar v^2) = 0$. As a consequence, 
\[
\s - \e^2 \Pi_C(3 \bar v^2) 
= \Pi_{T+C} \{ \om (\Pi_E \b_t) (1+\Pi_E \b_x) 
+ a_3 (\Pi_E \b_x) (2 +\Pi_E \b_x) 
+ (a_3 - \e^2 3 \bar v^2) \}.
\]
Then, using the estimates for $\Pi_E \b$, $(a_3 - \e^2 3 \bar v^2)$ and their derivatives,
\begin{align} 
|\s - \e^2 \Pi_C(3 \bar v^2)|_{s-1} 
& \leq \e^3 C(s,K) (1+\|u\|_{s+4}), 
\label{stima sigma}
\\
|\pa_u \s(u,\e)[h]|_{s-1} 
& \leq \e^4 C(s,K) (\|h\|_{s+4} + \|u\|_{s+4} \| h \|_4),
\qquad 1 \leq s \leq r,
\label{stima der u sigma} 
\\
|\pa_\e \s(u,\e) - \e \Pi_C(6 \bar v^2)|_{s-1} 
& \leq \e^2 C(s,K) (1+\|u\|_{s+4})
\label{stima der epsilon sigma} 
\end{align}
($s-1$ because $|\Pi_E \b_t|_{s-1} \leq |\Pi_E \b|_s$). 

By \eqref{nu gamma}, $\mu_1 = \Pi_C (\s)$, and the estimates for $\mu_1$ follow from \eqref{stima sigma},\eqref{stima der u sigma},\eqref{stima der epsilon sigma} with $s=1$.

Since $\s - \mu_1 = \s - \Pi_C(\s) = \Pi_T(\s)$, 
by \eqref{propoli sigma} 
$\om \g' = \mu_1(1 + \a') - \s = \mu_1 \a' - \Pi_T(\s)$. 
By Poincar\'e inequality, $|\g|_s \leq C |\g'|_{s-1}$ because $\g \in Y$. 
The estimates for $\g = \Pi_T \b$ follow from those for $\s,\a,\mu_1$ and their derivatives, using the fact that $\om = 1 + 3\e^2$.
Hence \eqref{stima alfa beta E}, \eqref{stima alfa beta E der u}, \eqref{stima alfa beta E der epsilon} hold not only for $\Pi_E \b$, but also for $\g = \Pi_T \b$, and, as a consequence, for $\b$ too, 
for $1 \leq s \leq r$.

By Lemma \ref{lemma:utile}$(i)$, $|\tilde\a|_s + |\tilde\b|_s \leq C(s)(|\a|_s + |\b_s)$. Choose a smaller $\e_0$, if necessary, to have $\e_0^3 C(K) < 1/2$ in \eqref{stima diffeo psi basic}.
\eqref{stima Psi},\eqref{stima Psi infty} hold by Lemma \ref{lemma:utile}. 
Since 
\begin{equation} \label{equa diffeo}
\a(t) + \tilde\a(t + \a(t)) = 0, \quad 
\b(t,x) + \tilde\b \big( t + \a(t), x + \b(t,x) \big) = 0 \quad \forall (t,x) \in \T^2,
\end{equation}
the derivatives of $\tilde\a, \tilde\b$ with respect to the parameters $(u,\e)$ 
are obtained by differentiating \eqref{equa diffeo} with respect to $u$ or $\e$, whence 
\[
\pa_u \tilde\a [h] = - (1+\tilde \a_\t) \, \Psi\inv \{ \pa_u \a [h] \},
\quad 
\pa_u \tilde\b [h] = 
- (1 + \tilde\b_y) \, \Psi\inv \{ \pa_u \b [h] \}  
- \tilde\b_\t \, \Psi\inv \{ \pa_u \a[h] \},
\]
and similarly for $\pa_\e \tilde\a$, $\pa_\e \tilde\b$.
(Given a diffeomorphism depending on a parameter, this is nothing but the formula for the derivative of the inverse diffeomorphism with respect to the parameter.) 
Using \eqref{stima alfa beta E der u},\eqref{stima alfa beta E der epsilon} and \eqref{stima Psi infty}, for $s+1 \leq r$ we get
\[
|\pa_u \tilde\b[h]|_s 
\leq \e^4 C(s,K) (\| h \|_{s+4} + \| u \|_{s+5} \| h \|_5),
\quad 
|\pa_\e \tilde\b|_s 
\leq \e^2 C(s,K) (1 + \| u \|_{s+5}),
\]
and the same for $\tilde\a$. These inequalities also hold for $\a,\b$ (actually, $\a,\b$ satisfy \eqref{stima alfa beta E der u},\eqref{stima alfa beta E der epsilon}, which are stronger).

To prove \eqref{stima Psi-I}, consider the one-parameter family of changes of variables 
\[
(\Psi_\lm f)(t,x) = f(\psi_\lm(t,x)), \quad  
\psi_\lm(t,x) = \big( t+\lm\a(t), x+\lm\b(t,x) \big), \quad 
0 \leq \lm \leq 1.
\]
One has 
\[ 
(\Psi-I)f(t,x) 
= f(\psi_1(t,x)) - f(\psi_0(t,x))
= \int_0^1 (\gr f)(\psi_\lm(t,x)) 
\cdot \big( \a(t), \b(t,x) \big) \, d\lm.
\]
Use Lemma \ref{lemma:utile} 
to estimate $\| \Psi_\lm f_t \|_s$ and $\| \Psi_\lm f_x \|_s$, 
then use \eqref{mixed norms tame product}.
The same holds for $\Psi\inv$. 
The estimate for $\tilde\Psi, \tilde\Psi\inv$ hold because $\| \pp h \|_s \leq \| h \|_s$ for all $s$.
Repeat the same argument with norms $| \ |_s$ to prove \eqref{stima Psi-I infty}. 
By the chain rule, the derivative of $\Psi f$ with respect to $u$ in the direction $h$ is 
\[
\pa_u (\Psi f)[h] = \pa_u \{ f(t+\a(t),x+\b(t,x)) \}[h] 
= (\Psi f_t) \pa_u \a[h] + (\Psi f_x) \pa_u \b[h],
\]
therefore \eqref{stima Psi der u} follows using the interpolation \eqref{mixed norms tame product} for products. 
Similarly for \eqref{stima Psi der epsilon}.

Since 
\[
[1 + (\Psi\inv \a')(\t)] (1+\tilde\a'(\t)) = 1,
\]
$(\mM-I)$ is the multiplication by the factor $(\Psi\inv \a')
= - \tilde\a' / (1+\tilde\a') =:p$. 
Hence $(\tilde\mM - I)f = \pp (\mM-I) f = \pp (pf)$ for all $f \in Z_0$, because $\pp = I$ on $Z_0$. 
By Lemma \ref{lemma:composition of functions, Moser}, 
$p$ satisfies the same estimate as $\tilde\a'$, and  $|\tilde\a'|_s \leq C(s) |\a'|_s \leq C(s) |\a|_{s+1}$, then use \eqref{stima alfa beta E} and apply \eqref{asymmetric tame product} to get 
\[
\| pf \|_s \leq \e^3 C(K) \| f \|_s + \e^3 C(s,K) (1+\|u\|_{s+4}) \| f \|_2, \quad 2 \leq s \leq r.
\]
For the derivatives $\pa_u \mM [h]$, $\pa_\e \mM$ use \eqref{stima alfa beta der u},\eqref{stima alfa beta der epsilon}.
Apply Lemma \ref{lemma:tame estimates for operators} to obtain the estimates for $(\tilde\mM\inv-I)$ and its derivatives.

The estimates for $a_i$, $i=6,\ldots,9$ follow from formulae \eqref{mL3} and the estimates for $\Psi\inv$. 
In $a_7$ put the term $\e^2 3 \bar v^2$ in evidence, namely
write  
\[
\frac{\om \b_t + a_3(1+\b_x)}{1+\a'}\, = 
b + q, \quad 
b := \e^2 3 \bar v^2, \quad 
q := \frac{\om \b_t + (a_3-b)(1+\b_x) + b(\b_x - \a')}{1+\a'}\,,
\]
estimate $\Psi\inv(q)$ using \eqref{stima Psi infty}, the inequalities for $\a,\b,(a_3-b)$, and $|b|_s = C(s)$.  
For $\Psi\inv(b) = b + (\Psi\inv - I)b$, use \eqref{stima Psi-I}.
Similarly for $a_9$.
Similar calculations for the derivatives $\pa_u a_i[h]$, $\pa_\e a_i$.

To prove \eqref{stima mRmH}, write $\Psi$ as the composition of the two changes of variables $A$, $B$, 
\begin{equation*} 
\Psi = A B, \quad Ah(t,x) = h(t+\a(t),x), \quad 
Bh(t,x) = h(t,x+\b_1(t,x)),
\end{equation*}
where $\b_1 := A\inv(\b)$, namely $\b_1(t+\a(t),x) = \b(t,x)$. 
By Lemma \ref{lemma:commutators}$(ii)$, $\Psi\inv \mH \Psi = B\inv A\inv \mH A B = B\inv \mH B$. 
By the inequality \eqref{stima Psi infty} for the change of variable $A$, 
$|\b_1|_s \leq \e^3 C(s,K)(1+\|u\|_{s+4})$. Then apply Lemma \ref{lemma:commutators}$(iii)$.

In $\mR_1$ (see \eqref{mR1}) the coefficients of $\pa_y^k \mR_\mH$, $k=0,1,2$, are functions $f_k$ that satisfy $| f_k |_{s} \leq C(s,K) (1+ \|u\|_{s+5})$ for $s+1 \leq r$ (two of them are $a_6, a_8$ without the denominator $(1+\a')$, the other one is \eqref{propoli 1}).  
By \eqref{mixed norms tame product},\eqref{interpolation estremi}, and \eqref{stima mRmH},
\[
\| f_k \pa_y^k \mR_\mH  \pa_y^m h \|_s \leq \e^3 C(s,m,K) 
\big( \|h\|_s (1 + \|u\|_{m+7}) + \| h \|_0 \|u\|_{s+m+7}\big), 
\quad k=0,1,2, 
\]
for $m \geq 0$, $s+m+3 \leq r$. 
For the last term in $\mR_1$ use \eqref{[PiC,Psi]}, the estimate for $\Psi\inv a_5$, 
integration by parts $|\Pi_C(f \pa_y^m h)| = |\Pi_C[(\pa_y^mf)h]|$, 
the inequality $|\Pi_C(fh)| \leq C |f|_0 \|h\|_0$, Lemma \ref{lemma:utile}$(i)$ to pass from $\tilde\a,\tilde\b$ to $\a,\b$,
and \eqref{interpolation estremi}:
\begin{equation} \label{restino bono bono}
\| \pp (\Psi\inv a_5) [\Pi_C,\Psi] \pa_y^m h \|_s 
= \| \Psi\inv a_5 \|_s |[\Pi_C,\Psi] \pa_y^m h| 
\leq \e^5 C(s,m) (1+\|u\|_{s+m+4}) \|h\|_0.
\end{equation}
The estimate for $\mR_1$ follows.
$\mR_2$ satisfies the same estimate as $\mR_1$ because 
$\mR_2 = \mM\inv \mR_1$. 
For $\mR_3$, note that $\Pi_C \mL_2 = \Pi_C (a_9+\mR_2)$. 
Use \eqref{stima mM-I} for $\mM\inv$, then the same arguments as for  \eqref{restino bono bono}.
\end{proof}

\noindent
\textbf{Formula for $\mR_4$}. 
\begin{align*}
\mR_4 
& = \mR_3 \pp \Phi - a_9 \Pi_C \Phi 
\\
& \quad + \sum_{k=0}^3 \Big\{ \Pi_E^\perp \mu_2
	\Big(\b\k_{yy} \pa_y^{-k} 
	+ 2 \b\k_{y} \pa_y^{-k+1}  
	+ \b\k \pa_y^{-k+2} \Big) 
+ a_6 \pepe 
	\Big(\b\k_{y} \pa_y^{-k} 
	+ \b\k \pa_y^{-k+1} \Big)
\\
& \quad 
+ a_8 \pepe \b\k \pa_y^{-k} 
- 
\big(\mu_2 \b\k \pa_y^{-k+2} 
+ \mu_0 \b\k \pa_y^{-k} 
+ \mu_{-2} \b\k \pa_y^{-k-2}\big) \pepe \Big\}
\\
& \quad 
+ \Big( - \mH ( 2\mu_2 \a\uno_y + a_6 \a\uno) - (a_7 - \mu_1) \a\uno \Big) \pepe
+ 
\Big( ( 2\mu_2 \b\uno_y + a_6 \b\uno) - \mH (a_7 - \mu_1) \b\uno \Big) \pepe
\\
& \quad
+ \sum_{k=0}^3 \Big\{ 
[a_6,\mH] \Big( \a\k_y \pa_y^{-k} + \a\k \pa_y^{-k+1} \Big)
+ [a_7,\mH] \Big( \b\k_y \pa_y^{-k} + \b\k \pa_y^{-k+1} \Big)
+ [a_8,\mH] \a\k \pa_y^{-k}
\\ 
& \quad 
+ [a_9,\mH] \b\k \pa_y^{-k}  \Big\} 
+ \sum_{k=0}^3 [\b\k - \a\k, \mH] 
\big( \mu_2 \pa_y^{-k+2} + \mu_0 \pa_y^{-k} + \mu_{-2} \pa_y^{-k-2} \big)
\\
& \quad 
+ \Big( \om \a\tre_\t 
- \mu_2 \b\tre_{yy} 
- a_6 \b\tre_y
+ a_7 \a\tre_y 
- (a_8-\mu_0) \b\tre
+ a_9 \a\tre
+ \mu_{-2} \sum_{k=1}^3 \b\k \pa_y^{-k-2} \Big)  \, \pa_y^{-3}
\\
& \quad 
+ \mH \Big( \om \b\tre_\t 
+ \mu_2 \a\tre_{yy} 
+ a_6 \a\tre_y
+ a_7 \b\tre_y 
+ (a_8-\mu_0) \a\tre
+ a_9 \b\tre
- \mu_{-2} \sum_{k=1}^3 \a\k \pa_y^{-k-2} \Big) \, \pa_y^{-3}.
\end{align*}

\begin{proof}[\emph{\textbf{Proof of Proposition \ref{prop:Phi mR}}}]
From the estimates for $\mu_2,\mu_1,a_6,a_7,a_8,a_9$ of Proposition \ref{prop:Psi mM mR3} 
and formulae \eqref{formula Re ph},\eqref{formula Im ph} 
for $\ph$ it follows that 
\begin{align}
\label{stima ph}
\| \Re(\ph) \|_s + \| \Im(\ph) \|_s 
& \leq \e^2 C(s,K) (1+\| u \|_{s+c}), 
\\
\label{stima ph der u}
\| \pa_u \Re(\ph)[h] \|_s + \| \pa_u \Im(\ph)[h] \|_s 
& \leq \e^4 C(s,K) (\| h \|_{s+c} + \| u \|_{s+c} \| h \|_4), 
\\
\label{stima ph der epsilon}
\| \pa_\e \Re(\ph) \|_s + \| \pa_\e \Im(\ph) \|_s 
& \leq \e C(s,K) (1 + \| u \|_{s+c}),
\end{align} 
for $2 \leq s \leq r-1$, where $c = 6$
(in this proof we use \eqref{asymmetric tame product} to estimate any product).
As a consequence, by Lemma \ref{lemma:composition of functions, Moser} and \eqref{formula alfa beta zero}, 
$\a\zero - 1$ and $\b\zero$ and their derivatives satisfy the same estimates \eqref{stima ph},
\eqref{stima ph der u},
\eqref{stima ph der epsilon}, with $c=6$.

$g\zero$ is given by \eqref{S f zero}, therefore its real and imaginary part satisfy \eqref{stima ph}, \eqref{stima ph der u}, \eqref{stima ph der epsilon}, with $c=8$, for $2 \leq s \leq r-3$. The same for $\eta\uno$ because of \eqref{eta uno E}, \eqref{eta uno T}.
By formulae \eqref{g uno}, \eqref{eta due E}, \eqref{eta due T},
\eqref{g due}, \eqref{eta tre E}, the same holds 
for $g\uno, \eta\due$, with $c=10$, $2 \leq s \leq r-5$, 
and for $g\due, \eta\tre$, with $c=12$, $2 \leq s \leq r-7$. 
Since $f\k = \eta\k f\zero$, $k=1,2,3$, all coefficients 
$\a\k, \b\k$, $k=1,2,3$ and their derivatives satisfy \eqref{stima ph},
\eqref{stima ph der u}, \eqref{stima ph der epsilon}, with $c=12$, for all $2 \leq s \leq r-7$. By \eqref{asymmetric tame product},
\[
\| (\Phi-I)f \|_s \leq C  \| \,\text{coeff}\, \|_2 \| f \|_s 
+ C(s) \| \,\text{coeff}\, \|_s \| f \|_2,
\]
where `coeff' are $(\a\zero-1), \b\zero, \a\k, \b\k$, $k=1,2,3$, and $C$ does not depend on $s$. 
Therefore 
\begin{equation*} 
\| (\Phi-I)f \|_s \leq \e^2 C(K) \| f \|_s 
+ \e^2 C(s,K) (1+\| u \|_{s+12}) \| f \|_2.
\end{equation*}
The estimates for $\pa_u \Phi[h]$ and $\pa_\e \Phi$ are obtained in the same way, using the estimates for the derivatives of the coefficients.
Similarly, \eqref{Phi-I for bif 1},\eqref{Phi-I for bif 2} follow because
$\pa_\t (\Phi-I) f = (\Phi-I)\pa_\t f + \Phi_\t f$, where $\Phi_\t$ is the operator of the same type as $\Phi$ that has coefficients $\a\k_\t,\b\k_\t$ instead of $\a\k, \b\k$, $k=0,\ldots,3$.  
Since $\| \pp f \|_s \leq \| f \|_s$, all the estimate for $\Phi-I$ also hold for $\tilde\Phi - \pp = \pp (\Phi - I)\pp$.
\eqref{stima Phi-I},\eqref{stima Phi der u} and \eqref{stima Phi der epsilon} also hold for $\tilde\Phi\inv$ by Lemma \ref{lemma:tame estimates for operators}.

To prove \eqref{S property} for $\tilde\Phi\inv \tilde\mM\inv \tilde\Psi\inv$, write
\begin{equation*} 
\tilde\Phi\inv \tilde\mM\inv \tilde\Psi\inv 
= I + S, \quad 
S:= (\tilde\Psi\inv - I) 
+ (\tilde\mM\inv - I) \tilde\Psi\inv 
+ (\tilde\Phi\inv - I) \tilde\mM\inv \tilde\Psi\inv,
\end{equation*}
then apply \eqref{stima Psi-I}, \eqref{stima Psi}, \eqref{stima mM-I} and \eqref{stima Phi-I}. Similarly for the other operators.

The estimates for $\mu_0, \mu_{-2}$ and their derivatives follow from formulae \eqref{nu nu' 0}, \eqref{nu nu' -2} and the estimates for $\mu_2,a_6,a_7,a_8,a_9,\eta\due, g\zero$.

Now study the rest $\mR$. 
By \eqref{stima mR123}, for $2 \leq s \leq r-6$,
\begin{equation} \label{mR3 der y3}
\| \mR_3 \pa_y^m f \|_s 
\leq \e^3 C(s,K) ( \|f\|_s + \| f \|_0 \|u\|_{s+10}), 
\quad  0 \leq m \leq 3.
\end{equation}
By definition, $\Phi$ is a combination of multiplications and $\mH, \pa_y\inv$. Every $\pa_y$ can be moved from the right to the left of any multiplication operator with elementary calculus: $[a,\pa_y] = - a_y$, namely, for every $a,f$, 
\[
a \pa_y f = \pa_y (af) - a_y f, \quad 
a \pa_y^2 f = \pa_y^2 (af) -2 \pa_y( a_y f) + a_{yy} f, \quad
a \pa_y^3 f = \pa_y^3 (af) -3 \pa_y^2( a_y h) + 3 \pa_y (a_{yy} f) 
- a_{yyy} f.
\]
Recall that the coefficients $\a\k, \b\k$ satisfy \eqref{stima ph},
\eqref{stima ph der u}, \eqref{stima ph der epsilon}, with $c=12$, $2 \leq s \leq r-7$. 
Moving $\pa_y^m$ to the left of $\Phi$, $m=0,1,2,3$, 
the coefficients $\a\k, \b\k$ are subject to up to 3 derivatives in $y$. So applying \eqref{mR3 der y3} gives
\[
\| \mR_3 \pp \Phi \pa_y^m f \|_s 
\leq \e^5 C(s,K) (\| f \|_s + \| u \|_{s+10} \| f \|_2), \quad 
0 \leq m \leq 3, \quad 
2 \leq s \leq r-10.
\]
Each term $R_{(a)}$ of type $(a)$ containing $[b,\mH]$ can be estimated by Lemma \ref{lemma:commutators}$(i)$, whence
\[
\| R_{(a)} \pa_y^m f \|_s \leq \e^2 C(s,K) (\|f\|_s + \| u \|_{s+17} \| f \|_2), \quad 0 \leq m \leq 3, \quad 2 \leq s \leq r-12,
\]
and the same inequality also holds for each term $R_{(b)}$ of type $(b)$ that contains $\pepe$. Thus it holds for $\| \mR_4 \pa_y^m f \|_s$.
Since $\mR := \tilde\Phi\inv \pp \mR_4$ by \eqref{DR}, 
the estimate for $\mR \pa_y^m$ follows from \eqref{stima Phi-I}.
\end{proof}

\begin{proof}[\emph{\textbf{Proof of \eqref{rn}}}]

(The meaning of $A,B,a,b,c$ in the following proof is independent on the rest of the paper).
By \eqref{conj completa}, 
\begin{align}
F(u_n) + F'(u_n) h_{n+1} 
& = F(u_n) + P_\e\inv \tilde\Psi_n \tilde\mM_n \tilde\Phi_n \tilde\mL_4(u_n) \tilde\Phi_n\inv \tilde\Psi_n\inv h_{n+1} 
\notag 
\\ & 
= P_\e\inv \tilde\Psi_n \tilde\mM_n \tilde\Phi_n 
\big\{ 
\tilde\Phi_n\inv \tilde\mM_n\inv \tilde\Psi_n\inv P_\e F(u_n) 
+ \tilde\mL_4(u_n) \tilde\Phi_n\inv \tilde\Psi_n\inv h_{n+1} \big\}.
\label{parentesi per rn}
\end{align}
Let $p = \{ \ldots \}$ be the quantity in parentheses in \eqref{parentesi per rn}.
Let 
\[
c := \tilde\Phi_n\inv \tilde\mM_n\inv \tilde\Psi_n\inv P_\e F(u_n) 
= \Pi_{n+1} c + \Pi_{n+1}^\perp\,c,
\]
\[
\tilde\mL_4(u_n) = A + B, \quad 
A := \Pi_{n+1} \tilde\mL_4(u_n) \Pi_{n+1}, \quad  
B := \Pi_{n+1}^\perp\, \tilde\mL_4(u_n) \Pi_{n+1} 
+ \tilde\mL_4(u_n) \Pi_{n+1}^\perp.
\]
With these abbreviations, by the definition \eqref{sequence un}
$h_{n+1} = - \Pi_{n+1} \tilde\Psi_n \tilde\Phi_n A\inv \Pi_{n+1} c$, whence
\begin{equation*} 
\tilde\Phi_n\inv \tilde\Psi_n\inv h_{n+1} = a + b,
\quad 
a := -A\inv \Pi_{n+1} c, 
\quad 
b:= \tilde\Phi_n\inv \tilde\Psi_n\inv \Pi_{n+1}^\perp \, \tilde\Psi_n \tilde\Phi_n A\inv \Pi_{n+1} c.
\end{equation*}
Now $p = c + (A+B)(a+b)$, and $A a + \Pi_{n+1} c = 0$. Therefore 
\[
p = \Pi_{n+1}^\perp\,c + Ba + (A+B)b.
\]
$\Pi_{n+1}^\perp\, \tilde\mL_4(u_n) \Pi_{n+1} 
= \Pi_{n+1}^\perp\, \tilde\mR \Pi_{n+1}$ 
because $\tilde\mL_4(u_n) = \tilde\mD + \tilde\mR$ and $\tilde\mD$ is diagonal. 
Moreover $\tilde\mL_4(u_n) \Pi_{n+1}^\perp a = 0$ because $a \in Z_n$. 
Thus \eqref{rn} follows.
\end{proof}

\begin{proof}[\emph{\textbf{Proof of Lemma \ref{lemma:tame per F}}}]

$(i)$ Lemma \eqref{lemma:tame per F} simply follows from Lemma \ref{lemma:composition of functions, Moser}. 
In particular, $\bar v_2(\e)$ satisfies \eqref{linearized unperturbed bif eq}. By Proposition \ref{prop:bif}, 
$(\Pi_V A \Pi_V)$ : $V \cap X \to V \cap Y$, 
$h \mapsto 3 \pa_t h + \Pi_V \pa_x(3 \bar v_1^2 h)$ 
is invertible, with 
\begin{equation} \label{stima VAVinv}
\| (\Pi_V A \Pi_V)\inv h \|_s \leq C \| h \|_{s-1} \quad \forall h \in V \cap Y, \quad s \geq 1,
\end{equation}
where $C$ depends only on the set $\mK$, like in \eqref{beef 1}. 
By \eqref{g order 4} and \eqref{Taylor composition}, 
$\| \mN_4 (h) \|_s \leq C(s) \| h \|_4^3 \| h \|_{s+2}$ for $0 \leq s \leq r$. Hence 
\begin{equation} \label{v2<v1}
\| \bar v_2(\e) \|_s \leq C \e^{-4} \| \mN_4(\e \bar v_1) \|_{s-1} 
\leq C(s) \| \bar v_1 \|_4^3 \| \bar v_1 \|_{s+1} 
= C'(s)
\end{equation}
where $C'(s)$ depends on $s$ and $\| \bar v_1 \|_{s+1}$.
\eqref{v2<v1} for $s=4$ implies that 
$\e \| \bar v_1 \|_4 + \e^2 \| \bar v_2 \|_4 < \d_0$  
for all $\e < \e_0$, for some $\e_0$ depending on $\| \bar v_1 \|_5$.

To complete the proof of \eqref{stima v2}, differentiate \eqref{linearized unperturbed bif eq} with respect to $\e$, then use \eqref{stima VAVinv}, 
\[
\| \pa_\e \bar v_2(\e) \|_s 
\leq C (4 \e^{-5} \| \Pi_V \mN_4(\e \bar v_1) \|_{s-1} 
+ \e^{-4} \| \Pi_V \mN_4'(\e \bar v_1)[\bar v_1] \|_{s-1} )
\leq \e^{-1} C(s).
\]
\eqref{F(u0)} follows from formula \eqref{mF bar v1 bar v2} and estimates \eqref{stima v2}.
To prove $(ii)$, observe that 
\[
Q(u,h,\e) = \e^{-2} P_\e\inv \Big( 
\pa_x \{ 3 (\e \bar v_1 + \e^2 u)(\e^2 h)^2 + (\e^2 h)^3 \}
+ \mN_4(\e \bar v_1 + \e^2 u + \e^2 h) 
- \mN_4(\e \bar v_1 + \e^2 u) 
- \mN_4'(\e \bar v_1 + \e^2 u)[\e^2 h] \Big),
\]
then apply \eqref{Taylor composition} to $\mN_4$.

$(iii)$ follows from (\ref{formula F}) by the usual tame estimates.
\end{proof}

\vspace{15mm}

Pietro Baldi

Dipartimento di Matematica e Applicazioni ``R. Caccioppoli''

Universit\`a di Napoli Federico II

Via Cintia, Monte S. Angelo 

80126 Napoli 

Italy

\medskip

E-mail: pietro.baldi@unina.it

\end{document}